\documentclass{amsart}

\usepackage{enumerate, amsmath, amsfonts, amssymb, amsthm, wasysym, graphics, graphicx, xcolor, url, hyperref, hypcap, a4wide, pdflscape, multido, xargs, colortbl, multicol, multirow, calc, shuffle, ifthen, hvfloat, accents}
\hypersetup{colorlinks=true, citecolor=blue, linkcolor=blue}
\usepackage{tikz}\usetikzlibrary{trees,snakes,shapes,arrows,matrix,calc}
\graphicspath{{figures/}{figures/exmTwist/}}
\makeatletter
\def\input@path{{figures/}}
\makeatother

\title{Brick polytopes, lattice quotients, and Hopf algebras}

\author{Vincent Pilaud} 
\address{CNRS \& LIX, \'Ecole Polytechnique, Palaiseau}
\email{vincent.pilaud@lix.polytechnique.fr}
\urladdr{http://www.lix.polytechnique.fr/~pilaud/}
\thanks{Partially supported by the Spanish MICINN grant MTM2011-22792 and by the French ANR grants EGOS~(12\,JS02\,002\,01) and SC3A~(15\,CE40\,0004\,01).}

\newtheorem{theorem}{Theorem}
\newtheorem{corollary}[theorem]{Corollary}
\newtheorem{proposition}[theorem]{Proposition}
\newtheorem{lemma}[theorem]{Lemma}
\newtheorem{definition}[theorem]{Definition}

\theoremstyle{definition}
\newtheorem{example}[theorem]{Example}
\newtheorem{remark}[theorem]{Remark}
\newtheorem{question}[theorem]{Question}

\newcommand{\R}{\mathbb{R}} 
\newcommand{\N}{\mathbb{N}} 
\newcommand{\Z}{\mathbb{Z}} 
\newcommand{\K}{\mathbb{K}} 
\newcommand{\fS}{\mathfrak{S}} 
\newcommand{\fP}{\mathfrak{P}} 

\newcommand{\set}[2]{\left\{ #1 \;\middle|\; #2 \right\}} 
\newcommand{\bigset}[2]{\big\{ #1 \;|\; #2 \big\}} 
\newcommand{\ssm}{\smallsetminus} 
\newcommand{\one}{{1\!\!1}} 
\newcommand{\eqdef}{\mbox{\,\raisebox{0.2ex}{\scriptsize\ensuremath{\mathrm:}}\ensuremath{=}\,}} 
\newcommand{\polar}{^\lozenge} 
\newcommand{\blackPolar}{^\blacklozenge} 

\DeclareMathOperator{\conv}{conv} 
\DeclareMathOperator{\cone}{cone} 
\DeclareMathOperator{\coinv}{coinv} 

\newcommand{\boxsize}{.34}
\newlength{\verticalOffset}
\newlength{\verticalShift}
\newcounter{length}
\newcommand{\length}[1]{%
	\setcounter{length}{0}%
	\foreach \x in {#1} {%
		\stepcounter{length}%
	}%
}

\newcommand{\pipeDreamMonoColor}[3]{
	\length{#3}%
	\begin{tikzpicture}[baseline = \value{length}*\verticalShift+\verticalOffset, scale=1]
		\coordinate (origin) at (0,0);
		\newcount{\y} \y=0
		\newcount{\x}
		\foreach \line in {#3} {
			\x=0
			\foreach \t in \line {
				\coordinate (W) at ($ (origin) + ( \boxsize * \x , -\boxsize * \y ) + ( 0      , \boxsize / 2 ) $);
				\coordinate (E) at ($ (origin) + ( \boxsize * \x , -\boxsize * \y ) + ( \boxsize     , \boxsize / 2 ) $);
				\coordinate (N) at ($ (origin) + ( \boxsize * \x , -\boxsize * \y ) + ( \boxsize / 2 , \boxsize     ) $);
				\coordinate (S) at ($ (origin) + ( \boxsize * \x , -\boxsize * \y ) + ( \boxsize / 2 , 0 ) $);
				\coordinate (C) at ($ (origin) + ( \boxsize * \x , -\boxsize * \y ) + ( \boxsize / 2 , \boxsize / 2 ) $);
				\ifthenelse{\equal{\t}{e}}{
					\draw[rounded corners=\boxsize * 8, color=#1, thick] (W) -- (C) -- (N);
					\draw[rounded corners=\boxsize * 8, color=#1, thick] (S) -- (C) -- (E);			
				}{}
				\ifthenelse{\equal{\t}{c}}{
					\draw[color=#1, thick] (W) -- (E);
					\draw[color=#1, thick] (S) -- (N);
				}{}
				\ifthenelse{\equal{\t}{t}}{
					\draw[rounded corners=\boxsize * 8, color=#2] (W) -- (C) -- (N);
					\draw[rounded corners=\boxsize * 8, color=#1, thick] (S) -- (C) -- (E);			
				}{}
				\ifthenelse{\equal{\t}{b}}{
					\draw[rounded corners=\boxsize * 8, color=#1, thick] (W) -- (C) -- (N);
					\draw[rounded corners=\boxsize * 8, color=#2] (S) -- (C) -- (E);			
				}{}
				\ifthenelse{\equal{\t}{tb}}{
					\draw[rounded corners=\boxsize * 8, color=#2] (W) -- (C) -- (N);
					\draw[rounded corners=\boxsize * 8, color=#2] (S) -- (C) -- (E);			
				}{}
				\global\advance\x by 1
			}
			\global\advance\y by 1
		}
	\end{tikzpicture}%
}

\newcommand{\pipeDreamBiColor}[4]{
	\length{#4}%
	\begin{tikzpicture}[baseline = \value{length}*\verticalShift+\verticalOffset, scale=1]
		\coordinate (origin) at (0,0);
		\newcount{\y} \y=0
		\newcount{\x}
		\foreach \line in {#4} {
			\x=0
			\foreach \t/\colorW/\colorS in \line {
				\coordinate (W) at ($ (origin) + ( \boxsize * \x , -\boxsize * \y ) + ( 0      , \boxsize / 2 ) $);
				\coordinate (E) at ($ (origin) + ( \boxsize * \x , -\boxsize * \y ) + ( \boxsize     , \boxsize / 2 ) $);
				\coordinate (N) at ($ (origin) + ( \boxsize * \x , -\boxsize * \y ) + ( \boxsize / 2 , \boxsize     ) $);
				\coordinate (S) at ($ (origin) + ( \boxsize * \x , -\boxsize * \y ) + ( \boxsize / 2 , 0 ) $);
				\coordinate (C) at ($ (origin) + ( \boxsize * \x , -\boxsize * \y ) + ( \boxsize / 2 , \boxsize / 2 ) $);
				\ifthenelse{\equal{\t}{e}}{
					\ifthenelse{\equal{\colorW}{l}}{\draw[rounded corners=\boxsize * 8, color=#1, thick] (W) -- (C) -- (N);}{}
					\ifthenelse{\equal{\colorW}{r}}{\draw[rounded corners=\boxsize * 8, color=#2, thick] (W) -- (C) -- (N);}{}
					\ifthenelse{\equal{\colorW}{b}}{\draw[rounded corners=\boxsize * 8, color=#3] (W) -- (C) -- (N);}{}
					\ifthenelse{\equal{\colorS}{l}}{\draw[rounded corners=\boxsize * 8, color=#1, thick] (S) -- (C) -- (E);}{}
					\ifthenelse{\equal{\colorS}{r}}{\draw[rounded corners=\boxsize * 8, color=#2, thick] (S) -- (C) -- (E);}{}
					\ifthenelse{\equal{\colorS}{b}}{\draw[rounded corners=\boxsize * 8, color=#3] (S) -- (C) -- (E);}{}
				}{}
				\ifthenelse{\equal{\t}{c}}{
					\ifthenelse{\equal{\colorW}{l}}{\draw[color=#1, thick] (W) -- (E);}{}
					\ifthenelse{\equal{\colorW}{r}}{\draw[color=#2, thick] (W) -- (E);}{}
					\ifthenelse{\equal{\colorS}{l}}{\draw[color=#1, thick] (S) -- (N);}{}
					\ifthenelse{\equal{\colorS}{r}}{\draw[color=#2, thick] (S) -- (N);}{}
				}{}
				\global\advance\x by 1
			}
			\global\advance\y by 1
		}
	\end{tikzpicture}%
}

\newcommand{\pipeDreamTriColor}[5]{
	\length{#5}%
	\begin{tikzpicture}[baseline = \value{length}*\verticalShift+\verticalOffset, scale=1]
		\coordinate (origin) at (0,0);
		\newcount{\y} \y=0
		\newcount{\x}
		\foreach \line in {#5} {
			\x=0
			\foreach \t/\colorW/\colorS in \line {
				\coordinate (W) at ($ (origin) + ( \boxsize * \x , -\boxsize * \y ) + ( 0      , \boxsize / 2 ) $);
				\coordinate (E) at ($ (origin) + ( \boxsize * \x , -\boxsize * \y ) + ( \boxsize     , \boxsize / 2 ) $);
				\coordinate (N) at ($ (origin) + ( \boxsize * \x , -\boxsize * \y ) + ( \boxsize / 2 , \boxsize     ) $);
				\coordinate (S) at ($ (origin) + ( \boxsize * \x , -\boxsize * \y ) + ( \boxsize / 2 , 0 ) $);
				\coordinate (C) at ($ (origin) + ( \boxsize * \x , -\boxsize * \y ) + ( \boxsize / 2 , \boxsize / 2 ) $);
				\ifthenelse{\equal{\t}{e}}{
					\ifthenelse{\equal{\colorW}{l}}{\draw[rounded corners=\boxsize * 8, color=#1, thick] (W) -- (C) -- (N);}{}
					\ifthenelse{\equal{\colorW}{m}}{\draw[rounded corners=\boxsize * 8, color=#2, thick] (W) -- (C) -- (N);}{}
					\ifthenelse{\equal{\colorW}{r}}{\draw[rounded corners=\boxsize * 8, color=#3, thick] (W) -- (C) -- (N);}{}
					\ifthenelse{\equal{\colorW}{b}}{\draw[rounded corners=\boxsize * 8, color=#4] (W) -- (C) -- (N);}{}
					\ifthenelse{\equal{\colorS}{l}}{\draw[rounded corners=\boxsize * 8, color=#1, thick] (S) -- (C) -- (E);}{}
					\ifthenelse{\equal{\colorS}{m}}{\draw[rounded corners=\boxsize * 8, color=#2, thick] (S) -- (C) -- (E);}{}
					\ifthenelse{\equal{\colorS}{r}}{\draw[rounded corners=\boxsize * 8, color=#3, thick] (S) -- (C) -- (E);}{}
					\ifthenelse{\equal{\colorS}{b}}{\draw[rounded corners=\boxsize * 8, color=#4] (S) -- (C) -- (E);}{}
				}{}
				\ifthenelse{\equal{\t}{c}}{
					\ifthenelse{\equal{\colorW}{l}}{\draw[color=#1, thick] (W) -- (E);}{}
					\ifthenelse{\equal{\colorW}{m}}{\draw[color=#2, thick] (W) -- (E);}{}
					\ifthenelse{\equal{\colorW}{r}}{\draw[color=#3, thick] (W) -- (E);}{}
					\ifthenelse{\equal{\colorS}{l}}{\draw[color=#1, thick] (S) -- (N);}{}
					\ifthenelse{\equal{\colorS}{m}}{\draw[color=#2, thick] (S) -- (N);}{}
					\ifthenelse{\equal{\colorS}{r}}{\draw[color=#3, thick] (S) -- (N);}{}
				}{}
				\global\advance\x by 1
			}
			\global\advance\y by 1
		}
	\end{tikzpicture}%
}

\newcommand{\cross}[1][black]{\raisebox{.1cm}{\pipeDreamMonoColor{#1}{black}{c}}}
\newcommand{\crossBiColor}[2]{\raisebox{.1cm}{\pipeDreamBiColor{#1}{#2}{black}{c/l/r}}}
\newcommand{\elbow}[1][black]{\raisebox{.1cm}{\pipeDreamMonoColor{#1}{black}{e}}}

\newcommand{\pic}[1][black]{\raisebox{.15cm}{\!\!\!\pipeDreamBiColor{white}{#1}{black}{e/l/r}}}
\newcommand{\valley}[1][black]{\pipeDreamBiColor{#1}{white}{black}{e/l/r}\!\!}
\newcommandx{\pipe}[1][1 = p]{\mathrm{#1}} 
\newcommandx{\elb}[1][1 = e]{\mathrm{#1}} 
\newcommand{\SE}{\textsc{se}} 
\newcommand{\WN}{\textsc{wn}} 
\newcommand{\WE}{\textsc{we}} 
\newcommand{\SN}{\textsc{sn}} 
\newcommand{\subwordComplex}{\mathcal{SC}} 
\newcommand{\Q}{\mathrm{Q}} 
\newcommand{\sqc}{\mathrm{c}} 
\newcommand{\deletion}[2]{\ensuremath{#1\,\rotatebox{90}{\tiny $\ll$}\, #2}} 
\newcommand{\insertion}[2]{\ensuremath{#1\,\rotatebox{90}{\tiny $\gg$}\, #2}} 

\newcommand{\meet}{\wedge} 
\newcommand{\join}{\vee} 
\newcommand{\less}{\vartriangleleft} 
\newcommand{\more}{\vartriangleright} 
\newcommand{\contactLess}[1]{\less_{#1}} 
\newcommand{\contactMore}[1]{\more_{#1}} 
\newcommand{\projDown}{\pi_\downarrow} 
\newcommand{\projUp}{\pi^\uparrow} 

\newcommand{\contact}{^\#} 
\newcommand{\duality}{^\star} 
\newcommandx{\surjectionPermBrick}[1][1=k]{\mathsf{ins}^{#1}} 
\newcommandx{\surjectionBrickZono}[1][1=k]{\mathsf{can}^{#1}} 
\newcommandx{\surjectionPermZono}[1][1=k]{\mathsf{rec}^{#1}} 
\newcommandx{\restrictionBrick}[1][1=k\to\ell]{\mathsf{res}^{#1}} 
\newcommandx{\restrictionZono}[1][1=k\to\ell]{\mathsf{res}^{#1}} 
\newcommand{\linearExtensions}{\mathcal{L}} 

\newcommandx{\twist}[1][1=T]{\mathrm{#1}} 
\newcommand{\tree}{\mathrm{T}} 
\newcommand{\triangulation}{\mathrm{T}} 
\newcommandx{\orientation}[1][1=O]{\mathrm{#1}} 
\newcommandx{\Twists}[1][1 = k]{\mathcal{T}^{#1}} 
\newcommandx{\AcyclicTwists}[1][1 = k]{\mathcal{AT}^{#1}} 
\newcommandx{\AcyclicHyperTwists}[1][1 = k]{\mathcal{AHT}^{#1}} 
\newcommandx{\IndecomposableAcyclicTwists}[1][1 = k]{\mathcal{IAT}^{#1}} 
\newcommandx{\BaxterAcyclicTwists}[1][1 = k]{\mathcal{BAT}^{#1}} 
\newcommandx{\AcyclicOrientations}[1][1 = k]{\mathcal{AO}^{#1}} 
\newcommandx{\AcyclicPartialOrientations}[1][1 = k]{\mathcal{APO}^{#1}} 
\newcommandx{\TwistTuples}[1][1 = k]{\mathcal{TT}^{#1}} 

\newcommandx{\Perm}[2][1=k, 2=n]{\mathsf{Perm}^{#1}(#2)} 
\newcommandx{\Brick}[2][1=k, 2=n]{\mathsf{Brick}^{#1}(#2)} 
\newcommandx{\SummandBrick}[3][1=k, 2=n, 3=b]{\mathsf{Brick}^{#1}_{#3}(#2)} 
\newcommandx{\Zono}[2][1=k, 2=n]{\mathsf{Zono}^{#1}(#2)} 
\newcommandx{\polygon}[1][1=\signature]{\mathrm{P}^{#1}} 
\newcommand{\Cone}{\mathrm{C}} 
\newcommand{\HH}{\mathbb{H}} 
\newcommand{\Hyp}{\b{H}^=} 
\newcommand{\HS}{\b{H}^\ge} 
\newcommandx{\Gkn}[2][1=k, 2=n]{\graphG^{#1}(#2)} 
\newcommand{\graphG}{\mathrm{G}} 
\newcommand{\integerPointTransform}{\mathbb{Z}} 
\newcommand{\fan}{\mathcal{F}} 

\newcommand{\product}{\cdot} 
\newcommand{\coproduct}{\triangle} 
\newcommand{\shiftedShuffle}{\,\bar\shuffle\,} 
\newcommand{\convolution}{\star} 
\newcommand{\mirror}[1]{{#1}^{\bullet}} 
\newcommand{\FQSym}{\mathsf{FQSym}} 
\newcommand{\PBT}{\mathsf{PBT}} 
\newcommandx{\Camb}{\mathsf{Camb}} 
\newcommand{\OrdPart}{\mathsf{OrdPart}} 
\newcommandx{\DSym}[1][1 = k]{\mathsf{DSym}^{#1}} 
\newcommandx{\Rec}[1][1 = k]{\mathsf{Rec}^{#1}} 
\newcommandx{\Twist}[1][1 = k]{\mathsf{Twist}^{#1}} 
\newcommandx{\HyperTwist}[1][1 = k]{\mathsf{HyperTwist}^{#1}} 
\newcommandx{\HyperRec}[1][1 = k]{\mathsf{HyperRec}^{#1}} 
\newcommand{\F}{\mathbb{F}} 
\newcommand{\G}{\mathbb{G}} 
\newcommand{\EFQSym}{\mathbb{E}} 
\newcommand{\HFQSym}{\mathbb{H}} 
\newcommand{\PTwist}{\mathbb{P}} 
\newcommand{\QTwist}{\mathbb{Q}} 
\newcommand{\HTwist}{\mathbb{H}} 
\newcommand{\ETwist}{\mathbb{E}} 
\newcommand{\XRec}{\mathbb{X}} 
\newcommand{\underprod}[2]{{#1}\backslash{#2}} 
\newcommand{\overprod}[2]{{#1}\slash{#2}} 
\newcommand{\edgecut}[2]{\left( #1 \;\middle\|\; #2 \right)} 

\newcommand{\op}[1]{
	\foreach \letter in {#1} {
		\ifthenelse{\equal{\letter}{l}}{{\prec}}{}
		\ifthenelse{\equal{\letter}{m}}{{\prec\hspace*{-.27cm}\succ}}{} 
		\ifthenelse{\equal{\letter}{r}}{{\succ}}{}
	}
}
\newcommand{\coop}[1]{
	\foreach \letter in {#1} {
		\ifthenelse{\equal{\letter}{l}}{\rotatebox{90}{\hspace*{-.03cm}$\prec$}}{}
		\ifthenelse{\equal{\letter}{m}}{\raisebox{-.03cm}{\hspace*{.03cm}\rotatebox{90}{$\prec\hspace*{-.27cm}\succ$}\hspace*{.03cm}}}{}
		\ifthenelse{\equal{\letter}{r}}{\rotatebox{90}{\hspace*{-.03cm}$\succ$}}{}
	}
}
\newcommandx{\operation}[1][1=b]{\mathfrak{#1}} 
\newcommand{\Operations}{\mathfrak{B}} 
\newcommand{\algebra}{\mathsf{Alg}} 
\newcommandx{\cooperation}[1][1=d]{\mathfrak{#1}} 
\newcommand{\sfx}{\mathsf{x}} 
\newcommand{\sfy}{\mathsf{y}} 
\newcommand{\sfz}{\mathsf{z}} 

\newcommand{\signature}{\varepsilon} 
\newcommand{\signatures}{\mathcal{E}} 
\newcommandx{\shape}[2][1=k, 2=\signature]{\mathsf{Sh}_{#2}^{#1}} 
\newcommand{\psignature}{\signature_p} 
\newcommand{\vsignature}{\signature_v} 
\newcommand{\up}[1]{\overline{#1}} 
\newcommand{\upr}[1]{{\red \overline{#1}}} 
\newcommand{\upb}[1]{{\blue \overline{#1}}} 
\newcommand{\upw}[1]{\underaccent{\phantom{.}}{\up{#1}}}
\newcommand{\uptilde}[1]{\accentset{\vspace{-.03cm}\sim}{#1}}
\newcommand{\down}[1]{\underline{#1}} 
\newcommand{\downr}[1]{{\red \underline{#1}}} 
\newcommand{\downb}[1]{{\blue \underline{#1}}} 
\newcommand{\downw}[1]{\accentset{\phantom{.}}{\down{#1}}}
\newcommand{\downtilde}[1]{\underaccent{\,\sim}{#1}}
\newcommand{\simtilde}[1]{\accentset{\vspace{-.01cm}\sim}{#1}}
\newcommand{\tuple}{\mathbf{T}} 
\newcommand{\sep}{|} 

\newcommand{\fref}[1]{Figure~\ref{#1}} 
\newcommand{\ie}{\textit{i.e.}~} 
\newcommand{\eg}{\textit{e.g.}~} 
\newcommand{\aka}{\textit{aka.}~} 
\newcommand{\apriori}{\textit{a priori}} 
\newcommand{\para}[1]{\vspace{.3cm}\noindent\framebox{\textsc{#1}}} 
\newcommand{\defn}[1]{\emph{\textsf{\color{blue} #1}}} 
\definecolor{green}{RGB}{57,181,74} 
\newcommand{\blue}{\color{blue}} 
\newcommand{\red}{\color{red}} 
\renewcommand{\b}[1]{\mathbf{#1}} 
\newcommand{\dash}{\,\text{--}\,}

\makeatletter
\def\part{\@startsection{part}{1}%
\z@{.7\linespacing\@plus\linespacing}{.5\linespacing}%
{\LARGE\sffamily\centering}}
\@addtoreset{section}{part}
\makeatother

\makeatletter
\def\l@part{\@tocline{1}{6pt}{0pc}{}{}}
\def\l@section{\@tocline{1}{4pt}{0pc}{}{}}
\makeatother
\let\oldtocpart=\tocpart
\renewcommand{\tocpart}[2]{\hspace{0em}\bf\large\oldtocpart{#1}{#2}}
\let\oldtocsection=\tocsection
\renewcommand{\tocsection}[2]{\hspace{0em}\bf\oldtocsection{#1}{#2}}

\begin{document}

\begin{abstract}
This paper is motivated by the interplay between the Tamari lattice, J.-L.~Loday's realization of the associahedron, and J.-L.~Loday and M.~Ronco's Hopf algebra on binary trees. We show that these constructions extend in the world of acyclic $k$-triangulations, which were already considered as the vertices of V.~Pilaud and F.~Santos' brick polytopes. We describe combinatorially a natural surjection from the permutations to the acyclic $k$-triangulations. We show that the fibers of this surjection are the classes of the congruence~$\equiv^k$ on~$\fS_n$ defined as the transitive closure of the rewriting rule~$U ac V_1 b_1 \cdots V_k b_k W \equiv^k U ca V_1 b_1 \cdots V_k b_k W$ for letters~${a < b_1, \dots, b_k < c}$ and words~$U, V_1, \dots, V_k, W$ on~$[n]$. We then show that the increasing flip order on $k$-triangulations is the lattice quotient of the weak order by this congruence. Moreover, we use this surjection to define a Hopf subalgebra of C.~Malvenuto and C.~Reutenauer's Hopf algebra on permutations, indexed by acyclic $k$-triangulations, and to describe the product and coproduct in this algebra and its dual in term of combinatorial operations on acyclic $k$-triangulations. Finally, we extend our results in three directions, describing a Cambrian, a tuple, and a Schr\"oder version of these constructions.
\end{abstract}

\vspace*{-1cm}
\maketitle
\vspace*{-.5cm}
\enlargethispage{-.4cm}
\tableofcontents
\vspace*{-.7cm}
\enlargethispage{.3cm}


\section*{Introduction}

The motivation of this paper comes from relevant combinatorial, geometric, and algebraic structures on permutations, binary trees and binary sequences. Classical surjections from permutations to binary trees (BST insertion) and from binary trees to binary sequences (canopy) yield:
\begin{itemize}
\item lattice morphisms from the weak order, via the Tamari lattice, to the boolean lattice;
\item normal fan coarsenings from the permutahedron, via J.-L.~Loday's associahedron~\cite{Loday}, to the parallelepiped generated by the simple roots~$\b{e}_{i+1} - \b{e}_i$;
\item Hopf algebra refinements from C.~Malvenuto and C.~Reutenauer's algebra~\cite{MalvenutoReutenauer}, via \mbox{J.-L.~Loday} and M.~Ronco's algebra~\cite{LodayRonco}, to the descent algebra of~\cite{GelfandKrobLascouxLeclercRetakhThibon}.
\end{itemize}

These fascinating connections were widely extended by N.~Reading in his work on ``Lattice congruences, fans and Hopf algebras''~\cite{Reading-HopfAlgebras}. In particular, he proves that any lattice congruence~$\equiv$ of the weak order on the permutations of~$\fS_n$ (see Section~\ref{subsec:latticeCongruences} for proper definitions) defines a complete simplicial fan~$\fan_\equiv$ refined by the Coxeter fan, and he characterizes in terms of simple rewriting rules the families~$(\equiv_n)_{n \in \N}$ of lattice congruences of the weak orders on~$(\fS_n)_{n \in \N}$ which yield Hopf subalgebras of C.~Malvenuto and C.~Reutenauer's algebra on permutations. His work opens two natural questions. On the geometric side, it is not clear which of the fans~$\fan_\equiv$ are actually normal fans of polytopes, as in the previous example of the associahedron. On the algebraic side, this construction produces a combinatorial Hopf algebra whose basis is indexed by the congruence classes of~$(\equiv_n)_{n \in \N}$. However, the product and coproduct in this Hopf algebra are performed extrinsically: the algebra is embedded in C.~Malvenuto and C.~Reutenauer's algebra on permutations and the computations are performed at that level. The remaining challenge is to realize the resulting Hopf algebra intrinsically by attaching a combinatorial object to each congruence class of~$(\equiv_n)_{n \in \N}$ and working out the rules for product and coproduct directly on these combinatorial objects. The present paper answers these two questions for a relevant family of lattice congruences of the weak order, generalizing the classical sylvester congruence~\cite{HivertNovelliThibon-algebraBinarySearchTrees}. 

Our starting point is the world of acyclic multitriangulations. A \defn{$k$-triangulation} of a convex $(n+2k)$-gon is a maximal set of diagonals such that no $k+1$ of them are pairwise crossing. Multitriangulations were introduced by V.~Capoyleas and J.~Pach~\cite{CapoyleasPach} in the context of extremal theory for geometric graphs and studied for their rich combinatorial properties~\cite{Nakamigawa, DressKoolenMoulton, Jonsson, PilaudSantos-multitriangulations}. Using classical point-line duality, V.~Pilaud and M.~Pocchiola interpreted $k$-triangulations of the $(n+2k)$-gon as pseudoline arrangements on $n$-level sorting networks~\cite{PilaudPocchiola}, which can also be seen more combinatorially as beam arrangements in a trapezoidal shape~\cite[Section~4.1.4]{Pilaud-these}. In this paper, we call these specific arrangements \defn{$(k,n)$-twists}. As observed in~\cite{Stump, SerranoStump}, this connects multitriangulations to (specific) pipe dreams studied in~\cite{BergeronBilley, KnutsonMiller-GroebnerGeometry}. Motivated by possible polytopal realizations of the simplicial complex of $(k+1)$-crossing-free sets of diagonals of a convex $(n+2k)$-gon, V.~Pilaud and F.~Santos defined in~\cite{PilaudSantos-brickPolytope} the brick polytope of a sorting network, whose vertices correspond to certain \defn{acyclic} pseudoline arrangements on the network. When~$k = 1$, the pseudoline arrangements on the trapezoidal network correspond to the triangulations of the $(n+2)$-gon. They are all acyclic and the brick polytope coincides with J.-L.~Loday's associahedron~\cite{Loday}. The goal of this paper is to explore further properties of the acyclic $(k,n)$-twists for arbitrary~$k$ and~$n$.

The first part of this paper studies the combinatorial, geometric, and algebraic structure of the set of acyclic $(k,n)$-twists.
In Section~\ref{sec:combinatorics}, we present a purely combinatorial description of the natural map~$\surjectionPermBrick$ from the permutations of~$\fS_n$ to the acyclic $(k,n)$-twists in terms of successive insertions in a $k$-twist. Extending the sylvester congruence of~\cite{HivertNovelliThibon-algebraBinarySearchTrees}, we show that the fibers of this map are the classes of a congruence~$\equiv^k$ defined as the transitive closure of the rewriting rule~${U a c V_1 b_1 V_2 b_2 \cdots V_k b_k W \equiv^k U c a V_1 b_1 V_2 b_2 \cdots V_k b_k W}$ where $a, b_1, \dots, b_k, c \in [n]$ are such that~$a < b_i < c$ for all~$i \in [k]$ and $U, V_1, \dots, V_k, W$ are words on~$[n]$. This congruence is a lattice congruence of the weak order, so that the increasing flip order on the acyclic \mbox{$(k,n)$-twists} defines a lattice, generalizing the Tamari lattice. We also define a canopy map~$\surjectionBrickZono$ from the acyclic $(k,n)$-twists to the acyclic orientations of the graph~${\Gkn = ([n], \set{\{i,j\}}{|i-j| \le k})}$. For~$\tau \in \fS_n$, the composition~$\surjectionPermZono(\tau) = \surjectionBrickZono \big( \surjectionPermBrick(\tau) \big)$ records the relative positions in~$\tau$ of any entries~$i$ and~$j$ such that~$|i-j| \le k$, thus generalizing the recoils of the permutation~$\tau$. Note that this generalization of recoils was already considered by J.-C.~Novelli, C.~Reutenauer and J.-Y.~Thibon in~\cite{NovelliReutenauerThibon} with a slightly different presentation. To sum up, at the combinatorial level, we obtain a commutative triangle of lattice homomorphisms from the weak order, via the increasing flip order on the acyclic $(k,n)$-twists, to the lattice on acyclic orientations of~$\Gkn$. 

In Section~\ref{sec:geometry}, we survey and revisit the geometric aspects of these combinatorial maps. We recall the definitions of the classical permutahedron~$\Perm$, of the brick polytope~$\Brick$ of the $n$-level trapezoidal network~\cite{PilaudSantos-brickPolytope}, and of the zonotope~$\Zono$ of the graph~$\Gkn$. Their vertices correspond to the permutations of~$\fS_n$, to the acyclic $(k,n)$-twists, and to the acyclic orientations of~$\Gkn$, respectively. When oriented in the direction~$(n-1, n-3, \dots, 1-n)$, their $1$-skeleta are the Hasse diagrams of the weak order on permutations of~$\fS_n$, of the increasing flip lattice on acyclic $(k,n)$-twists, and of the lattice on acyclic orientations of~$\Gkn$. The maps~$\surjectionPermBrick$, $\surjectionPermZono$ and~$\surjectionBrickZono$ can be read as inclusions of normal cones of vertices of~$\Perm$, $\Brick$ and~$\Zono$. The reader can already glance at \fref{fig:PermBrickZono} on page~\pageref{fig:PermBrickZono} for an illustration of the geometric situation. Although most results in this section are not new, they provide the geometric side of the picture.

In Section~\ref{sec:algebra}, we present the new algebraic construction which motivated this paper. We consider C.~Malvenuto and C.~Reutenauer's Hopf algebra on permutations~\cite{MalvenutoReutenauer}, that we denote by~$\FQSym$. We consider the subspace~$\Twist$ generated by the sums of the elements of~$\FQSym$ over the fibers of~$\surjectionPermBrick$. Since these fibers are classes of a congruence~$\equiv^k$ which satisfy standard compatibility conditions with the shuffle and the standardization~\cite{Reading-HopfAlgebras, Hivert-habilitation, HivertNzeutchap, Priez}, this subspace automatically defines a Hopf subalgebra of~$\FQSym$. Our approach with $k$-twists provides a combinatorial interpretation for this subalgebra, and it is interesting to describe the product and the coproduct directly on acyclic $k$-twists. Note that our Hopf algebra~$\Twist$ on acyclic $k$-twists is sandwiched in between C.~Malvenuto and C.~Reutenauer's Hopf algebra~$\FQSym$ on permutations~\cite{MalvenutoReutenauer} and J.-C.~Novelli, C.~Reutenauer and J.-Y.~Thibon's $k$-recoil Hopf algebra~$\Rec$ on acyclic orientations of~$\Gkn$~\cite{NovelliReutenauerThibon}. We finally briefly study further algebraic properties of~$\Twist$: we define multiplicative bases and study their indecomposable elements, we connect it to integer point transforms of the normal cones of the brick polytope, and we define a natural extension of dendriform structures in the context of $k$-twists.

The second part of this paper is devoted to three independent extensions of these results. In Section~\ref{sec:Cambrianization}, we show that our constructions can be parametrized by a sequence of signs, following the same direction as~\cite{Reading-CambrianLattices, HohlwegLange, PilaudSantos-brickPolytope, ChatelPilaud}. We obtain generalizations of the Cambrian lattices, mention their connections to certain well-chosen brick polytopes, and construct a Cambrian Hopf algebra on acyclic $k$-twists with similar ideas as in~\cite[Part~1]{ChatelPilaud}. In Section~\ref{sec:tuplization}, we extend the Cambrian tuple algebra of G.~Chatel and V.~Pilaud \cite[Part~2]{ChatelPilaud}, which was motivated by the Hopf algebras on diagonal rectangulations described by S.~Law and N.~Reading~\cite{LawReading} and on twin binary trees described by S.~Giraudo~\cite{Giraudo}. Finally in Section~\ref{sec:Schroderization}, we study Hopf algebras on all faces of the brick polytope, motivated by the work of F.~Chapoton~\cite{Chapoton}, and its extension to the Cambrian Schr\"oder algebra in~\cite[Part~3]{ChatelPilaud}.

\newpage
\part{Acyclic twists: combinatorics, geometry, and algebra}
\label{part:acyclicTwists}

\medskip
\section{Combinatorics of twists}
\label{sec:combinatorics}

\subsection{Pipe dreams and twists}
\label{subsec:twists}

A \defn{pipe dream} is a filling of a triangular shape with crosses~\cross{} and elbows~\elbow{} so that all pipes entering on the left side exit on the top side. These objects were studied in the literature, under different names including ``pipe dreams''~\cite{KnutsonMiller-GroebnerGeometry}, ``RC-graphs''~\cite{BergeronBilley}, ``beam arrangements''~\cite{Pilaud-these,PilaudSantos-brickPolytope}. This paper is mainly concerned with the following specific family of pipe dreams which already appeared under different names in~\cite{Pilaud-these, PilaudPocchiola, Stump}.

\begin{definition}[\cite{Pilaud-these, PilaudPocchiola, Stump}]
\label{def:twist}
For~$k, n \in \N$, a \defn{$(k,n)$-twist} (we also use just \defn{$k$-twist}, or even just \defn{twist}) is a pipe dream with~$n+2k$ pipes such~that
\begin{itemize}
\item no two pipes cross twice (the pipe dream is reduced),
\item the pipe which enters in row~$i$ exits in column~$i$ if~$k \le i \le n+k$, and in column~${n+2k+1-i}$ otherwise.~Here and throughout the paper, rows are indexed from bottom to top and columns are indexed from left to right.
\end{itemize}
Besides the first $k$ and last $k$ trivial pipes, a $(k,n)$-twist has $n$ relevant pipes, labeled by~$[n]$ from bottom to top, or equivalently from left to right. In other words, the~$p$th pipe enters in row~$p+k$ and exits at column~$p+k$ of the $(n+2k) \times (n+2k)$-triangular shape. We denote by~$\Twists(n)$ the set of $(k,n)$-twists.
\end{definition}

\begin{definition}[\cite{PilaudSantos-brickPolytope}]
\label{def:contactGraph}
The \defn{contact graph} of a $(k,n)$-twist~$\twist$ is the directed multigraph~$\twist\contact$ with vertex set~$[n]$ and with an arc from the \SE-pipe to the \WN-pipe of each elbow in~$\twist$ involving two relevant pipes. We say that a twist~$\twist$ is \defn{acyclic} if its contact graph~$\twist\contact$ is (no oriented cycle), and we then let~$\contactLess{T}$ be the transitive closure of~$\twist\contact$. We denote by~$\AcyclicTwists(n)$ the set of \mbox{acyclic~$(k,n)$-twists}.
\end{definition}

\fref{fig:twists} illustrates examples of $(k,5)$-twists and their contact graphs for~$k = 0, 1, 2, 3$. The first two are acyclic, the last two are not. Except in Figures~\ref{fig:1twistsTriangulations} and~\ref{fig:ktwistsktriangulations}, we only represent the~$n$ relevant pipes of the $(k,n)$-twists and hide the other~$2k$ trivial pipes (the first~$k$ and last~$k$ pipes).

\begin{figure}[h]
	\centerline{\includegraphics[scale=1.4]{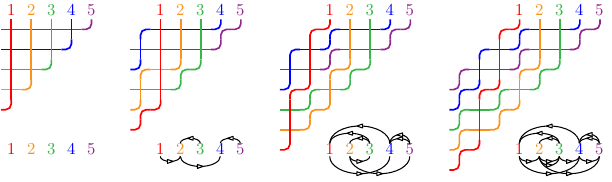}}
	\caption{$(k,5)$-twists (top) and their contact graphs (bottom) for~$k = 0, 1, 2, 3$.}
	\label{fig:twists}
\end{figure}

\begin{remark}[Subword complexes]
\label{rem:subwordComplexes}
Although we prefer a more visual presentation to develop an intuition which will be crucial later on, twists and their contact graphs can be defined in more algebraic terms. Namely, a $(k,n)$-twist could be equivalently defined as a reduced expression of the longest permutation~$w_\circ \eqdef [n, n-1, \dots, 2, 1]$ in the word~$\Q \eqdef (s_1 \cdots s_{n-1})^{k+1} (s_1 \cdots s_{n-2}) \cdots (s_1 s_2) s_1$ where~$s_i$ denotes the simple transposition~$(i \;\, i+1)$. In other words, it is a facet in the subword complex~$\subwordComplex(\Q,w_\circ)$, see~\cite{KnutsonMiller-subwordComplex}. The contact graph of the twist is then given by the root configuration of this facet~\cite{CeballosLabbeStump, PilaudStump-brickPolytope}. The twist is acyclic when this root configuration is pointed.
\end{remark}

\begin{example}[$0$-twists]
For any~$n \in \N$, there is a unique $(0,n)$-twist. It is the pipe dream where all boxes on the diagonal~$\set{(i,i)}{i \in [n]}$ are filled with elbows~\elbow{} while all other boxes are filled with crossings~\cross{}, see \fref{fig:twists}\,(left). Its contact graph is therefore an independent graph.
\end{example}

\begin{example}[$1$-twists, triangulations and binary trees]
\label{exm:1twistsTriangulations}
Following~\cite{Woo, Pilaud-these, PilaudPocchiola, Stump}, we observe that $(1,n)$-twists are in bijective correspondence with triangulations of a convex $(n+2)$-gon. This bijection is illustrated in \fref{fig:1twistsTriangulations}. To describe it explicitly, consider:
\begin{itemize}
\item the $(n+2) \times (n+2)$ triangular shape where rows (resp.~columns) are labeled by~$\{0, \dots, n+1\}$ from bottom to top (resp.~from left to right),
\item the convex $(n+2)$-gon where the vertices are labeled by~$\{0, \dots, n+1\}$ from left to right and where all vertices lie below the top boundary edge~$[0,n+1]$,
\item the map which sends an elbow in row~$i$ and column~$j$ of the triangular shape to the diagonal~$[i,j]$ of the~$(n+2)$-gon.
\end{itemize}
As illustrated in \fref{fig:1twistsTriangulations}, this map sends a $(1,n)$-twist~$\twist$ to a triangulation~$\twist\duality$ of the $(n+2)$-gon:
\begin{itemize}
\item the $p$th relevant pipe of~$\twist$ corresponds to the $p$th triangle of~$\twist\duality$ (with central vertex~$p$),
\item the elbows of~$\twist$ correspond to the diagonals of~$\twist\duality$,
\item the crossings of~$\twist$ correspond to the common bisectors between triangles of~$\twist\duality$,
\item the contact graph of~$\twist$ corresponds to the dual binary tree of~$\twist\duality$,
\item elbow flips in~$\twist$ (see \fref{fig:1twistsTriangulations} and Section~\ref{subsec:flips}) corresponds to diagonal flips in~$\twist\duality$.
\end{itemize}
\end{example}

\begin{figure}[h]
	\centerline{\includegraphics[scale=1.4]{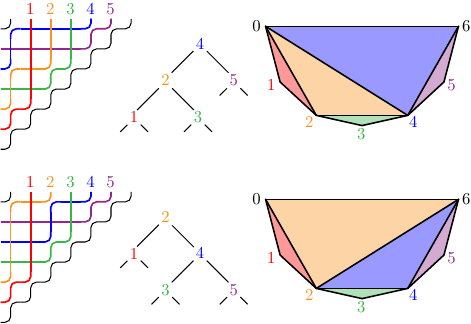}}
	\caption{The bijection between $(1,n)$-twists~$\twist$ (left) and triangulations~$\twist\duality$ of the $(n+2)$-gon (right) sends the pipes of~$\twist$ to the triangles of~$\twist\duality$, the elbows of~$\twist$ to the diagonals of~$\twist\duality$, the contact graph of~$\twist$ to the dual binary tree of~$\twist\duality$ (middle), and the elbow flips in~$\twist$ to the diagonal flips in~$\twist\duality$. See also Section~\ref{sec:Cambrianization}.}
	\label{fig:1twistsTriangulations}
	\vspace*{-.4cm}
\end{figure}

\subsection{Twists and multitriangulations}
\label{subsec:ktriangulations}

\enlargethispage{.2cm}
Following Example~\ref{exm:1twistsTriangulations}, we now recall the bijective correspondence between $(k,n)$-twists and $k$-triangulations of a convex $(n+2k)$-gon. This correspondence is not strictly required for our constructions, but was at the initial motivation of our study of $k$-twists. The reader familiar with multitriangulations can jump directly to Section~\ref{subsec:elementaryPropertiesPipes}.

Multitriangulations appeared in the context of extremal theory for geometric graphs~\cite{CapoyleasPach} and were studied for their combinatorial structure \cite{Nakamigawa, DressKoolenMoulton, Jonsson}. We refer to~\cite{PilaudSantos-multitriangulations} for a recent overview of their properties.

\begin{definition}
A \defn{$k$-triangulation} of a convex $(n+2k)$-gon is a maximal set of diagonals of the $(n+2k)$-gon such that no $k+1$ of them are pairwise crossing.
\end{definition}

\fref{fig:ktriangulations} illustrates this definition with examples of $k$-triangulations for $k = 0, 1,2,3$. Observe from the definition that $1$-triangulations are just classical triangulations. Note also that a diagonal of the $(n+2k)$-gon can appear in a $(k+1)$-crossing only if it has at least $k$ vertices of the $(n+2k)$-gon on each side. We say that those diagonals are \defn{$k$-relevant diagonals}, that the diagonals separating $k-1$ vertices of the $(n+2k)$-gon from the others are \defn{$k$-boundary diagonals}, and that the other ones are \defn{$k$-irrelevant diagonals}. To connect $k$-triangulations to $k$-twists, we need their decompositions into $k$-stars, introduced in~\cite{PilaudSantos-multitriangulations}. See \fref{fig:ktriangulations} for illustrations.

\begin{figure}
	\centerline{\includegraphics[width=\textwidth]{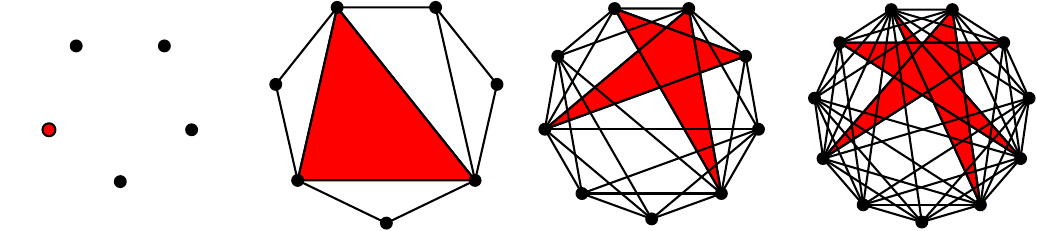}}
	\caption{$k$-triangulations of the $(5+2k)$-gon with a colored $k$-star, for~$k = 0, 1, 2, 3$.}
	\label{fig:ktriangulations}
\end{figure}

\begin{definition}
A \defn{$k$-star} is a star $(2k+1)$-gon, given by $2k+1$ vertices~$s_0, \dots, s_{2k}$ cyclically ordered connected by the $2k+1$ edges~$[s_i,s_{i+k}]$ (indices are taken modulo~$2k+1$).
\end{definition}

The $k$-stars were used in~\cite{PilaudSantos-multitriangulations} to show by double counting that the number of diagonals in a $k$-triangulation of the $(n+2k)$-gon only depends on~$k$ and~$n$, a property known in~\cite{Nakamigawa, DressKoolenMoulton}.

\begin{theorem}[\protect{\cite[Thm.~1.4\,(1\,\&\,2)]{PilaudSantos-multitriangulations}}]
A $k$-triangulation of the $(n+2k)$-gon contains precisely $n$ $k$-stars and $k(2n+2k-1)$ diagonals, $k(n-1)$ of which are $k$-relevant. Moreover, each $k$-relevant (resp.~$k$-boundary, resp.~$k$-irrelevant) diagonal of a $k$-triangulation~$\triangulation$ is contained in~$2$ (resp.~$1$, resp.~$0$) $k$-stars of~$\triangulation$.
\end{theorem}

Call \defn{dual graph} of a $k$-triangulation~$\triangulation$ the graph~$\triangulation\contact$ with one vertex for each $k$-star of~$\triangulation$ and an edge for each $k$-relevant diagonal~$\delta$ of~$\triangulation$ connecting the two $k$-stars of~$\triangulation$ containing~$\delta$. This graph has~$n$ nodes and~$k(n-1)$ edges, and it is know to be a \defn{$k$-arborescence}: it can be partitioned into $k$ edge-disjoint spanning trees~\cite[Prop.~4.20]{Pilaud-these}.

The $k$-stars were also used in~\cite{PilaudSantos-multitriangulations} to provide a local combinatorial description of the notion of flip in $k$-triangulations introduced in~\cite{Nakamigawa, DressKoolenMoulton, Jonsson}. In the next statement, a \defn{bisector} of a $k$-star~$S$ is a line passing through one of its vertices and with $k$ vertices of~$S$ on each side.

\begin{theorem}[\protect{\cite[Thm.~3]{Nakamigawa}, \cite[Thm.~1.4\,(3)]{PilaudSantos-multitriangulations}}]
Consider a $k$-relevant diagonal~$\delta$ in a \mbox{$k$-triangulation~$\triangulation$} of the $(n+2k)$-gon. Then there exists a unique other $k$-triangulation~$\triangulation'$ of the $(n+2k)$-gon such that~$\triangulation \ssm \{\delta\} = \triangulation' \ssm \{\delta'\}$ for some diagonal~$\delta' \in \triangulation'$. Moreover, $\delta'$ is the unique common bisector of the two $k$-stars of~$\triangulation$ containing~$\delta$.
\end{theorem}

We are now ready to recall the correspondence between $k$-twists and $k$-triangulations. Besides the special case~$k = 1$ treated in~\cite{Woo}, this correspondence first appeared for arbitrary values of~$k$ in the work of V.~Pilaud and M.~Pocchiola~\cite[Sections~3.2.3 and~4.1.4]{Pilaud-these}\cite{PilaudPocchiola}. C.~Stump~\cite{Stump} and L.~Serrano and C.~Stump~\cite{SerranoStump} rediscovered this connection and used it to get enumerative and bijective results discussed in Section~\ref{subsec:numerology}. This correspondence is illustrated in Figures~\ref{fig:1twistsTriangulations} and~\ref{fig:ktwistsktriangulations}. The $k$-triangulations of \fref{fig:ktriangulations} also correspond to the $k$-twists of \fref{fig:twists}.

\begin{theorem}[\protect{\cite[Thm.~23]{PilaudPocchiola}}]
\label{theo:ktwistsktriangulations}
Consider the map that sends an elbow in row~$i$ and column~$j$ of the $(n+2k) \times (n+2k)$ triangular shape to the diagonal~$[i,j]$ of a convex $(n+2k)$-gon. It maps a $(k,n)$-twist~$\twist$ to a $k$-triangulation~$\twist\duality$ of the $(n+2k)$-gon:
\begin{itemize}
\item the relevant pipes of~$\twist$ correspond to the $k$-stars of~$\twist\duality$,
\item the (relevant) elbows of~$\twist$ correspond to the (relevant) diagonals in~$\twist\duality$,
\item the crossings of~$\twist$ correspond to the common bisectors between $k$-stars of~$\twist\duality$,
\item the contact graph of~$\twist$ correspond to the dual graph of~$\twist\duality$ (its incidence graph of $k$-stars),
\item elbow flips in~$\twist$ (see \fref{fig:ktwistsktriangulations} and Section~\ref{subsec:flips}) correspond to diagonal flips in~$\twist\duality$.
\end{itemize}
Moreover, this map defines a bijective correspondence between the $(k,n)$-twists and the $k$-triangula\-tions of the $(n+2k)$-gon.
\end{theorem}

\begin{figure}[h]
	\centerline{\includegraphics[scale=1.4]{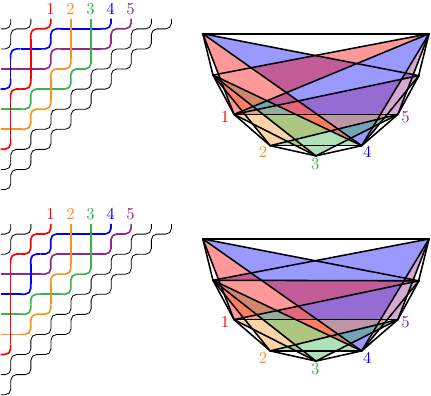}}
	\caption{The bijection between $(k,n)$-twists~$\twist$ (left) and $k$-triangulations~$\twist\duality$ of the $(n+2k)$-gon (right) sends the pipes of~$\twist$ to the $k$-stars of~$\twist\duality$, the elbows of~$\twist$ to the diagonals of~$\twist\duality$, the contact graph of~$\twist$ to the dual graph of~$\twist\duality$, and the elbow flips in~$\twist$ to the diagonal flips in~$\twist\duality$. See also Section~\ref{sec:Cambrianization}.}
	\label{fig:ktwistsktriangulations}
\end{figure}

\subsection{Numerology}
\label{subsec:numerology}

J.~Jonsson~\cite{Jonsson} proved that $k$-triangulations are counted by an Hankel determinant of Catalan numbers, which was known to count fans of $k$ non-crossing Dyck path by the celebrated non-intersecting path enumeration technique of I.~Gessel and X.~Viennot~\cite{GesselViennot}. The correspondence between $k$-triangulations and $k$-twists in Theorem~\ref{theo:ktwistsktriangulations} thus shows that the $k$-twists are also counted by this Hankel determinant. L.~Serrano and C.~Stump revisited independently this correspondence to $k$-twists and gave a remarkable explicit bijection between $k$-triangulations and $k$-tuples of non-crossing Dyck paths~\cite{SerranoStump}. We summarize these numerology results in the following statement, and Table~\ref{table:numbersAllTwists} gathers the numbers of $(k,n)$-twists for~$k, n \le 9$.

\begin{theorem}[\protect{\cite[Thms.~14 \,\&\,15]{Jonsson}, \cite[Thm.~23]{PilaudPocchiola}, \cite[Thm.~1.1]{SerranoStump}}]
\label{theo:numberktriangulations}
The three combinatorial families consisting in
\begin{itemize}
\item the $k$-twists with $n$ relevant pipes,
\item the $k$-triangulations of the $(n+2k)$-gon,
\item the fans of $k$ non-crossing Dyck paths of length~$2n$,
\end{itemize}
are equinumerous and counted by the Hankel determinant~$\det(C_{n+2k-i-j})_{i,j \in [k]}$ of Catalan numbers~$C_m \eqdef \frac{1}{m+1}\binom{2m}{m}$. See \href{https://oeis.org/A000108}{\cite[A000108]{OEIS}} and \href{https://oeis.org/A078920}{\cite[A078920]{OEIS}}.
\end{theorem}

To our knowledge, there is no enumerative results on acyclic twists. Table~\ref{table:numbersAcyclicTwists} gathers the numbers of acyclic $(k,n)$-twists for~$k < n \le 10$ (we will see later that the number of acyclic $(k,n)$-twists is always~$n!$ when~$k \ge n$). See \href{https://oeis.org/A263791}{\cite[A263791]{OEIS}}. The reader is invited to compare Tables~\ref{table:numbersAllTwists} and~\ref{table:numbersAcyclicTwists}. We leave open the general question of finding a counting formula for $(k,n)$-twists:

\begin{question}
\label{qu:numerology}
Is there a nice closed formula for the number of acyclic $(k,n)$-twists?
\end{question}

\begin{table}[h]
  \centerline{$
  	\begin{array}{l|rrrrrrrrr}
	\raisebox{-.1cm}{$k$} \backslash \, \raisebox{.1cm}{$n$}
	  & 1 & 2 & 3 & 4 & 5 & 6 & 7 & 8 & 9 \\[.1cm]
	\hline
	0 & 1 &  1 &   1 &     1 &       1 &          1 &             1 &                1 &                   1 \\
	1 & 1 &  2 &   5 &    14 &      42 &        132 &           429 &             1430 &                4862 \\
	2 & 1 &  3 &  14 &    84 &     594 &       4719 &         40898 &           379236 &             3711916 \\
	3 & 1 &  4 &  30 &   330 &    4719 &      81796 &       1643356 &         37119160 &           922268360 \\
	4 & 1 &  5 &  55 &  1001 &   26026 &     884884 &      37119160 &       1844536720 &        105408179176 \\
	5 & 1 &  6 &  91 &  2548 &  111384 &    6852768 &     553361016 &      55804330152 &       6774025632340 \\
	6 & 1 &  7 & 140 &  5712 &  395352 &   41314284 &    6018114036 &    1153471856900 &     278563453441350 \\
	7 & 1 &  8 & 204 & 11628 & 1215126 &  204951252 &   51067020290 &   17618122000050 &    8012241391386375 \\
	8 & 1 &  9 & 285 & 21945 & 3331251 &  869562265 &  354544250775 &  210309203300625 &  171879976152056250 \\
	9 & 1 & 10 & 385 & 38962 & 8321170 & 3245256300 & 2085445951875 & 2045253720802500 & 2885318087540733000
    \end{array}
  $}
  \caption{The number~$|\Twists(n)| = \det(C_{n+2k-i-j})_{i,j \in [k]}$ of $(k,n)$-twists for~$k, n \le 9$.}
  \label{table:numbersAllTwists}
  \vspace{-.4cm}
\end{table}
\begin{table}[h]
  \centerline{$
  	\begin{array}{l|rrrrrrrrrr}
	\raisebox{-.1cm}{$k$} \backslash \, \raisebox{.1cm}{$n$}
	  & 1 & 2 & 3 & 4 & 5 & 6 & 7 & 8 & 9 & 10 \\[.1cm]
	\hline
	0 & 1 & 1 & 1 &  1 &   1 &   1 &    1 &     1 &      1 &       1 \\
	1 & . & 2 & 5 & 14 &  42 & 132 &  429 &  1430 &   4862 &   16796 \\
	2 & . & . & 6 & 22 &  92 & 420 & 2042 & 10404 &  54954 &  298648 \\
	3 & . & . & . & 24 & 114 & 612 & 3600 & 22680 & 150732 & 1045440 \\
	4 & . & . & . &  . & 120 & 696 & 4512 & 31920 & 242160 & 1942800 \\
	5 & . & . & . &  . &   . & 720 & 4920 & 37200 & 305280 & 2680800 \\
	6 & . & . & . &  . &   . &   . & 5040 & 39600 & 341280 & 3175200 \\
	7 & . & . & . &  . &   . &   . &    . & 40320 & 357840 & 3457440 \\
	8 & . & . & . &  . &   . &   . &    . &     . & 362880 & 3588480 \\
	9 & . & . & . &  . &   . &   . &    . &     . &      . & 3628800
	\end{array}
  $}
  \caption{The number~$|\AcyclicTwists(n)|$ of acyclic $(k,n)$-twists for~$k < n \le 10$. Dots indicate that the value remains constant (equal to~$n!$) in the column.}
  \label{table:numbersAcyclicTwists}
  \vspace{-.4cm}
\end{table}

\subsection{Elementary properties of pipes and twists}
\label{subsec:elementaryPropertiesPipes}

We now give some elementary properties of pipes and twists needed in the next sections. For a given pipe in a twist, we~call:
\begin{itemize}
\item \defn{\WE-crosses} its horizontal crosses~$\crossBiColor{red}{lightgray}$, and \defn{\SN-crosses} its vertical crosses~$\crossBiColor{lightgray}{red}$,
\item \defn{\SE-elbows} its elbows~$\pic[red]$ (\aka ``peaks'') and \defn{\WN-elbows} its elbows~$\valley[red]$ (\aka ``valleys''),
\item \defn{internal steps} the segments between two consecutive elbows and \defn{external steps} the first and last steps of the pipe.
\end{itemize}

The following statement gathers some elementary properties of pipes in $k$-twists.

\begin{lemma}
\label{lem:pipeProperties}
The following properties hold for any~$k,n \in \N$:
\begin{enumerate}[(i)]
\item A $(k,n)$-twist has $\binom{n}{2}$ crosses~$\cross{}$ and~$kn$ elbows~$\elbow{}$.
\label{item:totalNumberCrosses}
\item The $p$th pipe of a $(k,n)$-twist has 
	\begin{itemize}
	\item $n-1$ crosses: $p-1$ \WE-crosses~$\crossBiColor{red}{lightgray}$ and $n-p$ \SN-crosses~$\crossBiColor{lightgray}{red}$,
	\item $2k+1$ elbows: $k$ \SE-elbows~$\pic[red]$ and $k+1$ \WN-elbows~$\valley[red]$,
	\item $2k+2$ steps: $k+1$ vertical steps and $k+1$ horizontal steps.
	\end{itemize}
\label{item:numberElbows}
\item The interior of the rectangle defined by two consecutive steps of a pipe contains no elbow.
\label{item:emptyRectangle}
\end{enumerate}
\end{lemma}

\begin{proof}
For~\eqref{item:totalNumberCrosses}, observe that any two pipes cross exactly once and that the number of relevant boxes in the triangular shape is~$kn+\binom{n}{2}$. Point~\eqref{item:numberElbows}, follows from the fact that a pipe~$\pipe$ has a \SN-cross for each pipe~$\pipe' < \pipe$ and a \WE-cross for each pipe~$\pipe' > \pipe$, and that the length of each pipe is~$n+2k$. (Note that the $2k+1$ elbows of pipe~$\pipe$ correspond to the $2k+1$ edges of the dual $k$-star~$\pipe\duality$.) To see~\eqref{item:emptyRectangle}, consider a rectangle~$R$ defined by two consecutive steps of a pipe~$\pipe$ and assume that~$R$ is located below~$\pipe$. Observe that the pipes entering on the west side of~$R$ have to exit on its east side of~$R$ (otherwise, they would exit on its north side, and would thus cross~$\pipe$ twice). This prevents an elbow inside~$R$. The argument is similar if the rectangle~$R$ is above~$\pipe$.
\end{proof}

The next statement provides a sufficient condition for the existence of a path between two pipes in the contact graph of a $k$-twist. The converse statement holds when~$k = 1$, but not~in~general.

\begin{lemma}
\label{lem:comparable}
Consider two pipes~$\pipe, \pipe'$ in a twist~$\twist$. If there is an elbow of~$\pipe$ located weakly south-east of an elbow of~$\pipe'$, then there is a path from~$\pipe$ to~$\pipe'$ in the contact graph~$\twist\contact$ of~$\twist$.
\end{lemma}

\begin{proof}
Up to adding one more edge at the begining or at the end of the path form~$\pipe$ to~$\pipe'$, we can assume without loss of generality that there is a \WN-elbow~$\elb$ of~$\pipe$ located south-east of a \SE-elbow~$\elb'$ of~$\pipe'$. We claim that there exists a sequence~$\elb = \elb_0, \elb_1, \dots, \elb_\ell = \elb'$ of elbows such that each elbow~$\elb_i$ is located weakly south-east of~$\elb'$ and the elbows~$\elb_i$ and~$\elb_{i+1}$ are connected by a pipe~$\pipe_{i+1}$ of~$\twist$. Indeed, assume that we have constructed~$\elb = \elb_0, \elb_1, \dots, \elb_i$, and consider the \WN-pipe~$\pipe_{i+1}$ at~$\elb_i$. By Lemma~\ref{lem:pipeProperties}\,\eqref{item:emptyRectangle}, either the previous or the next elbow along~$\pipe_{i+1}$ is still located weakly south-east of~$\elb'$. Choose~$\elb_{i+1}$ accordingly (pick arbitrarily one if the two options are possible). The process ends at~$\elb'$ since the distance to~$\elb'$ decreases at each step. Finally, the sequence of pipes~$\pipe = \pipe_1, \dots, \pipe_\ell = \pipe'$ gives a path from~$\pipe$ to~$\pipe'$ in the contact graph~$\twist\contact$.
\end{proof}

\begin{corollary}
\label{coro:comparable}
For two pipes~$\pipe, \pipe'$, write~$\pipe < \pipe'$ when~$\pipe$ starts below and ends to the left of~$\pipe'$.
\begin{enumerate}[(i)]
\item If~$\pipe < \pipe'$ are two pipes of a twist~$\twist$ such that the crossing between~$\pipe$ and~$\pipe'$ is on an internal step of~$\pipe'$, then there is a path from~$\pipe$ to~$\pipe'$ in the contact graph~$\twist\contact$ of~$\twist$.
\item If~$\pipe < \pipe'$ are two pipes of a twist~$\twist$ and are incomparable in the contact graph~$\twist\contact$ of~$\twist$, then the last vertical step of~$\pipe$ crosses the first horizontal step~of~$\pipe'$.
\label{item:comparable}
\item If a twist~$\twist$ is acyclic, then no two pipes of~$\twist$ cross at internal steps of both.
\end{enumerate}
\end{corollary}

When~$\twist$ is acyclic, we denote by~$\contactLess{\twist}$ the transitive closure of its contact graph~$\twist\contact$. Note the difference between our notations~$\pipe < \pipe'$ (meaning that~$\pipe$ starts below and ends to the left of~$\pipe'$) and $\pipe \contactLess{\twist} \pipe'$ (meaning that~$\pipe$ is smaller than~$\pipe'$ in the contact graph~$\twist\contact$). 

\subsection{Pipe insertion and deletion}
\label{subsec:pipeInsertionDeletion}

We now define the pipe insertion and deletion, two reverse operations on $k$-twists: an insertion transforms a $(k,n)$-twist into a $(k,n+1)$-twist by inserting a single pipe while a deletion transforms a $(k,n+1)$-twist into a $(k,n)$-twist by deleting a single pipe. Insertions are always possible (see Definition~\ref{def:insertion}), while only certain pipes are allowed to be deleted (see Definition~\ref{def:deletion}). We start with pipe insertions.

\begin{definition}
\label{def:insertion}
Consider a $(k,n)$-twist~$\twist$ with an increasing relabeling~$\lambda:[n] \to \N$ of its relevant pipes and an integer~$q \in \N$. Let~$p \in \{0, \dots, n\}$ be such that~$\lambda(p) \le q < \lambda(p+1)$, where we set the convention~$\lambda(0) = -\infty$ and~$\lambda(n+1) = +\infty$. The \defn{pipe insertion} of~$q$ in the relabeled $(k,n)$-twist~$\twist$ produces the relabeled $(k,n+1)$-twist~$\twist \insertion{} q$ obtained by:
\begin{itemize}
\item inserting a row and a column in the triangular shape between~$p+k-1$ and~$p+k$,
\item filling in with elbows~\elbow{} the new boxes~$(p+k, p), (p+k+1, p+1), \dots, (p+2k, p+k)$,
\item filling in with crosses~\cross{} all other new boxes~$(p+k, j)$ for~$j < p$ and~$(i, p+k)$ for~$i > p+2k$,
\item relabeling the $r$th relevant pipe of the resulting twist by~$\lambda(r)$ if~$r < p$, by~$q$ if~$r = p$, and by~$\lambda(r-1)$~if~$r > p$.
\end{itemize}
\end{definition}

\fref{fig:pipeInsertion} illustrates the insertion of~$4$ in the $(k,5)$-twists of \fref{fig:twists} relabeled by~$[2,3,6,8,9]$. The following statement is an immediate consequence of the definition of pipe insertion.

\begin{figure}
	\centerline{\includegraphics[scale=1.4]{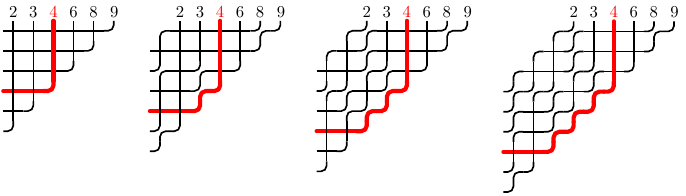}}
	\caption{Inserting~$4$ in the $(k,5)$-twists of \fref{fig:twists}. The inserted pipe is in bold red.}
	\label{fig:pipeInsertion}
\end{figure}

\begin{lemma}
\label{lem:source}
The contact graph~$(\twist \insertion{} q)\contact$ is obtained from~$\twist\contact$ by connecting to some existing nodes the new node corresponding to the inserted pipe~$q$. In particular, the node corresponding to the inserted pipe is a source of the contact graph~$(\twist \insertion{} q)\contact$. 
\end{lemma}

We now define the deletion, which just erases a pipe from a $(k,n+1)$-twist.

\begin{definition}
\label{def:deletion}
Consider a $(k,n+1)$-twist~$\twist$ with an increasing relabeling~$\lambda:[n+1] \to \N$ of its relevant pipes. Assume that the $p$th pipe of~$\twist$, labeled by~$\lambda(p) = q$, is a source of the  contact graph~$\twist\contact$. Then the \defn{pipe deletion} of~$q$ in the relabeled $(k,n+1)$-twist~$\twist$ produces the relabeled $(k,n)$-twist~$\twist \deletion{} q$ obtained by:
\begin{itemize}
\item deleting the $(p+k)$th row and column of~$\twist$,
\item relabeling the $r$th relevant pipe of the resulting twist by~$\lambda(r)$ if~$r<p$ and by~$\lambda(r+1)$~if~$r \ge p$.
\end{itemize}
\end{definition}

The following statement is an immediate consequence of the definition of pipe deletion.

\begin{lemma}
The contact graph~$(\twist \deletion{} q)\contact$ is obtained from~$\twist\contact$ by deleting the node corresponding to the deleted pipe~$q$.
\end{lemma}

Moreover, the insertion and deletion are clearly reverse to each other.

\begin{lemma}
\label{lem:insertionDeletion}
For any $(k,n)$-twist~$\twist$ relabeled by~$\lambda:[n] \to \N$ and any integer~$q \in \N$, we have
\begin{itemize}
\item $(\twist \insertion{} q) \deletion{} q = \twist$, and
\item $(\twist \deletion{} q) \insertion{} q = \twist$ as soon as $q  \in \lambda([n])$ labels a source of~$\twist$.
\end{itemize}
\end{lemma}

\begin{example}[Insertion in $1$-twists, triangulations and binary search trees]
\label{exm:bijectionsInsertions}
The following operations are equivalent under the bijections between $1$-twists, triangulations, and binary search trees (see Example~\ref{exm:1twistsTriangulations} and \fref{fig:1twistsTriangulations}):
\begin{itemize}
\item the pipe insertion of~$q$ in the (relabeled) $1$-twist~$\twist$,
\item the triangle insertion of~$q$ in the (relabeled) triangulation~$\twist\duality$,
\item the node insertion~$q$ in the (relabeled) binary search tree~$\twist\contact$.
\end{itemize}
\end{example}

\begin{example}[Pipe insertion in $k$-twists and $k$-crossing inflation in $k$-triangulations]
\label{exm:(k+1)crossinginflation}
As in Example~\ref{exm:bijectionsInsertions}, a pipe insertion in a $k$-twist~$\twist$ corresponds to the inflation into a $k$-star of a $k$-crossing formed by $k$ consecutive $k$-boundary edges in the $k$-triangulation~$\twist\duality$. Similarly, a pipe deletion in a $k$-twist~$\twist$ corresponds to the flattening into a $k$-crossing of a $k$-star incident to $k$ $k$-boundary edges. This $k$-crossing inflation and $k$-star flattening operations on $k$-triangulations were considered in~\cite{Nakamigawa, Jonsson, PilaudSantos-multitriangulations}, even in more generality. The detailed description of these operations is not needed here and can be found in~\cite{PilaudSantos-multitriangulations}.
\end{example}

\subsection{$k$-twist correspondence}
\label{subsec:twistCorrespondence}

We now present a natural surjection from permutations to acyclic $k$-twists. It relies on an insertion operation on pipe dreams similar to the insertion in binary search trees (see Example~\ref{exm:BSTinsertion} for details). It is motivated by the geometry of the normal fan of the corresponding brick polytope (see Section~\ref{sec:geometry} and~\cite{PilaudSantos-brickPolytope}).

We now describe this algorithm. From a permutation~$\tau \eqdef [\tau_1, \dots, \tau_n]$ (written in one-line notation), we construct a $(k,n)$-twist~$\surjectionPermBrick(\tau)$ obtain from the $(k,0)$-twist by successive pipe insertions of the entries~$\tau_n, \dots, \tau_1$ of~$\tau$ read from right to left. Equivalently, starting from the empty triangular shape, we insert the pipes~$\tau_n, \dots, \tau_1$ of the twist such that each new pipe is as northwest as possible in the space left by the pipes already inserted. This procedure is illustrated in \fref{fig:insertionAlgorithm} for the permutation~$31542$ and different values of~$k$.

\hvFloat[floatPos=p, capWidth=h, capPos=r, capAngle=90, objectAngle=90, capVPos=c, objectPos=c]{figure}
{\includegraphics[scale=1.08]{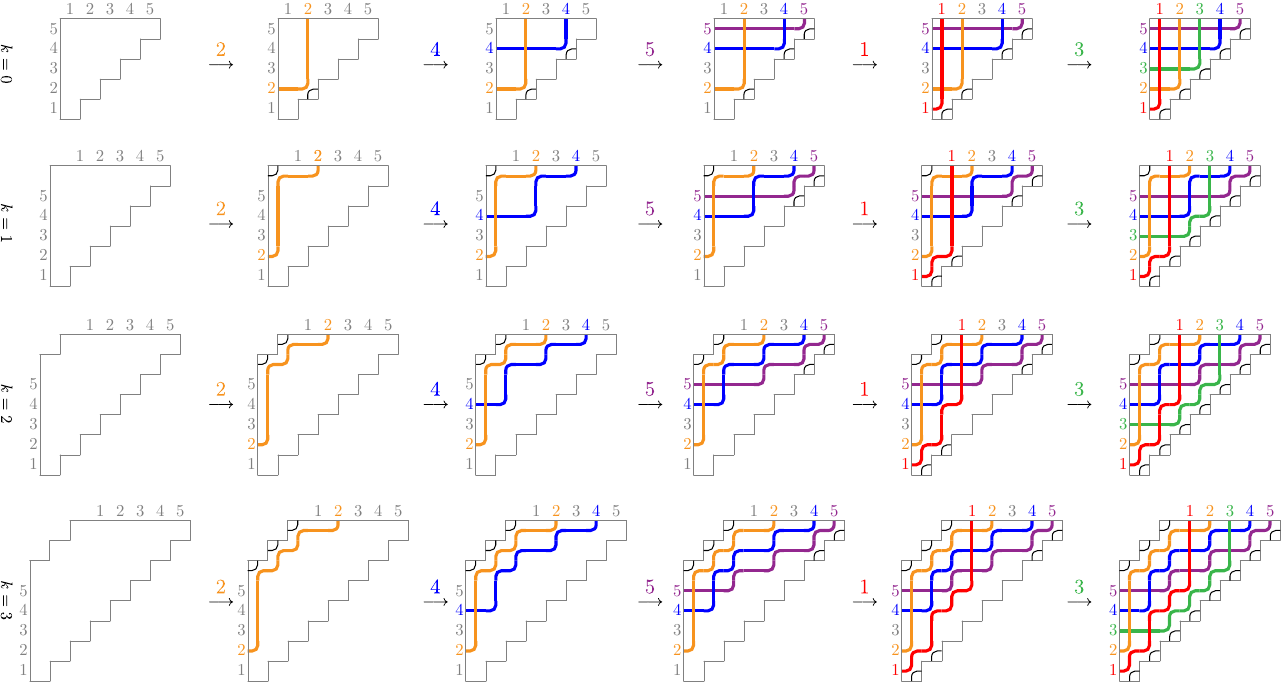}}
{Insertion of the permutation~$31542$ in $k$-twists for~$k = 0,1,2,3$.}
{fig:insertionAlgorithm}

\begin{proposition}
\label{prop:ktwistInsertion}
For any $(k,n)$-twist~$\twist$, the permutations~$\tau \in \fS_n$ such that~$\surjectionPermBrick(\tau) = \twist$ are precisely the linear extensions of the contact graph of~$\twist$. In particular, $\surjectionPermBrick$ is a surjection from the permutations of~$\fS_n$ to the acyclic $(k,n)$-twists.
\end{proposition}

\begin{proof}
We prove the result by induction on~$n$. Consider a permutation~$\tau = [\tau_1,\dots,\tau_n] \in \fS_n$, and let~$\tau' = [\tau_2,\dots,\tau_n]$. By definition, we have~$\surjectionPermBrick(\tau) = \surjectionPermBrick(\tau') \insertion{} \tau_1$. By induction hypothesis, $\tau'$ is a linear extension of the contact graph~$\surjectionPermBrick(\tau')\contact$ and Lemma~\ref{lem:source} ensures that~$\tau_1$ is a source of~$\surjectionPermBrick(\tau)\contact = (\surjectionPermBrick(\tau') \insertion{} \tau_1)\contact$. It follows that~$\tau$ is a linear extension of the contact graph~$\surjectionPermBrick(\tau)\contact$. Conversely, assume that~$\tau$ is a linear extension of the contact graph of a $(k,n)$-twist~$\twist$. Since~$\tau_1$ is a source of~$\twist\contact$, the twist~$\twist \deletion{} \tau_1$ is well-defined, and $\tau'$ is a linear extension of~$(\twist \deletion{} \tau_1)\contact$. By induction hypothesis, it follows that~$\surjectionPermBrick(\tau') = \twist \deletion{} \tau_1$ and thus that~$\surjectionPermBrick(\tau) = \surjectionPermBrick(\tau') \insertion{} \tau_1 = (\twist \deletion{} \tau_1) \insertion{} \tau_1 = \twist$ by Lemma~\ref{lem:insertionDeletion}.
\end{proof}

\begin{remark}[$k$-twist correspondence]
Remembering additionally the insertion order, $\surjectionPermBrick$ defines a bijection between permutations of~$\fS_n$ and \defn{leveled acyclic $(k,n)$-twists}, \ie $(k,n)$-twists whose relevant pipes have been relabeled by~${\lambda: [n] \to [n]}$ so that the label of the \SE-pipe is larger than the label of the \WN-pipe at each elbow. We call this bijection the \defn{$k$-twist correspondence}.
\end{remark}

\begin{example}[$1$-twist correspondence and sylvester correspondence]
\label{exm:BSTinsertion}
According to Example~\ref{exm:bijectionsInsertions}, the contact graph of the $1$-twist~$\surjectionPermBrick[1](\tau)$ is the binary search tree obtained by the successive insertions of the entries of~$\tau$ read from right to left. If we additionally remember the insertion order, we obtain the sylvester correspondence~\cite{HivertNovelliThibon-algebraBinarySearchTrees} between permutations and leveled binary~trees.
\end{example}

\subsection{$k$-twist congruence}
\label{subsec:twistCongruence}

We now characterize the fibers of~$\surjectionPermBrick$ as classes of a congruence~$\equiv^k$ defined by a simple rewriting rule, similar to the sylvester congruence~\cite{HivertNovelliThibon-algebraBinarySearchTrees}.

\begin{definition}
\label{def:ktwistCongruence}
Write the permutations of~$\fS_n$ as words in one-line notation. The \defn{$k$-twist congruence} is the equivalence relation~$\equiv^k$ on~$\fS_n$ defined as the transitive closure of the rewriting~rule
\[
U ac V_1 b_1 V_2 b_2 \cdots V_k b_k W \equiv^k U ca V_1 b_1 V_2 b_2 \cdots V_k b_k W \qquad\text{if } a < b_i < c \text{ for all } i \in [k],
\]
where~$a, b_1, \dots, b_k, c$ are elements of~$[n]$ while~$U, V_1, \dots, V_k, W$ are (possibly empty) words on~$[n]$. We say that~$b_1, \dots, b_k$ are \defn{$k$-twist congruence witnesses} for the exchange of~$a$ and~$c$.
\end{definition}

The congruence classes of~$\equiv^k$ for~$k \in \{1,2\}$ and~$n = 4$ are represented in \fref{fig:fibersTwistCongruence}, where the underlying graph is the Hasse diagram of the (right) weak order (see Section~\ref{subsec:twistClasses}).

\begin{figure}[b]
	\centerline{\includegraphics[width=\textwidth]{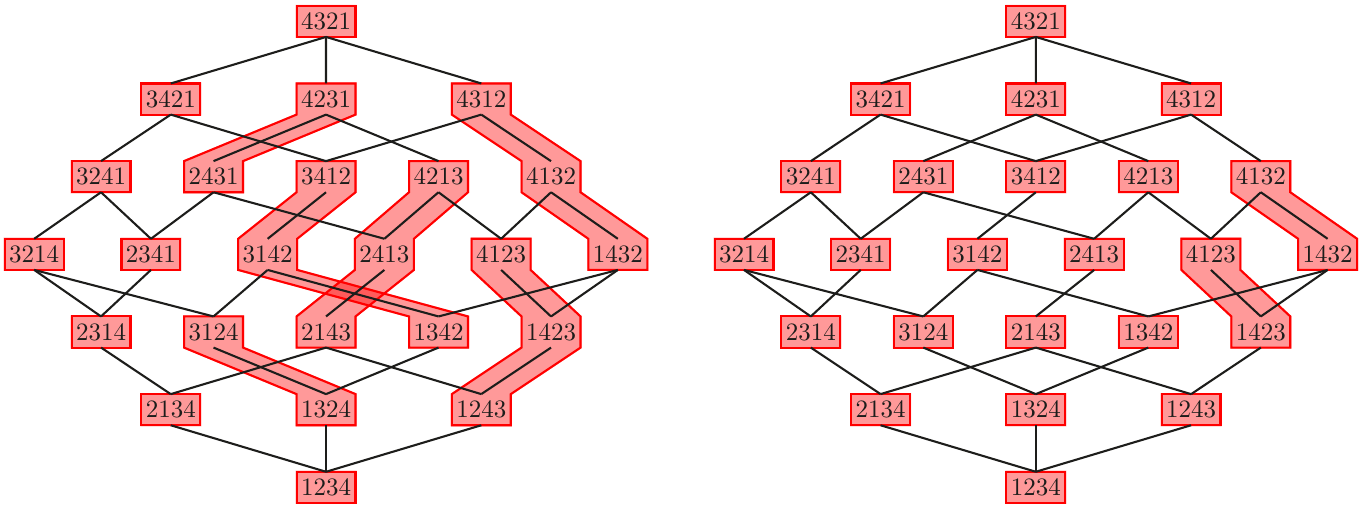}}
	\caption{The $k$-twist congruence classes on~$\fS_4$ for~$k = 1$ (left) and~$k = 2$ (right).}
	\label{fig:fibersTwistCongruence}
\end{figure}

\begin{proposition}
\label{prop:ktwistCongruence}
For any~$\tau, \tau' \in \fS_n$, we have~$\tau \equiv^k \tau' \iff \surjectionPermBrick(\tau) = \surjectionPermBrick(\tau')$. In other words, the fibers of~$\surjectionPermBrick$ are precisely the $k$-twist congruence classes.
\end{proposition}

\begin{proof}
\enlargethispage{.6cm}
From Proposition~\ref{prop:ktwistInsertion}, each fiber of~$\surjectionPermBrick$ gathers the linear extensions of a $k$-twist. Since the set of linear extensions of a poset is connected by simple transpositions, we just need to show that~$\tau \equiv^k \tau' \iff \surjectionPermBrick(\tau) = \surjectionPermBrick(\tau')$ for any two permutations~$\tau = U ac V$ and~$\tau' = U ca V$ of~$\fS_n$ which differ by the inversion of two consecutive values.

Let~$\twist = \surjectionPermBrick(V)$ denote the $k$-twist obtained after the insertion of~$V$. The positions where~$a$ and~$c$ will be inserted in~$\twist$ are separated by the letters~$b$ in~$V$ such that~$a < b < c$. Therefore, if there exists at least~$k$ such letters, the pipes~$a$ and~$c$ are not comparable in~${(\twist \insertion{} c) \insertion{} a = (\twist \insertion{} a) \insertion{} c}$ and we have~$\surjectionPermBrick(\tau) = \surjectionPermBrick(\tau')$. Conversely, if there are strictly less than~$k$ such letters, then~$a$ is below~$c$ in~$(\twist \insertion{} c) \insertion{} a$ while~$a$ is above~$c$ in~$(\twist \insertion{} a) \insertion{} c$, and thus we get~$\surjectionPermBrick(\tau) \ne \surjectionPermBrick(\tau')$.
\end{proof}

\begin{example}[$1$-twist congruence and sylvester congruence]
The $1$-twist congruence coincides with the sylvester congruence defined in~\cite{HivertNovelliThibon-algebraBinarySearchTrees} as the transitive closure of the rewriting rule ${UacVbW \equiv UcaVbW}$ for~$a < b < c$ elements of~$[n]$ while~$U,V,W$ are (possibly empty) words~on~$[n]$.
\end{example}

\begin{remark}[Numerology again]
\label{rem:numerologyAgain}
The number of $k$-twist congruence classes of~$\fS_n$ is the number of acyclic $(k,n)$-twists. We thus obtain that
\[
|\AcyclicTwists(n)| = n! \quad \text{if } n \le k+1
\qquad\text{and}\qquad
|\AcyclicTwists(k+2)| = (k+2)!-k!
\]
since there is no possible rewriting when~$n \le k+1$, and the only non-trivial $k$-twist congruence classes on~$\fS_{k+2}$ are the pairs~$\{1(k+2)(\sigma_1+1)\cdots(\sigma_k+1), (k+2)1(\sigma_1+1)\cdots(\sigma_k+1)\}$ for~$\sigma \in \fS_k$. The other numbers in Table~\ref{table:numbersAcyclicTwists} remain mysterious for us.
\end{remark}

\subsection{$k$-twist congruence classes and WOIPs}
\label{subsec:twistClasses}

We now take a moment to study an essential property of the $k$-twist congruence classes with respect to the weak order on permutations. Remember that the (right) weak order on~$\fS_n$ is defined as the inclusion order of coinversions, where a coinversion of~$\tau \in \fS_n$ is a pair of values~$i, j \in \N$ such that~$i < j$ while~$\tau^{-1}(i) > \tau^{-1}(j)$. For example, the Hasse diagram of the weak order on~$\fS_4$ is the underlying graph of \fref{fig:fibersTwistCongruence}. The following definitions of weak order intervals and weak order interval posets are connected in the Proposition~\ref{prop:WOIP} proved by A.~Bjorner and M.~Wachs~\cite{BjornerWachs}.

\begin{definition}
\label{def:WOIP}
A \defn{weak order interval (WOI)} is a subset~$[\tau, \tau']$ of~$\fS_n$ consisting of all permutations~$\sigma \in \fS_n$ such that~$\tau \le \sigma \le \tau'$ in weak order. A \defn{weak order interval poset (WOIP)} is a poset~$\less$ on~$[n]$ such that for all~$a < b < c$,
\[
a \less c \implies a \less b \text{ or } b \less c 
\qquad\text{and}\qquad
a \more c \implies a \more b \text{ or } b \more c.
\]
\end{definition}

\begin{proposition}[\protect{\cite[Thm.~6.8]{BjornerWachs}}]
\label{prop:WOIP}
The WOIs are the sets of linear extensions of WOIPs:
\begin{enumerate}[(i)]
\item the WOI~$[\tau, \tau']$ is the set of linear extensions of the WOIP obtained as the transitive closure of the relations~$i \less j$ for all non coinversions~$i,j$ of~$\tau$ and~$i \more j$ for all coinversions~$i,j$~of~$\tau'$;
\item the linear extensions of a WOIP~$\less$ is the WOI~$[\tau,\tau']$, where~$\tau$ and~$\tau'$ are respectively the greedy and antigreedy linear extensions of~$\less$.
\end{enumerate}
\end{proposition}

We use this characterization to prove the following statement, fundamental for the paper.

\begin{proposition}
\label{prop:intervals}
The $k$-twist congruence classes are intervals of the weak order.
\end{proposition}

\begin{proof}
By Propositions~\ref{prop:ktwistInsertion} and~\ref{prop:ktwistCongruence}, a $k$-twist congruence class is the insertion fiber of a $k$-twist~$\twist$, and thus the set of linear extensions of the transitive closure~$\contactLess{T}$ of the contact graph~$\twist\contact$. We therefore just need to show that~$\contactLess{T}$ satisfies the WOIP conditions of Definition~\ref{def:WOIP}.

Assume that~$a < b < c$ are such that~$a \contactLess{T} c$. Decompose the triangular shape into three regions: the region~$A$ of all points located south-west of the first elbow of the $b$th pipe of~$\twist$, the region~$B$ of all points located north-west or south-east of an elbow of the $b$th pipe of~$\twist$, and the region~$C$ of all points located north-east of the last elbow of the $b$th pipe of~$\twist$. Since~$a \contactLess{T} c$, there is a path~$\pi$ form the entering point of the $a$th pipe of~$\twist$ to the exiting point of the $c$th pipe of~$\twist$ which travels along the pipes and the elbows of~$\twist$. Since~$a < b < c$, the path~$\pi$ starts in region~$A$ and ends in region~$C$, so that it necessarily passes from region~$A$ to region~$C$. Since the north-east corner of~$A$ is located south-west of the south-west corner of~$C$, this forces an elbow~$\elb$ of~$\pi$ to lie in region~$B$. Lemma~\ref{lem:comparable} then ensures that either~$a \contactLess{T} b$ (if~$\elb$ is above the $b$th pipe), or~$b \contactLess{T} c$ (if~$\elb$ is below the~$b$th~pipe). The proof is similar if~$a \contactMore{T} c$.
\end{proof}

\begin{remark}[Interval property for arbitrary subword complexes]
\label{rem:notIntervalBrickPolytopes}
Surprisingly, Proposition~\ref{prop:intervals} fails for arbitrary subword complexes, even in type~$A$: it is shown in~\cite[Section~5.3]{PilaudStump-brickPolytope} that fibers are always closed by intervals, but not necessarily intervals themself. See \eg \cite[Figure~9]{PilaudStump-brickPolytope}.
\end{remark}

\enlargethispage{.5cm}
Let us sum up the last two sections by the following  bijective statement. 

\begin{corollary}
The following combinatorial objects are in explicit bijection:
\begin{itemize}
\item acyclic $(k,n)$-twists,
\item $k$-twist congruence classes of~$\fS_n$,
\item permutations of~$\fS_n$ avoiding~$1(k+2) \dash (\sigma_1+1) \dash \dots \dash (\sigma_k+1)$ for all~$\sigma \in \fS_k$ (maximums),
\item permutations of~$\fS_n$ avoiding~$(k+2)1 \dash (\sigma_1+1) \dash \dots \dash (\sigma_k+1)$ for all~$\sigma \in \fS_k$ (minimums).
\end{itemize}
\end{corollary}

When~$k = 1$, one easily sees that $13 \dash 2$ avoiding permutations are enumerated by Catalan numbers (as are $31 \dash 2$ avoiding ones). As mentioned in Question~\ref{qu:numerology}, it would be interesting to obtain closed formulas for the enumeration of acyclic $(k,n)$-twists for general~$k$.

\begin{remark}[$13 \dash 2$ versus $1 \dash 3 \dash 2$ avoiding permutations]
It is easy to see that a permutation avoids $13 \dash 2$ if and only if it avoids $1 \dash 3 \dash 2$. This property fails for larger values of~$k$. For example, the permutation~$13524$ avoids $14 \dash 2 \dash 3$ but not~$1 \dash 4 \dash 2 \dash 3$. Here, we deal with permutations avoiding the pattern~$1(k+2) \dash (\sigma_1+1) \dash \dots \dash (\sigma_k+1)$, where~$1$ and~$(k+2)$ are~consecutive.
\end{remark}

\subsection{Lattice congruences of the weak order}
\label{subsec:latticeCongruences}

Based on the result of the previous section, we now show that the $k$-twist congruence~$\equiv^k$ is a lattice congruence of the weak order. This follows as well from the work of N.~Reading~\cite{Reading-HopfAlgebras}, but we prefer to remind the short proof as we will use similar ideas in Part~\ref{part:extensions}. We first remind the reader with the definition of these congruences and refer to~\cite{Reading-LatticeCongruences, Reading-CambrianLattices} for further details.

\begin{definition}
An \defn{order congruence} is an equivalence relation~$\equiv$ on a poset~$P$ such that:
\begin{enumerate}[(i)]
\item Every equivalence class under~$\equiv$ is an interval of~$P$.
\item The projection~$\projDown : P \to P$ (resp.~$\projUp : P \to P$), which maps an element of~$P$ to the minimal (resp.~maximal) element of its equivalence class, is order preserving.
\end{enumerate}
The \defn{quotient}~$P/{\equiv}$ is a poset on the equivalence classes of~$\equiv$, where the order relation is defined by~$X \le Y$ in~$P/{\equiv}$ iff there exists representatives~$x \in X$ and~$y \in Y$ such that~$x \le y$ in~$P$. The quotient~$P/{\equiv}$ is isomorphic to the subposet of~$P$ induced by~$\projDown(P)$ (or equivalently by~$\projUp(P)$).

If moreover~$P$ is a finite lattice, then~$\equiv$ is a automatically \defn{lattice congruence}, meaning that it is compatible with meets and joins: for any~$x \equiv x'$ and~$y \equiv y'$, we have~$x \meet y \equiv x' \meet y'$ and~$x \join y \equiv x' \join y'$. The poset quotient~$P/{\equiv}$ then inherits a lattice structure where the meet~$X \meet Y$ (resp.~the join~$X \join Y$) of two congruence classes~$X$ and~$Y$ is the congruence class of~$x \meet y$ (resp.~of~$x \join y$) for arbitrary representatives~$x \in X$ and~$y \in Y$.
\end{definition}

\begin{theorem}
\label{theo:ktwistLatticeCongruence}
The $k$-twist congruence~$\equiv^k$ is a lattice congruence of the weak order.
\end{theorem}

\begin{proof}
\enlargethispage{.4cm}
We already observed in Proposition~\ref{prop:intervals} that the $k$-twist congruence classes are intervals of the weak order. We therefore just need to show that~$\tau < \tau'$ implies~$\min(C) < \min(C')$ and $\max(C) < \max(C')$ for any two permutations~$\tau, \tau' \in \fS_n$ from distinct $k$-twist congruence classes~$C,C'$. We only prove the result for the maximums, the proof for the minimums being similar. We first observe that we can assume that~$\tau' = \tau s_i$ for some simple transposition~${s_i = (i \;\, i+1)}$ which is not a descent of~$\tau$. The proof then works by induction on the weak order distance between~$\tau$ and~$\max(C)$. If~$\tau = \max(C)$, the result is immediate as~$\max(C) = \tau < \tau' \le \max(C')$. Otherwise, consider a permutation~$\sigma$ in~$C$ covering~$\tau$ in weak order. We write~$\sigma = \tau s_j$, and $s_j$ is not a descent of~$\tau$. We now distinguish four cases, according to the relative positions of~$i$ and~$j$:
\begin{enumerate}[(1)]
\item If~$j > i+1$, we have~$\tau = UabVcdW$, $\tau' = UbaVcdW$ and~$\sigma = UabVdcW$ for some letters~$a < b$, $c < d$ and words~$U,V,W$. Define~$\sigma' \eqdef \tau s_i s_j = \tau' s_j = \sigma s_i = UbaVdcW$. Since~$\tau \equiv^k \sigma$, there exist $k$-twist congruence witnesses~$c < w_1, \dots, w_k < d$ which all appear in~$W$. These letters are also $k$-twist congruence witnesses for the equivalence~$\tau' \equiv^k \sigma'$. Moreover, $\sigma < \sigma'$ since~$c < d$.
\item If~$j = i+1$, we have~$\tau = UabcW$, $\tau' = UbacW$ and~$\sigma = UacbW$ for some letters~$a < b < c$ and words~$U,W$. Define~$\sigma' \eqdef \tau s_i s_{i+1} s_i = \tau' s_{i+1} s_i = \sigma s_i s_{i+1} = UcbaW$. Since~$\tau \equiv^k \sigma$, there exist $k$-twist congruence witnesses~$b < w_1, \dots, w_k < c$ which all appear in~$W$. Since~$a < b$, these letters are thus also $k$-congruence witnesses for the equivalences~$\tau' \equiv^k \tau' s_{i+1} \equiv^k \tau' s_{i+1} s_i = \sigma'$. Moreover, since~$a < b < c$, we have~$\sigma < \sigma'$.
\item If~$j = i-1$, we proceed similarly as in Situation~(2).
\item If~$j < i-1$, we proceed similarly as in Situation~(1).
\end{enumerate}
In all cases, we have~$\tau \equiv^k \sigma < \sigma' \equiv^k \tau'$. Since~$\sigma$ is closer to~$\max(C)$ than~$\tau$, we obtain that ${\max(C) < \max(C')}$ by induction hypothesis.
\end{proof}

\subsection{Increasing flip lattice}
\label{subsec:flips}

We now recall the notion of flips in pipe dreams and study the graph of increasing flips in acyclic $k$-twists. A detailed treatment of increasing flips in finite type subword complexes can be found in~\cite{PilaudStump-ELlabelings}. However, we are not aware that the restriction of this graph to acyclic pipe dreams was ever studied, besides the undirected version in connection to brick polytopes~\cite{PilaudSantos-brickPolytope, PilaudStump-brickPolytope}.

\begin{definition}
An \defn{elbow flip} (or just \defn{flip}) in a $k$-twist is the exchange of an elbow \elbow{} between two relevant pipes~$\pipe, \pipe'$ with the unique crossing \cross{} between~$\pipe$ and~$\pipe'$. The flip is \defn{increasing} if the initial elbow is located (largely) south-west of the final elbow.
\end{definition}

\begin{remark}[Flips in $k$-triangulations]
As already mentioned in Example~\ref{exm:1twistsTriangulations} and Theorem~\ref{theo:ktwistsktriangulations}, an elbow flip in a $k$-twist~$\twist$ corresponds to a diagonal flip in the dual $k$-triangulation~$\twist\duality$. The $(m+2k)$-gon was chosen so that a flip is increasing if and only if the slope increases between the initial and the final diagonals. See Figures~\ref{fig:1twistsTriangulations} and~\ref{fig:ktwistsktriangulations} for illustrations.
\end{remark}

We are interested in the graph of increasing flips, restricted to the acyclic $(k,n)$-twists. See \fref{fig:increasingFlipLattices} for an illustration. Note that although the graph of flips is regular, its restriction to acyclic twists is not anymore regular in general: the first example appears for~$k = 2$ and~$n = 5$. It is known (see \eg via subword complexes in~\cite{PilaudStump-ELlabelings} or via multitriangulations in~\cite{PilaudSantos-multitriangulations}) that this graph is acyclic and it has a unique source (resp.~sink) given by the $(k,n)$-twist where all relevant elbows are in the first~$k$ columns (resp.~last~$k$ rows) while all crosses are on the last~$n$ columns (resp.~first $n$ rows).

We call \defn{increasing flip order} the transitive closure of the increasing flip graph on acyclic $k$-twists. Be aware that it is strictly contained in the restriction to acyclic $k$-twists of the transitive closure of the increasing flip graph on all $k$-twists: namely, there are pairs of acyclic $k$-twists so that any path of increasing flips between them passes through a cyclic $k$-twist.

\hvFloat[floatPos=p, capWidth=h, capPos=r, capAngle=90, objectAngle=90, capVPos=c, objectPos=c]{figure}
{\includegraphics[scale=.82]{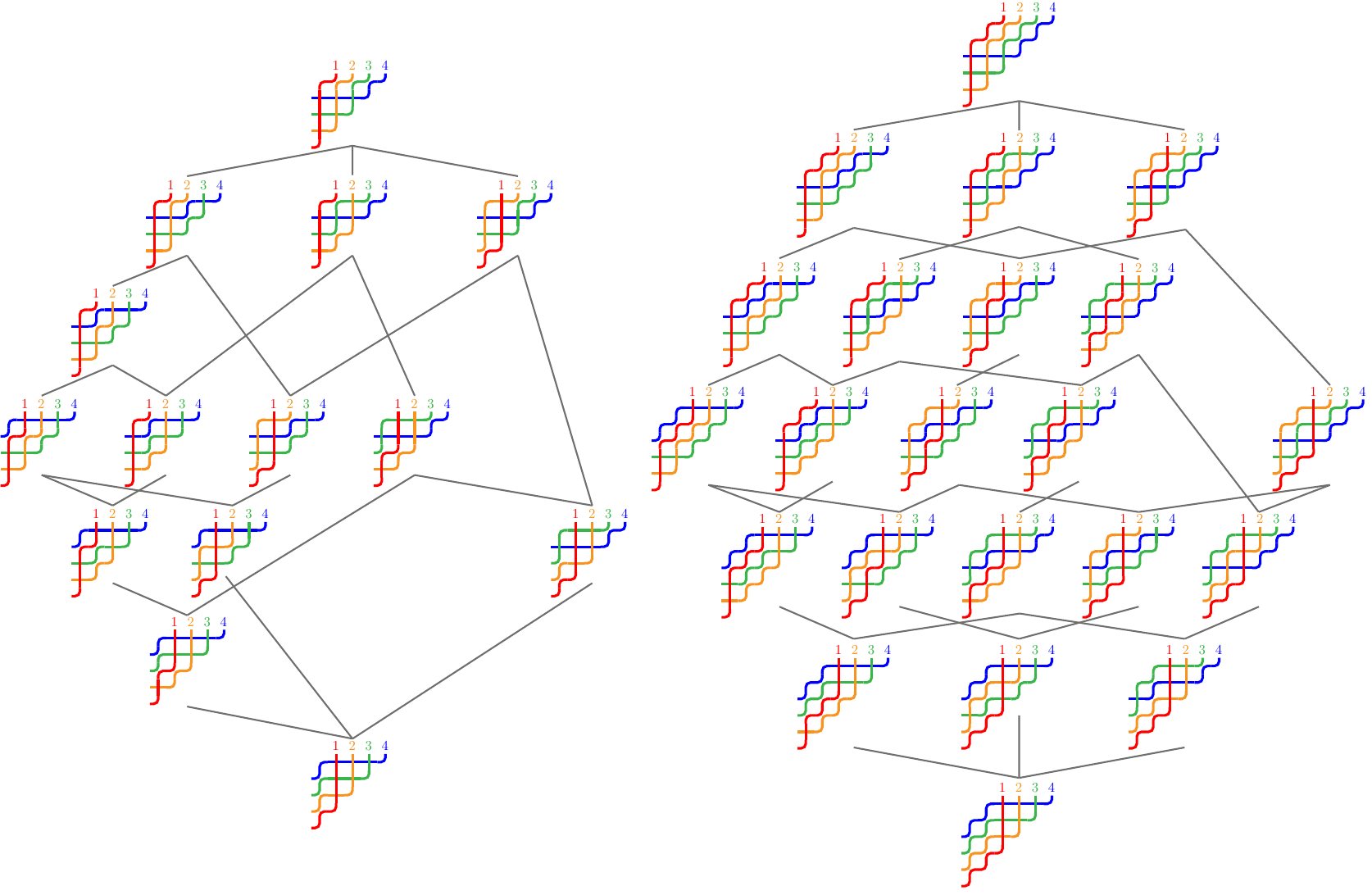}}
{The increasing flip lattice on $(k,4)$-twists for~$k = 1$ (left) and~$k = 2$ (right).}
{fig:increasingFlipLattices}

\begin{proposition}
\label{prop:latticeQuotient}
The following posets are all isomorphic:
\begin{itemize}
\item the increasing flip order on acyclic $k$-twists,
\item the quotient lattice of the weak order by the $k$-twist congruence~$\equiv^k$,
\item the subposet of the weak order induced by the permutations of~$\fS_n$ avoiding the pattern ${1(k+2) \dash (\sigma_1 + 1) \dash \cdots \dash (\sigma_k + 1)}$ for all~$\sigma \in \fS_k$,
\item the subposet of the weak order induced by the permutations of~$\fS_n$ avoiding the pattern ${(k+2)1 \dash (\sigma_1 + 1) \dash \cdots \dash (\sigma_k + 1)}$ for all~$\sigma \in \fS_k$.
\end{itemize}
\end{proposition}

\begin{proof}
Consider two distinct $k$-twists~$\twist, \twist'$ and their $k$-twist congruence~classes~$C, C'$. If there exist representatives~$\tau = UijV \in C$ and~$\tau' = UjiV \in C'$ which are adjacent in weak order, then~$\twist = \surjectionPermBrick(\tau)$ and~$\twist' = \surjectionPermBrick(\tau')$ differ by the flip of the first/last elbow between the $i$th and $j$th pipes, by definition of the map~$\surjectionPermBrick$. Conversely, if~$\twist$ and~$\twist'$ differ by the flip of the first/last elbow between the $i$th and $j$th pipes, then $i$ and~$j$ are connected in the contact graph~$\twist\contact$, so that there exists a linear extension~$\tau = UijV$ of~$\twist\contact$ where~$i$ and~$j$ are consecutive. Let~$\tau' = UjiV$ be the permutation obtained by the switch of~$i$ and~$j$ in~$\tau$. By definition of the map~$\surjectionPermBrick$, the twist~$\surjectionPermBrick(\tau')$ is obtained by flipping the first elbow between the $i$th and $j$th pipes in~$\surjectionPermBrick(\tau)$. The representatives~$\tau \in C$ and~$\tau' \in C'$ are thus adjacent in weak order. This proves that~$\surjectionPermBrick$~induces an isomorphism from the quotient lattice of the weak order by the $k$-twist congruence to the increasing flip order on acyclic $k$-twists. In turn, since the $k$-twist congruence is a lattice congruence, this quotient lattice is isomorphic to the subposet of the weak order induced by the minimal (resp.~maximal) elements of the classes. See~\cite{Reading-LatticeCongruences, Reading-CambrianLattices} for further details on quotient lattices. 
\end{proof}

\begin{example}[Tamari lattice]
When~$k = 1$, the increasing flip lattice is the classical Tamari lattice~\cite{TamariFestschrift}.
See \fref{fig:increasingFlipLattices}\,(left).
\end{example}

\begin{remark}[Increasing flip orders for arbitrary subword complexes]
\label{rem:notLatticesBrickPolytopes}
Again, let us warn the reader that Proposition~\ref{prop:latticeQuotient} does not hold for arbitrary type~$A$ subword complexes. First, as already mentioned in Remark~\ref{rem:notIntervalBrickPolytopes}, the fibers of the surjection map are not always intervals~\cite[Figure~9]{PilaudStump-brickPolytope}. But in fact, there are root-independent subword complexes for which the increasing flip order on facets is not even a lattice, see \eg \cite[Example~5.12]{PilaudStump-brickPolytope}.
\end{remark}

\subsection{$k$-recoil schemes}
\label{subsec:krecoils}

\enlargethispage{.3cm}
To prepare the definition of the $k$-canopy of an acyclic $k$-twist, we now recall the notion of $k$-recoil schemes of permutations, which was already defined by J.-C.~Novelli, C.~Reutenauer and J.-Y.~Thibon in~\cite{NovelliReutenauerThibon}. We use a description in terms of acyclic orientations of a certain graph as it is closer to the description of the vertices of the zonotope that we will use later in Section~\ref{sec:geometry}. For the convenience of the reader, we provide proof sketchs in terms of acyclic orientations of the results of~\cite{NovelliReutenauerThibon} needed later in this paper.

A \defn{recoil} in a permutation~$\tau \in \fS_n$ is a position~$i \in [n-1]$ such that~$\tau^{-1}(i) > \tau^{-1}(i+1)$ (in other words, it is a descent of the inverse of~$\tau$). The \defn{recoil scheme} of~$\tau \in \fS_n$ is the sign vector~$\surjectionPermZono[](\tau) \in \{-,+\}^{n-1}$ defined by~$\surjectionPermZono[](\tau)_i = -$ if~$i$ is a recoil of~$\tau$ and~$\surjectionPermZono[](\tau)_i = +$ otherwise.

To extend this definition to general~$k$, we consider the graph~$\Gkn$ with vertex set~$[n]$ and edge set~$\set{\{i,j\} \in [n]^2}{i < j \le i+k}$. For example, when~$k = 1$, the graph~$G^1(n)$ is just the $n$-path. We denote by~$\AcyclicOrientations(n)$ the set of acyclic orientations of~$\Gkn$ (\ie with no oriented~cycle).

\begin{proposition}[\protect{\cite[Prop.~2.1]{NovelliReutenauerThibon}}]
The number of acyclic orientations of~$\Gkn$ is
\[
|\AcyclicOrientations(n)| = \begin{cases} n! & \text{if } n \le k, \\ k! \, (k+1)^{n-k} & \text{if } n \ge k. \end{cases}
\]
\end{proposition}

\begin{proof}
When~$n \le k$, the graph~$\Gkn$ is complete, so an acyclic orientation of~$\Gkn$ is a permutation of~$[n]$. The case~$n \ge k$ then follows by induction on~$n$. Indeed, an acyclic orientation on~$\Gkn$ corresponds to an acyclic orientation on~$G^k(n-1)$ together with a choice among the $k+1$ possibilities of the position of~$n$ with respect to~$\{n-k, \dots, n-1\}$.
\end{proof}

\enlargethispage{.3cm}
Table~\ref{table:numbersAcyclicOrientations} gathers the numbers of acyclic orientations of~$G^k(n)$ for~$k < n \le 10$.

\begin{table}[h]
  \[
  	\begin{array}{l|rrrrrrrrrr}
	\raisebox{-.1cm}{$k$} \backslash \, \raisebox{.1cm}{$n$}
	  & 1 & 2 & 3 & 4 & 5 & 6 & 7 & 8 & 9 & 10 \\[.1cm]
	\hline
	0 & 1 & 1 & 1 &  1 &   1 &   1 &    1 &     1 &      1 &       1 \\
	1 & . & 2 & 4 &  8 &  16 &  32 &   64 &   128 &    256 &     512 \\
	2 & . & . & 6 & 18 &  54 & 162 &  486 &  1458 &   4374 &   13122 \\
	3 & . & . & . & 24 &  96 & 384 & 1536 &  6144 &  24576 &   98304 \\
	4 & . & . & . &  . & 120 & 600 & 3000 & 15000 &  75000 &  375000 \\
	5 & . & . & . &  . &   . & 720 & 4320 & 25920 & 155520 &  933120 \\
	6 & . & . & . &  . &   . &   . & 5040 & 35280 & 246960 & 1728720 \\
	7 & . & . & . &  . &   . &   . &    . & 40320 & 322560 & 2580480 \\
	8 & . & . & . &  . &   . &   . &    . &     . & 362880 & 3265920 \\
	9 & . & . & . &  . &   . &   . &    . &     . &      . & 3628800
	\end{array}
  \]
  \caption{The numbers~$|\AcyclicOrientations(n)| = k!\,(k+1)^{n-k}$ of acyclic orientations of the graph~$\Gkn$ for~$k < n \le 10$. Dots indicate that the value remains constant (equal to~$n!$) in the column.}
  \label{table:numbersAcyclicOrientations}
  \vspace{-.4cm}
\end{table}

We use these acyclic orientations to define the $k$-recoil scheme of a permutation and the corresponding $k$-recoil congruence.

\begin{definition}
The \defn{$k$-recoil scheme} of a permutation~$\tau \in \fS_n$ is the orientation~$\surjectionPermZono(\tau) \in \AcyclicOrientations(n)$ with an edge~$i \to j$ for all~$i,j \in [n]$ such that~$|i-j| \le k$ and~$\tau^{-1}(i) < \tau^{-1}(j)$. We call \defn{$k$-recoil map} the map~$\surjectionPermZono: \fS_n \to \AcyclicOrientations(n)$.
\end{definition}

The following statement is immediate and left to the reader.

\begin{proposition}
\label{prop:krecoilsFibers}
For~$\orientation \in \AcyclicOrientations(n)$, the fiber of~$\orientation$ by the $k$-recoil map is the set of linear extensions of the transitive closure of~$\orientation$.
\end{proposition}

\begin{definition}
The \defn{$k$-recoil congruence}~$\approx^k$ on~$\fS_n$ is the transitive closure of the rewriting rule
\[
UijV \approx^k UjiV \qquad\text{if } i + k < j,
\]
where~$i,j$ are elements of~$[n]$ while~$U,V$ are (possibly empty) words on~$[n]$.
\end{definition}

The congruence classes of~$\approx^k$ for~$k \in \{1,2\}$ and~$n = 4$ are represented in \fref{fig:fibersRecoilCongruence}, where the underlying graph is the Hasse diagram of the (right) weak order (see Section~\ref{subsec:twistClasses}).

\begin{figure}[t]
	\centerline{\includegraphics[width=\textwidth]{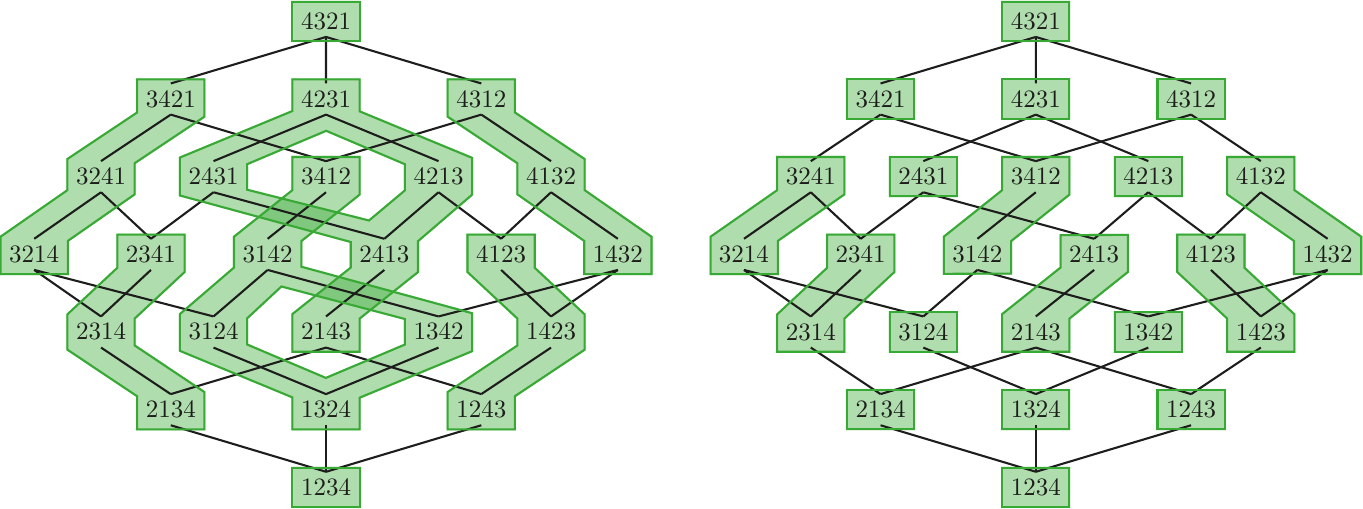}}
	\caption{The $k$-recoil congruence classes on~$\fS_4$ for~$k = 1$ (left) and~$k = 2$ (right).}
	\label{fig:fibersRecoilCongruence}
\end{figure}

\begin{proposition}[\protect{\cite[Prop.~2.2]{NovelliReutenauerThibon}}]
For any~$\tau, \tau' \in \fS_n$, we have~${\tau \approx^k \tau' \iff \surjectionPermZono(\tau) = \surjectionPermZono(\tau')}$. In other words, the fibers of~$\surjectionPermZono$ are precisely the $k$-recoil congruence classes.
\end{proposition}

\begin{proof}
By Proposition~\ref{prop:krecoilsFibers}, each fiber of~$\surjectionPermZono$ gathers the linear extensions of the transitive closure of an acyclic orientation~$\orientation$ of~$\Gkn$. Since the set of linear extensions of a poset is connected by simple transpositions, if suffices to show that~$\tau \approx^k \tau' \iff \surjectionPermZono(\tau) = \surjectionPermZono(\tau')$ for any two permutations~$\tau = U ij V$ and~$\tau' = U ji V$ of~$\fS_n$ which differ by the inversion of two consecutive~values~$i < j$. Indeed, the orientations~$\surjectionPermZono(\tau)$ and~$\surjectionPermZono(\tau')$ differ if and only if $\Gkn$ has an edge between~$i$ and~$j$, \ie if and only if~$i + k \ge j$. 
\end{proof}

\begin{proposition}[\protect{\cite[Prop.~2.4]{NovelliReutenauerThibon}}]
The $k$-recoil congruence classes are intervals of the weak order.
\end{proposition}

\begin{proof}
Immediate from the characterization of WOIs given in Definition~\ref{def:WOIP} and Proposition~\ref{prop:WOIP}.
\end{proof}

\begin{definition}
A \defn{direction flip} (or just \defn{flip}) in an acyclic orientation is the switch of the direction of an edge of~$\Gkn$. The flip is \defn{increasing} if the initial direction was increasing. Define the \defn{increasing flip order} on~$\AcyclicOrientations(n)$ to be the transitive closure of the increasing flip graph on~$\AcyclicOrientations(n)$.
\end{definition}

\begin{theorem}
The $k$-recoil congruence~$\approx^k$ is a lattice congruence of the weak order. The $k$-recoil map~$\surjectionPermZono$ defines an isomorphism from the quotient lattice of the weak order by the $k$-recoil congruence~$\approx^k$ to the increasing flip lattice on the acyclic orientations of~$\Gkn$.
\end{theorem}

\begin{proof}
Similar to that of Theorem~\ref{theo:ktwistLatticeCongruence} and Proposition~\ref{prop:latticeQuotient}.
\end{proof}

\subsection{$k$-canopy schemes}
\label{subsec:canopy}

We recall that the \defn{canopy} of a binary tree~$\tree$ with~$n$ nodes is the sign vector~$\surjectionBrickZono[](\tree) \in \{-,+\}^{n-1}$ defined by~$\surjectionBrickZono[](\tree)_i = -$ if the node~$i$ of~$\tree$ is above the node~$i+1$ of~$\tree$ and~$\surjectionBrickZono[](\tree)_i = +$ otherwise. This map was already used by J.-L.~Loday in~\cite{LodayRonco, Loday}, but the name ``canopy'' was coined by X.~Viennot~\cite{Viennot}. The binary search tree insertion map and the canopy map factorize the recoil map: $\surjectionBrickZono[] \circ \surjectionPermBrick[] = \surjectionPermZono[]$. This combinatorial fact can also be understood on the geometry of the normal fans of the permutahedron, the associahedron and the cube, see Section~\ref{subsec:normalFans}.

We now define an equivalent of the canopy map for general~$k$. To ensure that Definition~\ref{def:canopy} is valid, we need the following simple observation on comparisons of closed pipes in a $k$-twist.

\begin{lemma}
\label{lem:canopy}
If~$|i-j| \le k$, the $i$th and $j$th pipes in an acyclic $k$-twist~$\twist$ are comparable~for~$\contactLess{T}$.
\end{lemma}

\begin{proof}
By Corollary~\ref{coro:comparable}\,\eqref{item:comparable}, if~$i < j$ and the $i$th and $j$th pipes are incomparable, then the last vertical step of~$i$ crosses the first horizontal step of~$j$. Since each pipe has~$k$ horizontal steps by Lemma~\ref{lem:pipeProperties}\,\eqref{item:numberElbows}, it ensures that~$j > i+k$.
\end{proof}

\begin{definition}
\label{def:canopy}
The \defn{$k$-canopy scheme} of a $(k,n)$-twist~$\twist$ is the orientation~$\surjectionBrickZono(\twist) \in \AcyclicOrientations(n)$ with an edge~$i \to j$ for all~$i,j \in [n]$ such that~$|i-j| \le k$ and~$i \contactLess{\twist} j$. It indeed defines an acyclic orientation of~$\Gkn$ by Lemma~\ref{lem:canopy}. We call \defn{$k$-canopy} the map~$\surjectionBrickZono: \AcyclicTwists(n) \to \AcyclicOrientations(n)$.
\end{definition}

\begin{proposition}
\label{prop:latticeHomomorphisms}
The maps~$\surjectionPermBrick$, $\surjectionBrickZono$, and~$\surjectionPermZono$ define the following commutative diagram of lattice homomorphisms:
\[
\begin{tikzpicture}
  \matrix (m) [matrix of math nodes,row sep=1.2em,column sep=5em,minimum width=2em]
  {
     \fS_n  	&						& \AcyclicOrientations(n)	\\
				& \AcyclicTwists(n) 	&							\\
  };
  \path[->>]
    (m-1-1) edge node [above] {$\surjectionPermZono$} (m-1-3)
                 edge node [below] {$\surjectionPermBrick\quad$} (m-2-2.west)
    (m-2-2.east) edge node [below] {$\surjectionBrickZono$} (m-1-3);
\end{tikzpicture}
\]
\end{proposition}

\begin{proof}
Consider a permutation~$\tau$ and let~$i,j \in [n]$ with~$|i-j| < k$ such that~$\tau^{-1}(i) < \tau^{-1}(j)$. By definition, there is an arc~$i \to j$ in~$\surjectionPermZono(\tau)$. Moreover, the $i$th pipe is inserted after the $j$th pipe in~$\surjectionPermBrick(\tau)$, so that~$i \contactLess{\surjectionPermBrick(\tau)} j$ and there is also an arc~$i \to j$ in~$\surjectionBrickZono \circ \surjectionPermBrick(\tau)$.
\end{proof}

The fibers of these maps for~$k \in \{1, 2\}$ and~$n = 4$ appear separately in Figures~\ref{fig:fibersTwistCongruence} and~\ref{fig:fibersRecoilCongruence}, and together in \fref{fig:latticeHomomorphisms}. See also Section~\ref{subsec:normalFans} for a geometric interpretation of Proposition~\ref{prop:latticeHomomorphisms}.

\begin{figure}[t]
	\centerline{\includegraphics[width=\textwidth]{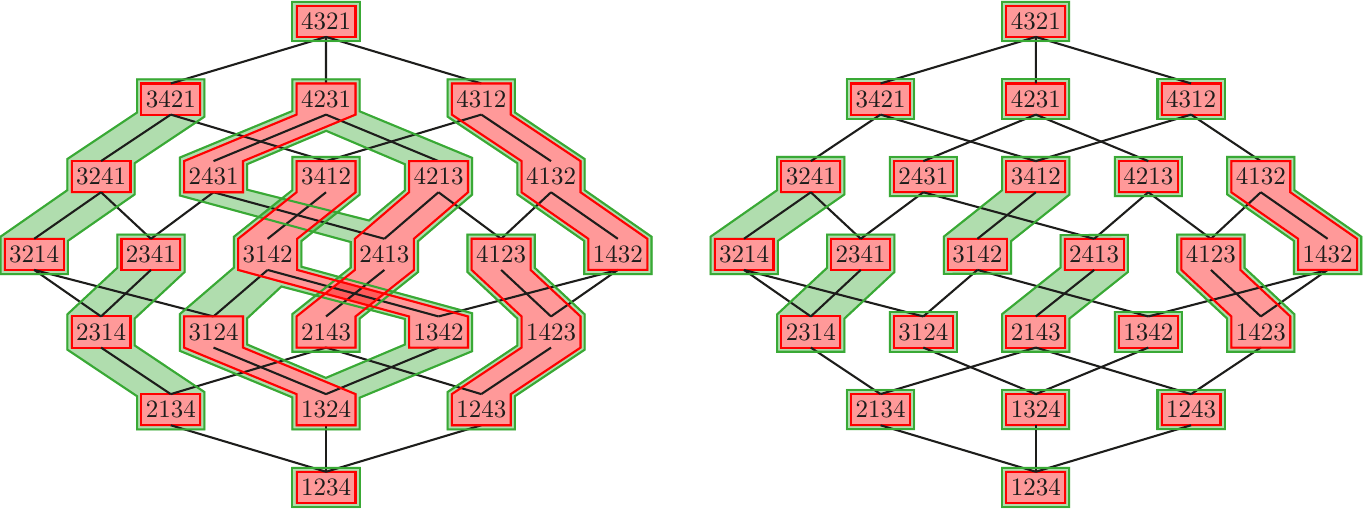}}
	\caption{The fibers of the maps~$\surjectionPermBrick$ ($k$-twist congruence classes, red) and~$\surjectionPermZono$ ($k$-recoils congruence classes, green) on the weak order on~$\fS_4$ for~$k = 1$ (left) and~$k = 2$ (right). See also Figures~\ref{fig:fibersTwistCongruence} and~\ref{fig:fibersRecoilCongruence} where these fibers appear separately.}
	\label{fig:latticeHomomorphisms}
\end{figure}

\begin{remark}[Combinatorial inclusions]
\label{rem:combinatorialInclusions}
When $k > \ell$, the $k$-twist congruence~$\equiv^k$ refines the $\ell$-twist congruence~$\equiv^\ell$ (meaning that~$\tau \equiv^k \tau'$ implies~$\tau \equiv^\ell \tau'$) and the $k$-recoil congruence~$\approx^k$ refines the $\ell$-recoil congruence~$\approx^\ell$. We can thus define surjective restriction maps~${\restrictionBrick : \AcyclicTwists[k](n) \to \AcyclicTwists[\ell](n)}$ and~${\restrictionZono : \AcyclicOrientations[k](n) \to \AcyclicOrientations[\ell](n)}$ by:
\begin{itemize}
\item for an acyclic $k$-twist~$\twist$, the $\ell$-twist~$\restrictionBrick(\twist)$ is obtained by insertion of any linear extension of~$\twist\contact$ (it is independent of the choice of this linear extension),
\item for an acyclic orientation~$\theta \in \AcyclicOrientations[k](n)$, the sign vector~$\restrictionZono(\theta) \in \AcyclicOrientations[\ell](n)$ is obtained by restriction of~$\theta$ to the edges of~$G^\ell(n)$.
\end{itemize}
We therefore obtain the following commutative diagram of lattice homomorphisms:
\[
\begin{tikzpicture}
  \matrix (m) [matrix of math nodes,row sep=3em,column sep=3em,minimum width=2em]
  {
	\AcyclicTwists[k](n) 		& 		 & \AcyclicTwists[\ell](n) \\
					 			& \fS_n  & 	\\
	\AcyclicOrientations[k](n) 	&		 & \AcyclicOrientations[\ell](n) \\
  };
  \path[->>]
    (m-2-2) edge node [below left] {$\surjectionPermBrick[k]\!\!\!\!$} (m-1-1)
		    edge node [below right] {$\!\!\!\!\surjectionPermBrick[\ell]$} (m-1-3)
            edge node [above left] {$\surjectionPermZono[k]\!\!\!\!$} (m-3-1)
            edge node [above right] {$\!\!\!\!\surjectionPermZono[\ell]$} (m-3-3)
    (m-1-1) edge node [left] {$\surjectionBrickZono[k]$} (m-3-1)
    (m-1-3) edge node [right] {$\surjectionBrickZono[\ell]$} (m-3-3)
    (m-1-1) edge node [above] {$\restrictionBrick$} (m-1-3)
    (m-3-1) edge node [below] {$\restrictionZono$} (m-3-3);
\end{tikzpicture}
\]
\end{remark}

\section{Geometry of acyclic twists}
\label{sec:geometry}

This section is devoted to the polyhedral geometry of permutations of~$\fS_n$, acyclic twists of~$\AcyclicTwists(n)$, and acyclic orientations of~$\AcyclicOrientations(n)$. It is mainly based on properties of brick polytopes of sorting networks, defined and studied by V.~Pilaud and F.~Santos in~\cite{PilaudSantos-brickPolytope}. This construction has been further extended to root-independent subword complexes on any finite Coxeter group by V.~Pilaud and C.~Stump in~\cite{PilaudStump-brickPolytope}. However, the context of~\cite{PilaudStump-brickPolytope} does not apply here since the present paper deals with a very specific family of sorting networks which are not root-independent. We note that this family was already discussed in~\cite[Section~5]{PilaudSantos-brickPolytope}.

To keep this section short, we skip all proofs of its statements as they follow directly from~\cite{PilaudSantos-brickPolytope}. This section should be seen as a brief geometric motivation for the combinatorial and algebraic construction of this paper. The reader familiar with the geometry of the brick polytope is invited to proceed directly with Section~\ref{sec:algebra}.

\begin{figure}[p]
	\centerline{\includegraphics[scale=.68]{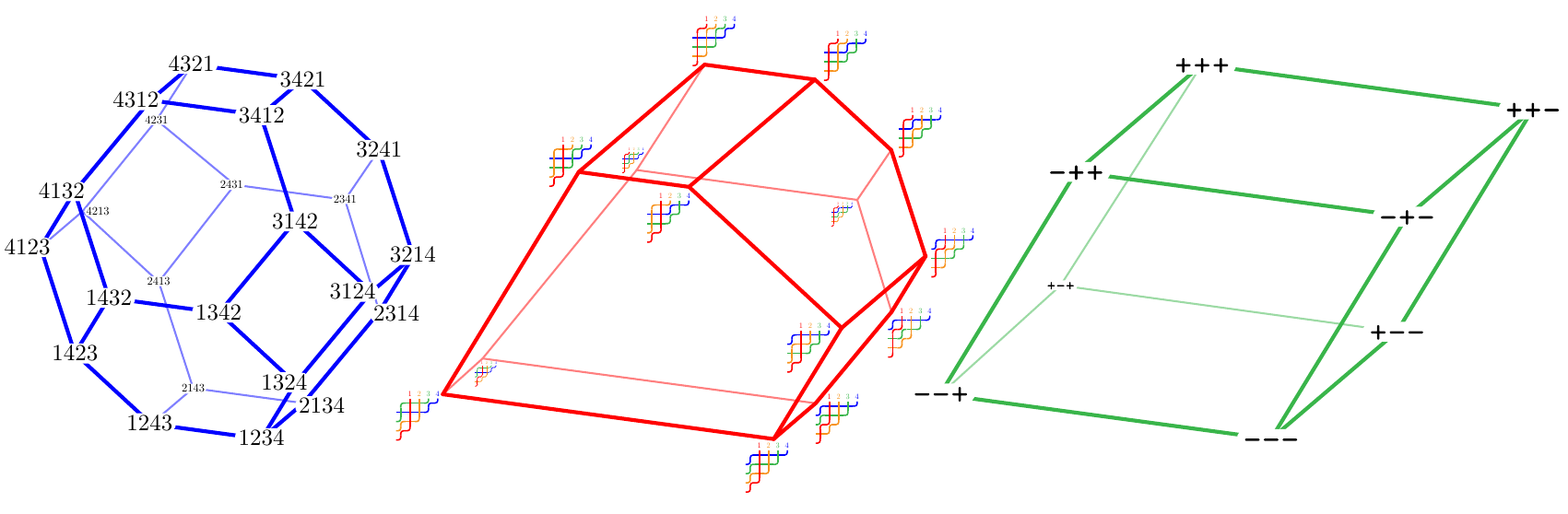}} \medskip
	\centerline{\includegraphics[scale=.68]{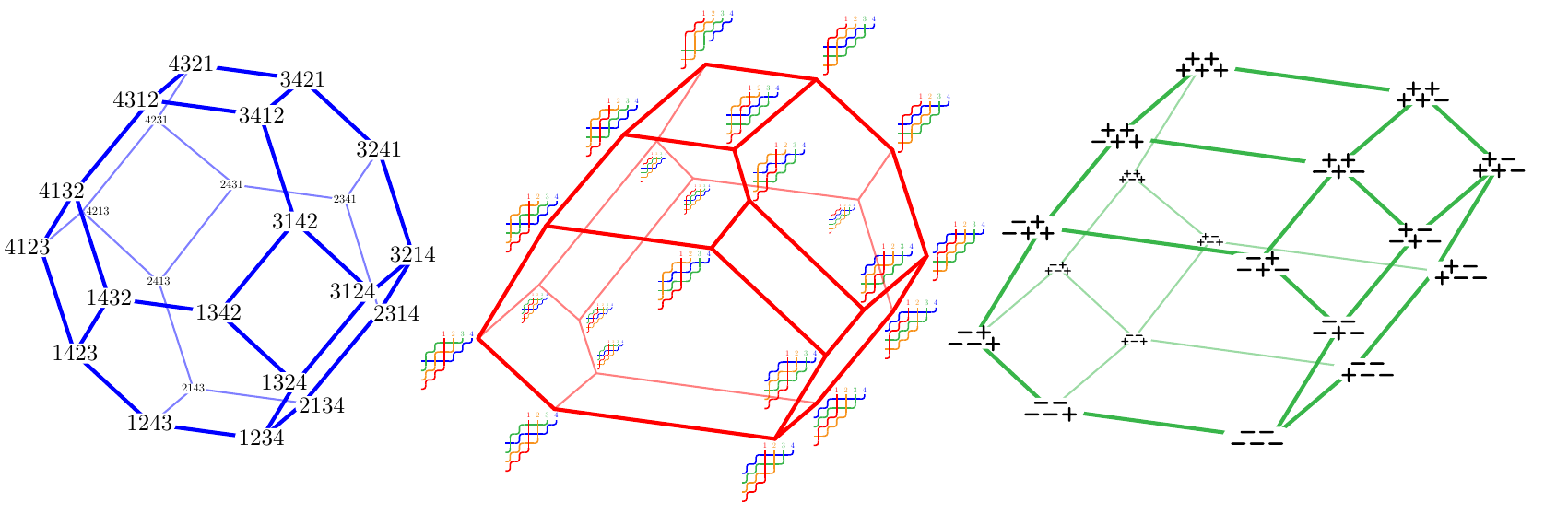}} \medskip
	\centerline{\includegraphics[scale=.68]{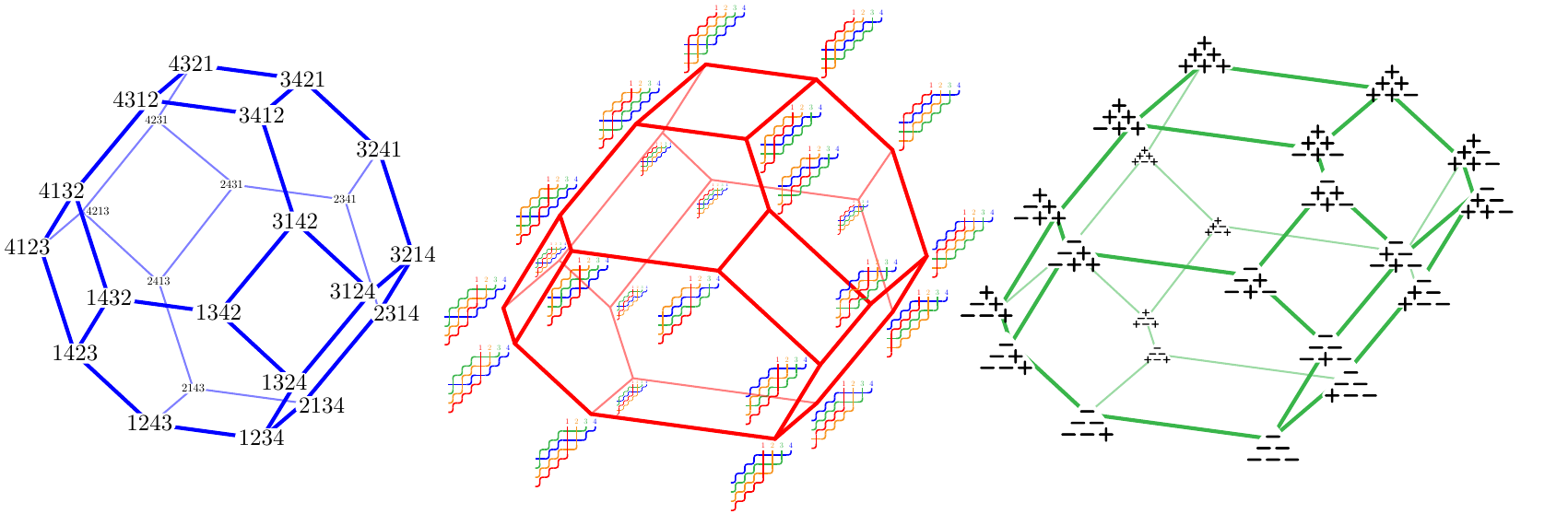}}
	\caption{The permutahedron~$\Perm[k][4]$ (left), the brick polytope~$\Brick[k][4]$ (middle) and the zonotope~$\Zono[k][4]$ (right) for $k = 1$ (top), $k = 2$ (middle) and~$k = 3$ (bottom). For readability, we represent orientations of~$\Gkn$ by pyramids of signs.}
	\label{fig:PermBrickZono}
\end{figure}

\subsection{Permutahedra, brick polytopes, and zonotopes}
\label{subsec:PermBrickZono}

We first recall the definition of three families of polytopes, which are illustrated in \fref{fig:PermBrickZono}. We refer to~\cite[Lectures~0 to~2]{Ziegler} for background on polytopes. We denote by~$(\b{e}_i)_{i \in [n]}$ the canonical basis of~$\R^n$ and let~$\one \eqdef \sum_{i \in [n]} \b{e}_i$.

\para{Permutahedra}
The permutahedron is a classical polytope~\cite[Lecture~0]{Ziegler} whose geometric and combinatorial properties reflect that of the symmetric group~$\fS_n$.

\begin{definition}
The \defn{permutahedron}~$\Perm[][n]$ is the $(n-1)$-dimensional polytope 
\[
\Perm[][n] \; \eqdef \; \conv\set{\b{x}(\tau)}{\tau \in \fS_n} \; = \; \Hyp([n]) \cap \!\!\!\! \bigcap_{\varnothing \ne I \subsetneq [n]} \!\!\!\! \HS(I) \; = \; \one + \!\!\!\! \sum_{1 \le i < j \le n} \!\!\!\! [\b{e}_i, \b{e}_j],
\]
defined equivalently as
\begin{itemize}
\item the convex hull of the points~$\b{x}(\tau) \eqdef [\tau^{-1}(i)]_{i \in [n]} \in \R^n$ for all permutations~$\tau \in \fS_n$,
\item the intersection of the hyperplane~$\Hyp([n]) \eqdef \bigset{\b{x} \in \R^n}{\sum_{i \in [n]} x_i = \binom{n+1}{2}}$ with the half-spaces $\HS(I) \eqdef \bigset{\b{x} \in \R^n}{\sum_{i \in I} x_i \ge \binom{|I|+1}{2}}$ for all proper non-empty subset~${\varnothing \ne I \subsetneq [n]}$,
\item the Minkowski sum of the point~$\one$ with the segments~$[\b{e}_i, \b{e}_j]$ for~$1 \le i < j \le n$.
\end{itemize}
\end{definition}

We consider a dilated and translated copy of the permutahedron~$\Perm[][n]$, which will fit better the other two families of polytopes defined later (see Remark~\ref{rem:geometricInclusions} for a precise statement). Namely, we set
\[
\Perm \; \eqdef \; k\,\Perm[][n] - \frac{k(n+1)}{2} \, \one.
\]
Observe that~$\Perm$ now lies in the hyperplane~$\HH \eqdef \bigset{\b{x} \in \R^n}{\sum_{i \in [n]} x_i = 0}$.

\para{Brick polytopes}
To define the brick polytope, we essentially follow~\cite{Pilaud-these, PilaudSantos-brickPolytope} except that we apply again a translation in direction~$\one$ to obtain a polytope in the hyperplane~$\HH$.

\begin{definition}
We call \defn{bricks} the squares~$[i,i+1] \times [j,j+1]$ of the triangular shape. The \defn{brick area} of a pipe~$\pipe$ is the number of bricks located below~$\pipe$ but inside the axis-parallel rectangle defined by~the two endpoints of~$\pipe$. The \defn{brick vector} of a $k$-twist~$\twist$ is the vector~$\b{x}(\twist) \in \R^n$ whose $i$th coordinate is the brick area of the $i$th pipe of~$\twist$ minus~$\frac{k(n+k)}{2}$. The \defn{brick polytope}~$\Brick$ is the polytope defined as the convex hull of the brick vectors of all $(k,n)$-twists.
\end{definition}

As for the permutahedron described above, we know three descriptions of the brick polytopes: its vertex description, its hyperplane description, and a Minkowski sum description. These properties are proved in~\cite{PilaudSantos-brickPolytope} (modulo our translation in the direction~$\one$).

\begin{proposition}[\cite{PilaudSantos-brickPolytope}]
The brick polytope~$\Brick$ has the following properties for any~${n,k \in \N}$.
\begin{enumerate}[(i)]
\item It lies in the hyperplane~$\HH \eqdef \bigset{\b{x} \in \R^n}{\sum_{i \in [n]} x_i = 0}$ and has dimension~$n-1$.
\item The vertices of~$\Brick$ are precisely the brick vectors of the acyclic $(k,n)$-twists.
\item The normal vectors of the facets of~$\Brick$ are given by the \defn{proper $k$-connected $\{0,1\}$-sequences} of size~$n$, \ie the sequences of~$\{0,1\}^n$ distinct from~$0^n$ and~$1^n$ and which do not contain a subsequence $10^\ell1$ for~$\ell \ge k$. The number of facets of~$\Brick$ is therefore the coefficient of~$t^n$ in
\[
\frac{t^2(2-t^k)}{(1-2t+t^{k+1})(1-t)}.
\]
\item For a brick~$b$, let~$\b{x}_{b}(\twist)$ denote the characteristic vector of the pipes of a $(k,n)$-twist~$\twist$ whose brick area contain~$b$, and define~$\SummandBrick$ to be the convex hull of the vectors~$\b{x}_{b}(\twist)$ for all~$(k,n)$-twists~$\twist$. Then, up to the translation of vector~$\frac{k(n+k)}{2} \one$, the brick polytope $\Brick$ is also the Minkowski sum of the polytopes~$\SummandBrick$ over all bricks~$b$.
\end{enumerate}
\end{proposition}

\begin{example}[J.-L.~Loday's associahedron]
When~$k = 1$, the brick polytope~$\Brick[1]$ coincides (up to translation) with J.-L.~Loday's associahedron~\cite{Loday}. See \fref{fig:PermBrickZono}\,(top center).
\end{example}

\para{Zonotopes}
Zonotopes are particularly important polytopes which are constructed equivalently as projections of cubes, or as Minkowski sums of segments. The combinatorics of a zonotope is completely determined by the matroid of the vector configuration defined by these summands. We refer to~\cite[Lecture~7]{Ziegler} for a presentation of these polytopes and their relations to oriented matroids. Notable examples are graphical zonotopes, defined as follows.

\begin{definition}
The \defn{graphical zonotope}~$\Zono[][\graphG]$ of a graph~$\graphG$ is the Minkowski sum of the segments~$[\b{e}_i,\b{e}_j]$ for all edges~$\{i,j\}$ of~$\graphG$.
\end{definition}

The following classical statement gives the vertex and facet descriptions of graphical zonotopes.

\begin{proposition}
The graphical zonotope~$\Zono[][\graphG]$ has the following properties for any graph~$\graphG$.
\begin{enumerate}[(i)]
\item The dimension of~$\Zono[][\graphG]$ is the number of edges of a maximal cycle-free subgraph of~$\graphG$.
\item The vertices of~$\Zono[][\graphG]$ correspond to the acyclic orientations of~$\graphG$. The $i$th coordinate of the vertex~$\b{x}(\orientation)$ of~$\Zono[][\graphG]$ corresponding to an acyclic orientation~$\orientation$ of~$\graphG$ is the indegree of vertex~$i$ in~$\orientation$.
\item The facets of~$\Zono[][\graphG]$ correspond to minimal cuts of~$\graphG$.
\end{enumerate}
\end{proposition}

For example, the permutahedron~$\Perm[][n]$ is the graphical zonotope of the complete graph, its vertices correspond to permutations of~$[n]$ (acyclic tournaments), and its facets correspond to proper subsets of~$[n]$ (minimal cuts). Here, we focus on the zonotope of~$\Gkn$ whose vertices correspond to the acyclic orientations in~$\AcyclicOrientations(n)$ and whose facets correspond to minimal cuts of~$\Gkn$. As for the previous polytopes, we perturb this zonotope to fit the other two polytopes better (see Remark~\ref{rem:geometricInclusions} for a precise statement), and thus define
\[
\Zono \; \eqdef \; \sum_{1 \le i < j \le n} \lambda(i,j,k,n) \cdot [\b{e}_i, \b{e}_j] \; - \; \frac{(n-1)(n+3k-2)}{6} \, \one,
\]
where
\[
\lambda(i,j,k,n) \eqdef
\begin{cases}
n+k-2|i-j| & \text{if } |i-j| < k, \\
\min(i, n+1-j) \big( n + k - 1 - \min(i, n+1-j) \big) & \text{if } |i-j| = k, \\
0 & \text{if } |i-j| > k.
\end{cases}
\]
Note that this perturbation is only cosmetic and preserves the combinatorics. Indeed, observe that for all~$1 \le i < j \le n$, we have~$\lambda_{i,j,k,n} \ne 0$ if and only if~$|i-j| \le k$. Therefore, the zonotopes~$\Zono$ and~$\Zono[][\Gkn]$ have the same normal fan (see Section~\ref{subsec:normalFans}) and thus the same face lattice. Note that the translation ensures that~$\Zono$ lies in the hyperplane~${\HH \eqdef \bigset{\b{x} \in \R^n}{\sum_{i \in [n]} x_i = 0}}$.

\begin{example}[Cube]
When~$k = 1$, the zonotope~$\Zono[1]$ coincides (up to translation and scaling) with the parallelotope generated by the simple roots~$\b{e}_{i+1}-\b{e}_i$. It has the same combinatorics as the $(n-1)$-dimensional cube. See \fref{fig:PermBrickZono}\,(top right).
\end{example}

\begin{remark}[Geometric inclusions]
\label{rem:geometricInclusions}
We have chosen our normalizations (dilations and translations) so that the polytopes~$\Perm$, $\Brick$ and~$\Zono$ all leave in the hyperplane $\HH \eqdef \bigset{\b{x} \in \R^n}{\sum_{i \in [n]} x_i = 0}$ and fulfill the following inclusions:
\[
\begin{tikzpicture}
  \matrix (m) [matrix of math nodes,row sep=2em,column sep=1em,minimum width=2em]
  {
	\frac{1}{k}\Brick[k] 	& 		  	& \frac{1}{\ell}\Brick[\ell]	\\
							& \Perm[1] 	& 								\\
	\frac{1}{k}\Zono[k]		&		  	& \frac{1}{\ell}\Zono[\ell] 	\\
  };
  \path[->>]
    (m-2-2) edge[draw=none] node [sloped, allow upside down] {$\subseteq$} (m-1-1)
		    edge[draw=none] node [sloped, allow upside down] {$\subseteq$} (m-1-3)
            edge[draw=none] node [sloped, allow upside down] {$\subseteq$} (m-3-1)
            edge[draw=none] node [sloped, allow upside down] {$\subseteq$} (m-3-3)
    (m-1-1) edge[draw=none] node [sloped, allow upside down] {$\subseteq$} (m-3-1)
    (m-1-3) edge[draw=none] node [sloped, allow upside down] {$\subseteq$} (m-3-3)
    (m-1-1) edge[draw=none] node [sloped, allow upside down] {$\subseteq$} (m-1-3)
    (m-3-1) edge[draw=none] node [sloped, allow upside down] {$\subseteq$} (m-3-3);
\end{tikzpicture}
\]
These inclusions are illustrated in Figures~\ref{fig:PermBrickZono}, \ref{fig:PermBrickZonoInclusions} and~\ref{fig:geometricInclusions}. Compare to Remark~\ref{rem:combinatorialInclusions}.

\enlargethispage{.3cm}
Observe moreover that
\[
\Zono = \Zono[n-1] + (k+1-n) \, \Perm[1]
\]
for all~$k \ge n-1$. We therefore obtain that the rescaled polytopes $\frac{1}{k}\Brick$ and~$\frac{1}{k}\Zono$ both converge to~$\frac{1}{k}\Perm = \Perm[1]$ when $k$ tend to~$\infty$ (see \fref{fig:geometricInclusions}).

\begin{figure}[t]
	\centerline{\includegraphics[scale=.68]{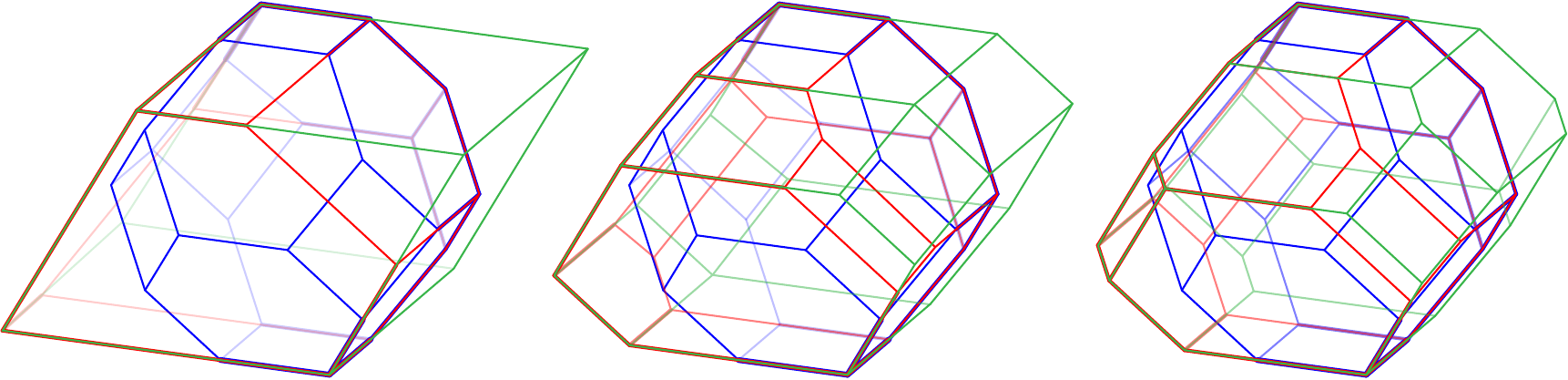}}
	\caption{The inclusions~$\Perm \, \subseteq \, \Brick \, \subseteq \, \Zono$ for~$k = 1$ (left), $k = 2$ (middle) and~$k = 3$ (right). The permutahedron~$\Perm$ is in blue, the brick polytope~$\Brick$ in red and the zonotope~$\Zono$ in green.}
	\label{fig:PermBrickZonoInclusions}
\end{figure}

\begin{figure}[h]
	\centerline{\includegraphics[scale=.68]{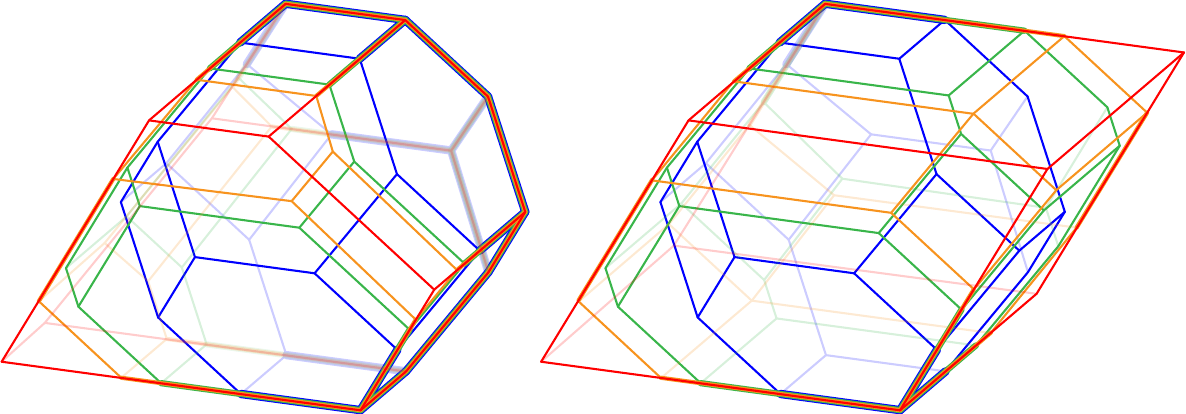}}
	\caption{The inclusions of the brick polytopes~$\frac{1}{k}\Brick[k][4]$ (left) and of the zonotopes~$\frac{1}{k}\Zono[k][4]$ (right) for~$k = 1$ (red), $k = 2$ (orange) and~$k = 3$ (green). Both tend to the classical permutahedron~$\Perm[1][4]$ (blue) when~$k$ tends to~$\infty$.}
	\label{fig:geometricInclusions}
\end{figure}
\end{remark}

\subsection{The geometry of the surjections~$\surjectionPermBrick$, $\surjectionBrickZono$, and~$\surjectionPermZono$}
\label{subsec:normalFans}

Besides Remark~\ref{rem:geometricInclusions}, the main geometric connection between the three polytopes~$\Perm$, $\Brick$ and~$\Zono$ is given by their normal fans. Remember that a \defn{polyhedral fan} is a collection of polyhedral cones of~$\R^n$ closed under faces and which intersect pairwise along faces, see \eg~\cite[Lecture~7]{Ziegler}. The (outer) \defn{normal cone} of a face~$F$ of a polytope~$P$ is the cone generated by the outer normal vectors of the facets of~$P$ containing~$F$. Finally, the (outer) \defn{normal fan} of~$P$ is the collection of the (outer) normal cones of all its faces.

The \defn{incidence cone}~$\Cone(\less)$ and the \defn{braid cone}~$\Cone\polar(\less)$ of a poset~$\less$ are the polyhedral cones defined~by
\[
\Cone(\less) \eqdef \cone\set{\b{e}_i-\b{e}_j}{\text{for all } i \less j}
\quad\text{and}\quad
\Cone\polar(\less) \eqdef \set{\b{x} \in \HH}{x_i \le x_j \text{ for all } i \less j}.
\]
These two cones lie in the space~$\HH$ and are polar to each other. For a permutation~${\tau \in \fS_n}$ (resp.~a twist~$\twist \in \AcyclicTwists(n)$, resp.~an orientation~$\orientation \in \AcyclicOrientations(n)$), we slightly abuse notation to write~$\Cone(\tau)$ (resp.~$\Cone(\twist)$, resp.~$\Cone(\orientation)$) for the incidence cone of the chain~$\tau_1 \less \cdots \less \tau_n$ (resp.~of the transitive closure~$\contactLess{}$ of the contact graph~$\twist\contact$, resp.~of the transitive closure~$\less$ of~$\orientation$). We define similarly the braid cone~$\Cone\polar(\tau)$ (resp.~$\Cone\polar(\twist)$, resp.~$\Cone\polar(\orientation)$). These cones (together with all their faces) form the normal fans of the polytopes of Section~\ref{subsec:PermBrickZono}.

\begin{proposition}
The collections of cones
\[
\bigset{\Cone\polar(\tau)}{\tau \in \fS_n}, \qquad \bigset{\Cone\polar(\twist)}{\twist \in \AcyclicTwists(n)} \qquad\text{and}\qquad \bigset{\Cone\polar(\orientation)}{\orientation \in \AcyclicOrientations(n)},
\]
together with all their faces, are the normal fans of the permutahedron~$\Perm$, the brick polytope~$\Brick$ and the zonotope~$\Zono$ respectively.
\end{proposition}

Observe moreover that the normal fan of~$\Perm$ is also the collection of chambers of the Coxeter arrangement given by all hyperplanes~$\set{\b{x} \in \HH}{x_i = x_j}$ for all~$i,j \in [n]$. Similarly, the normal fan of~$\Zono$ is also the collection of chambers of the graphical arrangement given by the hyperplanes~$\set{\b{x} \in \HH}{x_i = x_j}$ for all edges~$\{i,j\}$ in~$\Gkn$.

Using these normal fans, one can interpret geometrically the maps~$\surjectionPermBrick$, $\surjectionBrickZono$, and~$\surjectionPermZono$ as follows.

\begin{proposition}
The insertion map~$\surjectionPermBrick : \fS_n \to \AcyclicTwists(n)$, the $k$-canopy~${\surjectionBrickZono : \AcyclicTwists(n) \to \AcyclicOrientations(n)}$ and the $k$-recoil map~${\surjectionPermZono : \fS_n \to \AcyclicOrientations(n)}$ are characterized by
\begin{gather*}
\twist = \surjectionPermBrick(\tau) \iff \Cone(\twist) \subseteq \Cone(\tau) \iff \Cone\polar(\twist) \supseteq \Cone\polar(\tau), \\
\orientation = \surjectionBrickZono(\twist) \iff \Cone(\orientation) \subseteq \Cone(\twist) \iff \Cone\polar(\orientation) \supseteq \Cone\polar(\twist), \\
\orientation = \surjectionPermZono(\tau) \iff \Cone(\orientation) \subseteq \Cone(\tau) \iff \Cone\polar(\orientation) \supseteq \Cone\polar(\tau).
\end{gather*}
\end{proposition}

Finally, the lattices studied in Section~\ref{sec:combinatorics} also appear naturally in the geometry of the polytopes~$\Perm$, $\Brick$ and~$\Zono$. Denote by $U$ the vector
\[
U \eqdef (n,n-1,\dots,2,1) - (1,2,\dots,n-1,n) = \sum_{i \in [n]} (n+1-2i) \, \b{e}_i.
\]

\begin{proposition}
When oriented in the direction~$U$, the $1$-skeleton of the permutahedron~$\Perm$ (resp.~of the brick polytope~$\Brick$, resp.~of the zonotope~$\Zono$) is the Hasse diagram of the weak order on permutations (resp.~of the increasing flip lattice on acyclic $(k,n)$-twists, resp.~of the increasing flip lattice on acyclic orientations of~$\Gkn$).
\end{proposition}

\newpage
\section{Algebra of acyclic twists}
\label{sec:algebra}

Motivated by the Hopf algebra on binary trees constructed by J.-L.~Loday and M.~Ronco~\cite{LodayRonco} as a subalgebra of the Hopf algebra on permutations of C.~Malvenuto and C.~Reutenauer (see also~\cite{HivertNovelliThibon-algebraBinarySearchTrees, AguiarSottile}), we define a Hopf algebra with bases indexed by acyclic $k$-twists. We then give combinatorial interpretations of the product and coproduct of this algebra and its dual in terms of $k$-twists. We finally conclude with further algebraic properties of this algebra.

\subsection{Hopf algebras~$\FQSym$ and~$\FQSym^*$}
\label{subsec:FQSym}

We briefly recall here the definition and some elementary properties of C.~Malvenuto and C.~Reutenauer's Hopf algebra on permutations~\cite{MalvenutoReutenauer}. We denote this algebra by~$\FQSym$ to stress out its connection to free quasi-symmetric functions. We will however not use this connection in this paper. We denote by~$\fS \eqdef \bigsqcup_{n \in \N} \fS_n$ the set of all permutations, of arbitrary size.

For~$n,n' \in \N$, let
\[
\fS^{(n,n')} \eqdef \set{\tau \in \fS_{n+n'}}{\tau_1 < \dots < \tau_n \text{ and } \tau_{n+1} < \dots < \tau_{n+n'}}
\]
denote the set of permutations of~$\fS_{n+n'}$ with at most one descent, at position~$n$. 
The \defn{shifted concatenation}~$\tau\bar\tau'$, the \defn{shifted shuffle}~$\tau \shiftedShuffle \tau'$, and the \defn{convolution}~$\tau \convolution \tau'$ of two permutations~$\tau \in \fS_n$ and~$\tau' \in \fS_{n'}$ are classically defined by
\begin{gather*}
\tau\bar\tau' \eqdef [\tau_1, \dots, \tau_n, \tau'_1 + n, \dots, \tau'_{n'} + n] \in \fS_{n+n'}, \\
\tau \shiftedShuffle \tau' \eqdef \bigset{(\tau\bar\tau') \circ \pi^{-1}}{\pi \in \fS^{(n,n')}} 
\qquad\text{and}\qquad
\tau \convolution \tau' \eqdef \bigset{\pi \circ (\tau\bar\tau')}{\pi \in \fS^{(n,n')}}.
\end{gather*}
For example,
\begin{align*}
{\red 12} \shiftedShuffle {\blue 231} & = \{ {\red 12}{\blue 453}, {\red 1}{\blue 4}{\red 2}{\blue 53}, {\red 1}{\blue 45}{\red 2}{\blue 3}, {\red 1}{\blue 453}{\red 2}, {\blue 4}{\red 12}{\blue 53}, {\blue 4}{\red 1}{\blue 5}{\red 2}{\blue 3}, {\blue 4}{\red 1}{\blue 53}{\red 2}, {\blue 45}{\red 12}{\blue 3}, {\blue 45}{\red 1}{\blue 3}{\red 2}, {\blue 453}{\red 12} \}, \\
{\red 12} \convolution {\blue 231} & = \{ {\red 12}{\blue 453}, {\red 13}{\blue 452}, {\red 14}{\blue 352}, {\red 15}{\blue 342}, {\red 23}{\blue 451}, {\red 24}{\blue 351}, {\red 25}{\blue 341}, {\red 34}{\blue 251}, {\red 35}{\blue 241}, {\red 45}{\blue 231} \}.
\end{align*}
\enlargethispage{.1cm}
We also use the notation~$\underprod{\tau}{\tau'} \eqdef \tau\bar\tau'$ and~$\overprod{\tau}{\tau'} \eqdef \bar\tau'\tau$.

\begin{definition}
We denote by~$\FQSym$ the Hopf algebra with basis~$(\F_\tau)_{\tau \in \fS}$ and whose product and coproduct are defined by
\[
\F_\tau \product \F_{\tau'} = \sum_{\sigma \in \tau \shiftedShuffle \tau'} \F_\sigma
\qquad\text{and}\qquad
\coproduct \F_\sigma = \sum_{\sigma \in \tau \convolution \tau'} \F_\tau \otimes \F_{\tau'}.
\]
This algebra is graded by the size of the permutations.
\end{definition}

\begin{proposition}
\label{prop:productIntervalWeakOrder}
A product of weak order intervals in~$\FQSym$ is a weak order interval: for any two weak order intervals~$[\mu, \omega]$ and~$[\mu', \omega']$, we have
\[
\bigg( \sum_{\mu \le \tau \le \omega} \F_\tau \bigg) \product \bigg( \sum_{\mu' \le \tau' \le \omega'} \F_{\tau'} \bigg) = \sum_{\underprod{\mu}{\mu'} \le \sigma \le \overprod{\omega}{\omega'}} \F_\sigma.
\]
\end{proposition}

\begin{corollary}
For~$\tau \in \fS_n$, define
\[
\EFQSym^\tau = \sum_{\tau \le \tau'} \F_{\tau'}
\qquad\text{and}\qquad
\HFQSym^\tau = \sum_{\tau' \le \tau} \F_{\tau'}
\]
where~$\le$ is the weak order on~$\fS_n$. Then~$(\EFQSym_\tau)_{\tau \in \fS}$ and~$(\HFQSym_\tau)_{\tau \in \fS}$ are multiplicative bases of~$\FQSym$:
\[
\EFQSym^\tau \product \EFQSym^{\tau'} = \EFQSym^{\underprod{\tau}{\tau'}}
\qquad\text{and}\qquad
\HFQSym^\tau \product \HFQSym^{\tau'} = \HFQSym^{\overprod{\tau}{\tau'}}.
\]
A permutation~$\tau \in \fS_n$ is $\EFQSym$-decomposable (resp.~$\HFQSym$-decomposable) if and only if there exists~${k \in [n-1]}$ such that~$\tau([k]) = [k]$ (resp.~such that~$\tau([k]) = [n] \ssm [k]$). Moreover, $\FQSym$ is freely generated by the elements~$\EFQSym^\tau$ (resp.~$\HFQSym^\tau$) for the $\EFQSym$-indecomposable (resp.~$\HFQSym$-indecomposable) permutations.
\end{corollary}

We will also consider the dual Hopf algebra of~$\FQSym$, defined as follows.

\begin{definition}
We denote by~$\FQSym^*$ the Hopf algebra with basis~$(\G_\tau)_{\tau \in \fS}$ and whose product and coproduct are defined by
\[
\G_\tau \product \G_{\tau'} = \sum_{\sigma \in \tau \convolution \tau'} \G_\sigma
\qquad\text{and}\qquad
\coproduct \G_\sigma = \sum_{\sigma \in \tau \shiftedShuffle \tau'} \G_\tau \otimes \G_{\tau'}.
\]
This algebra is graded by the size of the permutations.
\end{definition}

\subsection{Subalgebra of~$\FQSym$}
\label{subsec:subalgebra}

We denote by~$\Twist$ the vector subspace of~$\FQSym$ generated by the elements
\[
\PTwist_{\twist} \eqdef \sum_{\substack{\tau \in \fS \\ \surjectionPermBrick(\tau) = \twist}} \F_\tau = \sum_{\tau \in \linearExtensions(\twist\contact)} \F_\tau,
\]
for all acyclic~$k$-twists~$\twist$.
For example, for the $(k,5)$-twists of \fref{fig:insertionAlgorithm}, we have
\[
\PTwist_{\includegraphics[scale=.5]{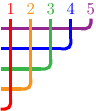}} = \sum_{\tau \in \fS_5} \F_\tau
\qquad
\PTwist\raisebox{.05cm}{$_{\hspace{-.1cm}\includegraphics[scale=.5]{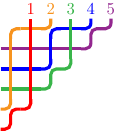}}$} = \!\!\!\! \begin{array}[t]{c} \phantom{+} \; \F_{13542} + \F_{15342} \\ + \; \F_{31542} + \F_{51342} \\ + \; \F_{35142} + \F_{53142} \\ + \; \F_{35412} + \F_{53412} \end{array}
\qquad
\PTwist\raisebox{.1cm}{$_{\hspace{-.2cm}\includegraphics[scale=.5]{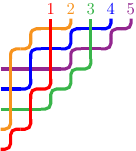}}$} = \!\!\!\! \begin{array}[t]{c} \phantom{+} \; \F_{31542} \\ + \; \F_{35142} \end{array}
\qquad
\PTwist\raisebox{.15cm}{$_{\hspace{-.3cm}\includegraphics[scale=.5]{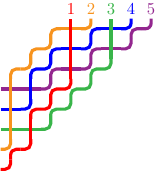}}$} = \F_{31542}.
\]

\begin{theorem}
\label{theo:twistSubalgebra}
$\Twist$ is a Hopf subalgebra of~$\FQSym$.
\end{theorem}

\begin{proof}
This statement is a particular case of the results of~\cite{Reading-HopfAlgebras}. Alternatively, we could also invoke the formalism of~\cite{Hivert-habilitation, HivertNzeutchap, Priez} and just observe that the $k$-twist congruence~$\equiv^k$ is compatible with the standardization and the restriction to intervals. However, we prefer to repeat the straightforward argument for the reader who sees it for the first time.

We need to show that~$\Twist$ is stable by product and coproduct. For the product, consider two $k$-twists~$\twist \in \AcyclicTwists(n)$ and~$\twist' \in \AcyclicTwists(n')$ and two permutations~${\sigma = U a c V_1 b_1 \cdots V_k b_k W}$ and ${\tilde\sigma = U c a V_1 b_1 \cdots V_k b_k W}$ with~${a < b_1, \dots, b_k < c}$. Assume that~$\F_\sigma$ appears in the product~${\PTwist_{\twist} \product \PTwist_{\twist'}}$. Thus, there exist~${\tau \in \linearExtensions(\twist)}$ and~$\tau' \in \linearExtensions(\twist')$ such that~$\sigma$ belongs to~$\tau \shiftedShuffle \tau'$. Then
\begin{itemize}
\item if~$a \le n < c$, then $\tilde\sigma$ also belongs to~$\tau \shiftedShuffle \tau'$, so that~$\F_{\tilde\sigma}$ also appears in the product~${\PTwist_{\twist} \product \PTwist_{\twist'}}$.
\item if~$c \le n$, then~$\tau = X a c Y_1 b_1 \dots Y_k b_k Z$ is $\equiv^k$-congruent to~$\tilde\tau = X c a Y_1 b_1 \dots Y_k b_k Z$. Since~$\tilde\sigma$ belongs to~$\tilde\tau \shiftedShuffle \tau'$, the element~$\F_{\tilde\sigma}$ also appears in the product~${\PTwist_{\twist} \product \PTwist_{\twist'}}$.
\item If~$n < a$, the argument is similar, exchanging~$ac$ to~$ca$ in~$\tau'$.
\end{itemize}
We conclude that if~$\sigma \equiv^k \tilde\sigma$, then~$\F_\sigma$ appears in the product~${\PTwist_{\twist} \product \PTwist_{\twist'}}$ if and only if~$\F_{\tilde\sigma}$ does. Thus, the product of two $k$-twist congruence classes decomposes into a sum of $k$-twist congruence classes.

For the coproduct, consider a $k$-twist~$\twist[S]$ and three permutations~$\tau$, ${\tau' =  U a c V_1 b_1 \cdots V_k b_k W}$ and ${\tilde\tau' = U c a V_1 b_1 \cdots V_k b_k W}$ with~${a < b_1, \dots, b_k < c}$. Assume that the tensor~$\tau \otimes \tau'$ appears in the coproduct~$\coproduct \PTwist_{\twist[S]}$. Thus, there exist~$\sigma \in \linearExtensions(\twist[S]) \cap (\tau \convolution \tau')$. By definition of the convolution, we can thus write~${\sigma = \hat\tau \hat U \hat a \hat c \hat V_1 \hat b_1 \cdots \hat V_k \hat b_k \hat W}$ for some~$\hat a < \hat b_1, \dots, \hat b_k < \hat c$. Then the permutation $\tilde\sigma \eqdef \hat\tau \hat U \hat c \hat a \hat V_1 \hat b_1 \cdots \hat V_k \hat b_k \hat W$ belongs to~$\linearExtensions(\twist[S]) \cap (\tau \convolution \tilde\tau')$, so that the tensor~$\tau \otimes \tilde\tau'$ also appears in the coproduct~$\coproduct \PTwist_{\twist[S]}$. By symmetry, we conclude that if~$\tau \equiv^k \tilde\tau$ and~${\tau' \equiv^k \tilde\tau'}$, the tensor~$\F_\tau \otimes \F_{\tau'}$ appears in the coproduct~$\coproduct\PTwist_{\twist[S]}$ if and only if the tensor~$\F_{\tilde\tau} \otimes \F_{\tilde\tau'}$ does. In other words, the coproduct of a $k$-twist congruence class decomposes as a sum of tensor products of $k$-twist congruence~classes.
\end{proof}

\begin{example}[Polynomial ring]
The bijection which sends the unique~$(0,n)$-twist to~$X^n/n!$ defines an isomorphism from the $0$-twist algebra~$\Twist[0]$ to the polynomial ring~$\K[X]$.
\end{example}

\begin{example}[J.-L.~Loday and M.~Ronco's algebra]
The bijection given in Example~\ref{exm:1twistsTriangulations} (see also \fref{fig:1twistsTriangulations}) defines an isomorphism from the $1$-twist algebra~$\Twist[1]$ to M.~Ronco and J.-L.~Loday's Hopf algebra~$\PBT$ on planar binary trees~\cite{LodayRonco, HivertNovelliThibon-algebraBinarySearchTrees}.
\end{example}

\enlargethispage{.2cm}
We now aim at understanding the product and the coproduct in~$\Twist$ directly on $k$-twists. Although they are not always as satisfactory, our descriptions naturally extend classical results on the binary tree Hopf algebra~$\PBT$ described in~\cite{LodayRonco, AguiarSottile, HivertNovelliThibon-algebraBinarySearchTrees}.

\para{Product}
To describe the product in~$\Twist$, we need the following notation, which is illustrated in \fref{fig:underOver}. For a $(k,n)$-twist~$\twist$ and a $(k,n')$-twist~$\twist'$, we denote by~$\underprod{\twist}{\twist'}$ the $(k,n+n')$-twist obtained by inserting~$\twist$ in the first rows and columns of~$\twist'$ and by~$\overprod{\twist}{\twist'}$ the $(k,n+n')$-twist obtained by inserting~$\twist'$ in the last rows and columns of~$\twist$.

\begin{figure}[t]
	\centerline{\includegraphics[scale=1.4]{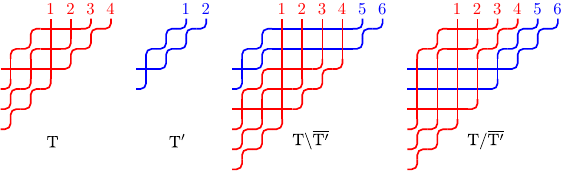}}
	\caption{Two twists~$\twist,\twist'$ (left) and the two twists~$\underprod{\twist}{\twist'}$ and~$\overprod{\twist}{\twist'}$ (right).}
	\label{fig:underOver}
\end{figure}

\begin{proposition}
\label{prop:product}
For any acyclic $k$-twists~$\twist \in \AcyclicTwists(n)$ and~$\twist' \in \AcyclicTwists(n')$, the product~$\PTwist_{\twist} \product \PTwist_{\twist'}$ is given by
\[
\PTwist_{\twist} \product \PTwist_{\twist'}  = \sum_{\twist[S]} \PTwist_{\twist[S]},
\]
where~$\twist[S]$ runs over the interval between~$\underprod{\twist}{\twist'}$ and~$\overprod{\twist}{\twist'}$ in the $(k,n+n')$-twist lattice.
\end{proposition}

\begin{proof}
Consider two acyclic $k$-twists~$\twist, \twist'$. By Proposition~\ref{prop:intervals}, their fibers under~$\surjectionPermBrick$ are intervals of the weak order, which we denote by~$[\mu, \omega]$ and~$[\mu', \omega']$ respectively. By Proposition~\ref{prop:productIntervalWeakOrder}, the product~$\PTwist_{\twist} \product \PTwist_{\twist'}$ is therefore the weak order interval~$[\underprod{\mu}{\mu'}, \overprod{\omega}{\omega'}]$. Theorem~\ref{theo:twistSubalgebra} ensures that this interval is partitioned into various fibers of~$\surjectionPermBrick$. In particular, the fiber of~$\underprod{\twist}{\twist'}$ contains~$\underprod{\mu}{\mu'}$ while the fiber of~$\overprod{\twist}{\twist'}$ contains~$\overprod{\omega}{\omega'}$. Theorem~\ref{theo:ktwistLatticeCongruence} finally ensures that $[\underprod{\mu}{\mu'}, \overprod{\omega}{\omega'}]$ is precisely the union of the fibers of the increasing flip interval~$[\underprod{\twist}{\twist'}, \overprod{\twist}{\twist'}]$.
\end{proof}

\fref{fig:exmProduct} illustrates the product between two $2$-twists in~$\Twist[2]$.

\begin{figure}[h]
    \centerline{$
    \begin{array}{@{}c@{\;=\;}c@{+}c@{+}c@{+}c@{+}c@{+}c@{+}c@{+}c@{}}
    \PTwist\raisebox{.1cm}{$_{\hspace{-.2cm}\includegraphics[scale=.5]{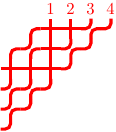}}$} \product \PTwist\raisebox{.1cm}{$_{\hspace{-.2cm}\includegraphics[scale=.5]{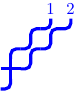}}$} & 
    \multicolumn{8}{@{}l}{(\F_{\red 1423} + \F_{\red 4123}) \product \F_{\blue 21}}
    \\[-.6cm]
    & \begin{bmatrix} \phantom{+ \, } \F_{{\red 1423}{\blue 65}} \\ + \, \F_{{\red 4123}{\blue 65}} \end{bmatrix}
    & \begin{bmatrix} \phantom{+ \, } \F_{{\red 142}{\blue 6}{\red 3}{\blue 5}} \\ + \, \F_{{\red 14}{\blue 6}{\red 23}{\blue 5}} \\ + \, \F_{{\red 412}{\blue 6}{\red 3}{\blue 5}} \\ + \, \F_{{\red 41}{\blue 6}{\red 23}{\blue 5}} \\ + \, \F_{{\red 4}{\blue 6}{\red 123}{\blue 5}} \end{bmatrix}
    & \begin{bmatrix} \phantom{+ \, } \F_{{\red 1}{\blue 6}{\red 423}{\blue 5}} \\ + \, \F_{{\blue 6}{\red 1423}{\blue 5}} \\ + \, \F_{{\blue 6}{\red 4123}{\blue 5}} \end{bmatrix}
    & \begin{bmatrix} \phantom{+ \, } \F_{{\red 142}{\blue 65}{\red 3}} \\ + \, \F_{{\red 14}{\blue 6}{\red 2}{\blue 5}{\red 3}} \\ + \, \F_{{\red 412}{\blue 65}{\red 3}} \\ + \, \F_{{\red 41}{\blue 6}{\red 2}{\blue 5}{\red 3}} \\ + \, \F_{{\red 4}{\blue 6}{\red 12}{\blue 5}{\red 3}} \end{bmatrix}
    & \begin{bmatrix} \phantom{+ \, } \F_{{\red 1}{\blue 6}{\red 42}{\blue 5}{\red 3}} \\ + \, \F_{{\blue 6}{\red 142}{\blue 5}{\red 3}} \\ + \, \F_{{\blue 6}{\red 412}{\blue 5}{\red 3}} \end{bmatrix}
    & \begin{bmatrix} \phantom{+ \, } \F_{{\red 14}{\blue 65}{\red 23}} \\ + \, \F_{{\red 41}{\blue 65}{\red 23}} \\ + \, \F_{{\red 4}{\blue 6}{\red 1}{\blue 5}{\red 23}} \\ + \, \F_{{\red 4}{\blue 65}{\red 123}} \end{bmatrix}
    & \begin{bmatrix} \phantom{+ \, } \F_{{\red 1}{\blue 6}{\red 4}{\blue 5}{\red 23}} \\ + \, \F_{{\blue 6}{\red 14}{\blue 5}{\red 23}} \\ + \, \F_{{\blue 6}{\red 41}{\blue 5}{\red 23}} \\ + \, \F_{{\blue 6}{\red 4}{\blue 5}{\red 123}} \end{bmatrix}
    & \begin{bmatrix} \phantom{+ \, } \F_{{\red 1}{\blue 65}{\red 423}} \\ + \, \F_{{\blue 6}{\red 1}{\blue 5}{\red 423}} \\ + \, \F_{{\blue 65}{\red 1423}} \\ + \, \F_{{\blue 65}{\red 4123}} \end{bmatrix}
    \\[1cm]
    & \PTwist\raisebox{.1cm}{$_{\hspace{-.2cm}\includegraphics[scale=.5]{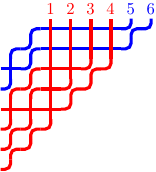}}$}
    & \PTwist\raisebox{.1cm}{$_{\hspace{-.2cm}\includegraphics[scale=.5]{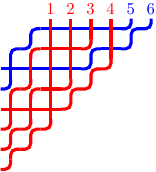}}$}
    & \PTwist\raisebox{.1cm}{$_{\hspace{-.2cm}\includegraphics[scale=.5]{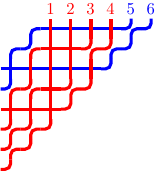}}$}
    & \PTwist\raisebox{.1cm}{$_{\hspace{-.2cm}\includegraphics[scale=.5]{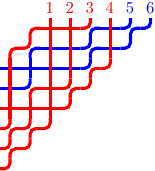}}$}
    & \PTwist\raisebox{.1cm}{$_{\hspace{-.2cm}\includegraphics[scale=.5]{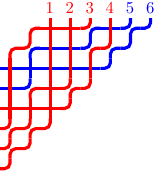}}$}
    & \PTwist\raisebox{.1cm}{$_{\hspace{-.2cm}\includegraphics[scale=.5]{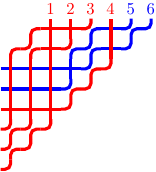}}$}
    & \PTwist\raisebox{.1cm}{$_{\hspace{-.2cm}\includegraphics[scale=.5]{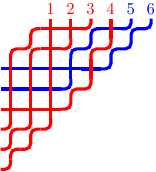}}$}
    & \PTwist\raisebox{.1cm}{$_{\hspace{-.2cm}\includegraphics[scale=.5]{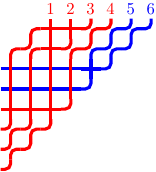}}$}
    \end{array}
    $}
    \caption{An example of product in the $2$-twist algebra~$\Twist[2]$.}
	\label{fig:exmProduct}
\end{figure}

\para{Coproduct}
Our description of the coproduct in~$\Twist$ is unfortunately not as simple as the coproduct in~$\PBT$. It is very closed to the description of the direct computation using the coproduct of~$\FQSym$. We need the following definition. A \defn{cut} in a $k$-twist~$\twist[S]$ is a set~$\gamma$ of edges of the contact graph~$\twist[S]\contact$ such that any path in~$\twist[S]\contact$ from a leaf to the root contains precisely one edge of~$\gamma$. We then denote by~$A\contact(\twist[S], \gamma)$ (resp.~$B\contact(\twist[S], \gamma)$) the restriction of the contact graph~$\twist[S]\contact$ to the nodes above (resp.~below)~$\gamma$. Moreover, $A\contact(\twist[S], \gamma)$ is the contact graph of the $k$-twist~$A(\twist[S], \gamma)$ obtained from~$\twist[S]$ by deleting all pipes below~$\gamma$ in~$\twist[S]\contact$. Nevertheless, note that~$B\contact(\twist[S], \gamma)$ is not \apriori{} the contact graph of a $k$-twist.

\begin{proposition}
For any acyclic $k$-twist~$\twist[S] \in \AcyclicTwists(m)$, the coproduct~$\coproduct\PTwist_{\twist[S]}$ is given by
\[
\coproduct\PTwist_{\twist[S]} = \sum_\gamma \bigg( \sum_\tau \PTwist_{\surjectionPermBrick(\tau)} \bigg) \otimes \PTwist_{A(\twist[S], \gamma)},
\]
where~$\gamma$ runs over all cuts of~$\twist[S]$ and~$\tau$ runs over a set of representatives of the $k$-twist congruence classes of the linear extensions of~$B\contact(\twist[S], \gamma)$.
\end{proposition}

\begin{proof}
By Theorem~\ref{theo:twistSubalgebra}, any element of~$\coproduct\twist[S]$ is of the form~$\surjectionPermBrick(\tau) \otimes \surjectionPermBrick(\tau')$ for some permutations~$\tau \in \fS_n$ and~$\tau' \in \fS_{n'}$ such that~$\tau \convolution \tau'$ contains a linear extension~$\sigma$ of~$\twist[S]\contact$. Let~$\gamma$ denote the cut of~$\twist[S]$ that separates~$\sigma(\{1, \dots, n\})$ from~$\sigma(\{n+1, \dots n+n'\})$. Then~$\tau$ and~$\tau'$ are linear extensions of~$B\contact(\twist[S], \gamma)$ and~$A\contact(\twist[S], \gamma)$ respectively, so that~$\surjectionPermBrick(\tau) \otimes \surjectionPermBrick(\tau')$ indeed appear in the sum on the right hand side. Conversely, for any cut~$\gamma$ of~$\twist[S]$ and linear extensions~$\tau$ of~$B\contact(\twist[S], \gamma)$ and~$\tau'$ of~$A\contact(\twist[S], \gamma)$, there is a linear extension~$\sigma$ of~$\twist[S]\contact$ in~$\tau \convolution \tau'$, so that~$\surjectionPermBrick(\tau) \otimes A(\twist[S], \gamma) = \surjectionPermBrick(\tau) \otimes \surjectionPermBrick(\tau')$ appears in~$\coproduct\twist[S]$. Finally, we have to prove that the coproduct is boolean, meaning that only~$0/1$ coefficients may appear: this follows from the fact that we can reconstruct the cut~$\gamma$ from~$A(\twist[S], \gamma)$ and the $k$-twist congruence class of~$\tau$ from~$\surjectionPermBrick(\tau)$.
\end{proof}

\fref{fig:exmCoproduct} illustrates the coproduct of a $2$-twist in~$\Twist[2]$.

\begin{figure}[h]
    \begin{align*}
    \coproduct \PTwist\raisebox{.1cm}{$_{\hspace{-.2cm}\includegraphics[scale=.5]{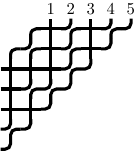}}$} \;\; = & \;\; \coproduct (\F_{31542} + \F_{35142}) \\[-.7cm]
    = &
    \phantom{+ \;\;}  1 \otimes (\F_{31542} + \F_{35142})
    \; + \; \F_1 \otimes (\F_{1432} + \F_{4132})
    \; + \; \F_{21} \otimes \F_{321}
    \; + \; \F_{12} \otimes \F_{132} \\
    & + \; \F_{213} \otimes \F_{21}
    \; + \; \F_{231} \otimes \F_{21}
    \; + \; \F_{2143} \otimes \F_1
    \; + \; \F_{2413} \otimes \F_1
    \; + \; (\F_{31542} + \F_{35142}) \otimes 1 \\[.3cm]
    = &
    \phantom{+ \;\;} 1 \otimes \PTwist\raisebox{.1cm}{$_{\hspace{-.2cm}\includegraphics[scale=.5]{exmCoproduct/2twist31542}}$}
    \; + \; \PTwist\raisebox{.1cm}{$_{\hspace{-.2cm}\includegraphics[scale=.5]{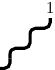}}$} \otimes \PTwist\raisebox{.1cm}{$_{\hspace{-.2cm}\includegraphics[scale=.5]{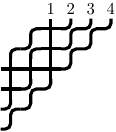}}$}
    \; + \; \PTwist\raisebox{.1cm}{$_{\hspace{-.2cm}\includegraphics[scale=.5]{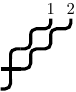}}$} \otimes \PTwist\raisebox{.1cm}{$_{\hspace{-.2cm}\includegraphics[scale=.5]{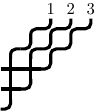}}$} 
    \; + \; \PTwist\raisebox{.1cm}{$_{\hspace{-.2cm}\includegraphics[scale=.5]{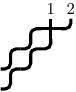}}$} \otimes \PTwist\raisebox{.1cm}{$_{\hspace{-.2cm}\includegraphics[scale=.5]{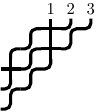}}$} \\
    & + \; \PTwist\raisebox{.1cm}{$_{\hspace{-.2cm}\includegraphics[scale=.5]{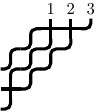}}$} \otimes \PTwist\raisebox{.1cm}{$_{\hspace{-.2cm}\includegraphics[scale=.5]{exmCoproduct/2twist21}}$} 
    \; + \; \PTwist\raisebox{.1cm}{$_{\hspace{-.2cm}\includegraphics[scale=.5]{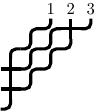}}$} \otimes \PTwist\raisebox{.1cm}{$_{\hspace{-.2cm}\includegraphics[scale=.5]{exmCoproduct/2twist21}}$} 
    \; + \; \PTwist\raisebox{.1cm}{$_{\hspace{-.2cm}\includegraphics[scale=.5]{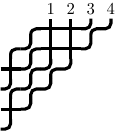}}$} \otimes \PTwist\raisebox{.1cm}{$_{\hspace{-.2cm}\includegraphics[scale=.5]{exmCoproduct/2twist1}}$} 
    \; + \; \PTwist\raisebox{.1cm}{$_{\hspace{-.2cm}\includegraphics[scale=.5]{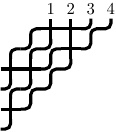}}$} \otimes \PTwist\raisebox{.1cm}{$_{\hspace{-.2cm}\includegraphics[scale=.5]{exmCoproduct/2twist1}}$}
    \; + \; \PTwist\raisebox{.1cm}{$_{\hspace{-.2cm}\includegraphics[scale=.5]{exmCoproduct/2twist31542}}$} \otimes 1
    \end{align*}
	\caption{An example of coproduct in the $2$-twist algebra~$\Twist[2]$.}
	\label{fig:exmCoproduct}
\end{figure}

\para{Matriochka algebras}
We now connect the twist algebra to the $k$-recoil algebra~$\Rec$ defined as the Hopf subalgebra of~$\FQSym$ generated by the elements
\[
\XRec_{\orientation} \eqdef \sum_{\substack{\tau \in \fS \\ \surjectionPermZono(\tau) = \orientation}} \F_\tau,
\]
for all acyclic orientations~$\orientation$ of the graph~$\Gkn$ for all~$n \in \N$. This algebra was first defined by J.-C.~Novelli, C.~Reutenauer and J.-Y.~Thibon in~\cite{NovelliReutenauerThibon} (the dual of~$\Rec$ is denoted~$\DSym$ in their paper).
The commutative diagram of Proposition~\ref{prop:latticeHomomorphisms} ensures that
\[
\XRec_{\orientation} = \sum_{\substack{\twist \in \AcyclicTwists \\ \surjectionBrickZono(\twist) = \orientation}} \PTwist_{\twist},
\]
and thus that~$\Rec$ is a Hopf subalgebra of~$\Twist$.

\begin{remark}[Algebraic inclusions]
\label{rem:algebraicInclusions}
Following Remarks~\ref{rem:combinatorialInclusions} and~\ref{rem:geometricInclusions}, note that we have in fact the following inclusions of subalgebras for~$k > \ell$:
\[
\begin{tikzpicture}
  \matrix (m) [matrix of math nodes,row sep=2em,column sep=1em,minimum width=2em]
  {
	\Twist[k] 	& 		 & \Twist[\ell] \\
				& \FQSym & 				\\
	\Rec[k] 	&		 & \Rec[\ell] 	\\
  };
  \path[->>]
    (m-2-2) edge[draw=none] node [sloped, allow upside down] {$\supseteq$} (m-1-1)
		    edge[draw=none] node [sloped, allow upside down] {$\supseteq$} (m-1-3)
            edge[draw=none] node [sloped, allow upside down] {$\supseteq$} (m-3-1)
            edge[draw=none] node [sloped, allow upside down] {$\supseteq$} (m-3-3)
    (m-1-1) edge[draw=none] node [sloped, allow upside down] {$\supseteq$} (m-3-1)
    (m-1-3) edge[draw=none] node [sloped, allow upside down] {$\supseteq$} (m-3-3)
    (m-1-1) edge[draw=none] node [sloped, allow upside down] {$\supseteq$} (m-1-3)
    (m-3-1) edge[draw=none] node [sloped, allow upside down] {$\supseteq$} (m-3-3);
\end{tikzpicture}
\]
\end{remark}

\subsection{Quotient algebra of~$\FQSym^*$}
\label{subsec:quotientAlgebra}

The following statement is automatic from Theorem~\ref{theo:twistSubalgebra}.

\begin{theorem}
The graded dual~$\Twist[k*]$ of the $k$-twist algebra is the quotient of~$\FQSym^*$ under the $k$-twist congruence~$\equiv^k$. The dual basis~$\QTwist_{\twist}$ of~$\PTwist_{\twist}$ is expressed as~$\QTwist_{\twist} = \pi(\G_\tau)$, where~$\pi$ is the quotient map and~$\tau$ is any permutation such that~$\surjectionPermBrick(\tau) = \twist$.
\end{theorem}

Similarly as in the previous section, we try to describe combinatorially the product and coproduct of $\QTwist$-basis elements of~$\Twist[k*]$ in terms of operations on Cambrian trees.

\para{Product} 
Once more, our description of the product in the dual twist algebra is not as simple as the product in~$\PBT^*$, and is very closed to the description of the direct computation using the product of~$\FQSym^*$. We use the following notation. For~$X = \{x_1 < \dots < x_n\} \in \binom{[n+n']}{n}$, $\tau \in \fS_n$, and~$\twist' \in \AcyclicTwists(n')$, we denote by~$\twist' \insertion{} (\tau \cdot X)$ the result of iteratively inserting~$x_{\tau_n}, \dots, x_{\tau_1}$ in the $k$-twist~$\twist'$ relabeled increasingly by~$[n+n'] \ssm X$.

\begin{proposition}
For any acyclic $k$-twists~$\twist \in \AcyclicTwists(n)$ and~$\twist' \in \AcyclicTwists(n')$, the product~$\QTwist_{\twist} \product \QTwist_{\twist'}$ is given~by
\[
\QTwist_{\twist} \product \QTwist_{\twist'}  = \sum_X \QTwist_{\twist' \insertion{} (\tau \cdot X)}
\]
where~$X$ runs over~$\binom{[n+n']}{n}$ and~$\tau$ is an arbitrary permutation such that~$\surjectionPermBrick(\tau) = \twist$.
\end{proposition}

\begin{proof}
Consider~$\tau \in \fS_n$ and~$\tau' \in \fS_{n'}$ such that~$\surjectionPermBrick(\tau) = \twist$ and~$\surjectionPermBrick(\tau') = \twist'$, and a permutation~$\sigma$ in the convolution~$\tau \convolution \tau'$. Let~$X$ denote the first~$n$ values in~$\sigma$. Since the relative order of the last~$n'$ entries in~$\sigma$ is that of the entries of~$\tau'$, the insertion of the last~$n'$ values creates a copy of~$\twist'$. The remaining entries are then inserted in this copy of~$\twist'$ at the positions given by~$X$ according to the order given by of~$\tau$. The result immediately follows.
\end{proof}

\fref{fig:exmProductDual} illustrates the product of two $2$-twists in~$\Twist[2*]$.

\begin{figure}[h]
	\centerline{$
	\begin{array}{c@{\;=\;}c@{\;+\;}c@{\;+\;}c@{\;+\;}c@{\;+\;}c@{\;+\;}c}
	\QTwist\raisebox{.1cm}{$_{\hspace{-.2cm}\includegraphics[scale=.5]{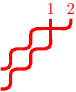}}$} \product \QTwist\raisebox{.1cm}{$_{\hspace{-.2cm}\includegraphics[scale=.5]{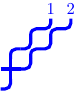}}$}
	& \multicolumn{6}{@{}l}{\G_{\red 12} \product \G_{\blue 21}}
	\\[-.3cm]
	& \G_{{\red 12}{\blue 43}}
	& \G_{{\red 13}{\blue 42}}
	& \G_{{\red 14}{\blue 32}}
	& \G_{{\red 23}{\blue 41}}
	& \G_{{\red 24}{\blue 31}}
	& \G_{{\red 34}{\blue 21}}
	\\[.1cm]
	& \QTwist\raisebox{.1cm}{$_{\hspace{-.2cm}\includegraphics[scale=.5]{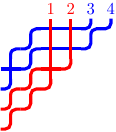}}$}
	& \QTwist\raisebox{.1cm}{$_{\hspace{-.2cm}\includegraphics[scale=.5]{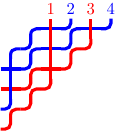}}$}
	& \QTwist\raisebox{.1cm}{$_{\hspace{-.2cm}\includegraphics[scale=.5]{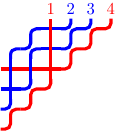}}$}
	& \QTwist\raisebox{.1cm}{$_{\hspace{-.2cm}\includegraphics[scale=.5]{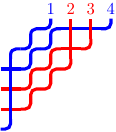}}$}
	& \QTwist\raisebox{.1cm}{$_{\hspace{-.2cm}\includegraphics[scale=.5]{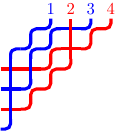}}$}
	& \QTwist\raisebox{.1cm}{$_{\hspace{-.2cm}\includegraphics[scale=.5]{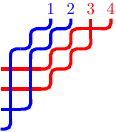}}$}
	\end{array}
	$}
	\caption{An example of product in the dual $2$-twist algebra~$\Twist[2*]$.}
	\label{fig:exmProductDual}
\end{figure}

\vspace{-.3cm}
\enlargethispage{-.4cm}
\para{Coproduct} 
Our description of the coproduct is more satisfactory. It is a special case of a coproduct on arbitrary pipe dreams studied by N.~Bergeron and C.~Ceballos~\cite{BergeronCeballos}. We need the following notations, illustrated in \fref{fig:untangle}. For an acyclic $(k,m)$-twist~$\twist[S]$ and a position~${p \in \{0, \dots, m\}}$, we define two $k$-twists~$L(\twist[S],p) \in \AcyclicTwists(p)$ and~$R(\twist[S],p) \in \AcyclicTwists(m-p)$ as follows. The twist~$L(\twist[S],p)$ is obtained by erasing the last~$m-p$ pipes in~$\twist[S]$ and glide the elbows of the remaining pipes as northwest as possible. More precisely, each elbow~$\elb$ of one of the first $p$ pipes is translated one step north (resp.~west) for each of the last~$m-p$ pipes passing north (resp.~west) of~$\elb$. The definition is similar for~$R(\twist[S],p)$, except that we erase the first~$p$ pipes instead of the last~$m-p$ pipes.

\begin{figure}[t]
	\centerline{\includegraphics[scale=1.4]{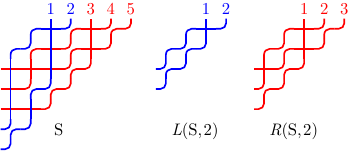}}
	\caption{A twist~$\twist[S]$ (left) and the two twists~$L(\twist[S],p)$ (middle) and~$R(\twist[S],p)$ (right).}
	\label{fig:untangle}
\end{figure}

\begin{proposition}
For any acyclic $k$-twist~$\twist[S]  \in \AcyclicTwists(m)$, the coproduct~$\coproduct\QTwist_{\twist[S]}$ is given by
\[
\coproduct\QTwist_{\twist[S]} = \sum_{p \in \{0, \dots, m\}} \QTwist_{L(\twist[S],p)} \otimes \QTwist_{R(\twist[S], p)}.
\]
\end{proposition}

\begin{proof}
Consider~$\sigma \in \fS_m$ such that~$\surjectionPermBrick(\sigma) = \twist[S]$, let~$p \in \{0, \dots, m\}$, and let~$\tau \in \fS_p, \tau' \in \fS_{m-p}$ be the two permutations such that~$\sigma \in \tau \shiftedShuffle \tau'$. By definition, $\tau$ (resp.~$\tau'$) is given by the relative order of the first~$p$ (resp.~last~$m-p$) values of~$\sigma$. It is immediate to see that the insertion process then gives~$\surjectionPermBrick(\tau) = L(\twist[S],p)$ and~$\surjectionPermBrick(\tau') = R(\twist[S],p)$. The result follows.
\end{proof}

\fref{fig:exmCoproductDual} illustrates the coproduct of a $2$-twist in~$\Twist[2*]$.

\begin{figure}[h]
	\centerline{$
	\begin{array}{c@{\;=\;}c@{\;+\;}c@{\;+\;}c@{\;+\;}c@{\;+\;}c@{\;+\;}c}
	\coproduct \QTwist\raisebox{.1cm}{$_{\hspace{-.2cm}\includegraphics[scale=.5]{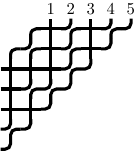}}$}
	& \multicolumn{6}{@{}l}{\coproduct \G_{31542}}
	\\[-.7cm]
	& 1 \otimes \G_{31542}
	& \G_1 \otimes \G_{2431}
	& \G_{12} \otimes \G_{132}
	& \G_{312} \otimes \G_{21}
	& \G_{3142} \otimes \G_1
	& \G_{31542} \otimes 1
	\\[.3cm]
	& 1 \otimes \QTwist\raisebox{.1cm}{$_{\hspace{-.2cm}\includegraphics[scale=.5]{exmCoproductDual/2twist31542}}$}
	& \QTwist\raisebox{.1cm}{$_{\hspace{-.2cm}\includegraphics[scale=.5]{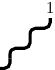}}$} \otimes \QTwist\raisebox{.1cm}{$_{\hspace{-.2cm}\includegraphics[scale=.5]{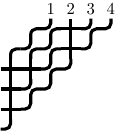}}$}
	& \QTwist\raisebox{.1cm}{$_{\hspace{-.2cm}\includegraphics[scale=.5]{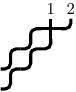}}$} \otimes \QTwist\raisebox{.1cm}{$_{\hspace{-.2cm}\includegraphics[scale=.5]{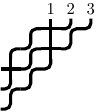}}$}
	& \QTwist\raisebox{.1cm}{$_{\hspace{-.2cm}\includegraphics[scale=.5]{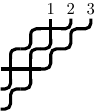}}$} \otimes \QTwist\raisebox{.1cm}{$_{\hspace{-.2cm}\includegraphics[scale=.5]{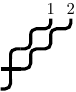}}$}
	& \QTwist\raisebox{.1cm}{$_{\hspace{-.2cm}\includegraphics[scale=.5]{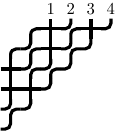}}$} \otimes \QTwist\raisebox{.1cm}{$_{\hspace{-.2cm}\includegraphics[scale=.5]{exmCoproductDual/2twist1}}$}
	& \QTwist\raisebox{.1cm}{$_{\hspace{-.2cm}\includegraphics[scale=.5]{exmCoproductDual/2twist31542}}$} \otimes 1
	\end{array}
	$}
	\caption{An example of coproduct in the dual $2$-twist algebra~$\Twist[2*]$.}
	\label{fig:exmCoproductDual}
\end{figure}

\subsection{Multiplicative bases and irreducible elements}
\label{subsec:multiplicativeBases}

In this section, we define multiplicative bases of~$\Twist$ and study the indecomposable elements of~$\Twist$ for these bases. For an acyclic $(k,n)$-twist~$\twist$, we define
\[
\ETwist^{\twist} \eqdef \sum_{\twist \le \twist'} \PTwist_{\twist'}
\qquad\text{and}\qquad
\HTwist^{\twist} \eqdef \sum_{\twist' \le \twist} \PTwist_{\twist'},
\]
where~$\le$ denotes the increasing flip lattice on acyclic $(k,n)$-twists. As the elements~$\ETwist^{\twist}$ and~$\HTwist^{\twist}$ have symmetric properties, we focus our analysis on~$\ETwist^{\twist}$. The reader is invited to translate the statements and proofs below to~$\HTwist^{\twist}$. We first observe that these elements can also be seen as elements of the multiplicative bases~$(\EFQSym^\tau)_{\tau \in \fS}$ and~$(\HFQSym^\tau)_{\tau \in \fS}$ of~$\FQSym$.

\begin{lemma}
\label{lem:indecomposableFQSymTwist}
For any acyclic $k$-twist~$\twist$, we have~$\ETwist^{\twist} = \EFQSym^\mu$ and~$\HTwist^{\twist} = \HFQSym^\omega$, where~$\mu$ and~$\omega$ respectively denote the weak order minimal and maximal permutations in the fiber of~$\twist$ under~$\surjectionPermBrick$.
\end{lemma}

\begin{proof}
We directly obtain from the definition that
\[
\ETwist^{\twist} = \sum_{\twist \le \twist'} \PTwist_{\twist'} = \sum_{\twist \le \twist'} \sum_{\substack{\tau' \in \fS_n \\ \surjectionPermBrick(\tau') = \twist'}} \F_{\tau'} = \sum_{\substack{\tau' \in \fS_n \\ \twist \le \surjectionPermBrick(\tau')}} \F_{\tau'} = \sum_{\substack{\tau' \in \fS_n \\ \mu \le \tau'}} \F_{\tau'} = \EFQSym^\mu. \qedhere
\]
\end{proof}

To describe the product of two elements of the~$\ETwist$- or $\HTwist$-basis, remember that the twist~$\underprod{\twist}{\twist'}$ (resp.~$\overprod{\twist}{\twist'}$) is obtained by inserting~$\twist$ in the first rows and columns of~$\twist'$ (resp.~$\twist'$ in the last rows and columns of~$\twist$). Examples are given in~\fref{fig:underOver}.

\begin{proposition}
\label{prop:productMultiplicativeBases}
$(\ETwist^{\twist})_{\twist \in \AcyclicTwists}$ and~$(\HTwist^{\twist})_{\twist \in \AcyclicTwists}$ are multiplicative bases of~$\Twist$:
\[
\ETwist^{\twist} \product \ETwist^{\twist'} = \ETwist^{\underprod{\twist}{\twist'}}
\qquad\text{and}\qquad
\HTwist^{\twist} \product \HTwist^{\twist'} = \HTwist^{\overprod{\twist}{\twist'}}.
\]
\end{proposition}

\begin{proof}
Let~$\mu$ and~$\mu'$ respectively denote the minimal elements of the fibers of~$\twist$ and~$\twist'$ under~$\surjectionPermBrick$. Using Lemma~\ref{lem:indecomposableFQSymTwist} and the fact that~$\surjectionPermBrick \big( \underprod{\mu}{\mu'} \big) = \underprod{\twist}{\twist'}$ and~$\underprod{\mu}{\mu'}$ is minimal in its $k$-twist congruence class, we write
\[
\ETwist^{\twist} \product \ETwist^{\twist'} = \EFQSym^{\mu} \product \EFQSym^{\mu'} = \EFQSym^{\underprod{\mu}{\mu'}} = \ETwist^{\underprod{\twist}{\twist'}}. \qedhere
\]
\end{proof}

\enlargethispage{.2cm}
We now consider multiplicative decomposability. We call \defn{cut} of an acyclic oriented graph any partition~$\edgecut{X}{Y}$ of its vertices such that all edges between~$X$ and~$Y$ are oriented from~$X$ to~$Y$.

\begin{proposition}
\label{prop:indecomposableCharacterization}
The following properties are equivalent for an acyclic $k$-twist~$\twist[S]$:
\begin{enumerate}[(i)]
\item $\ETwist^{\twist[S]}$ can be decomposed into a product~$\ETwist^{\twist[S]} = \ETwist^{\twist} \product \ETwist^{\twist'}$ for non-empty acyclic $k$-twists~$\twist, \twist'$; \label{enum:decomposable}
\item $\edgecut{[k]}{[n] \ssm [k]}$ is a cut of~$\twist[S]\contact$ for some~$k \in [n-1]$; \label{enum:cut}
\item at least one linear extension~$\tau$ of~$\twist[S]\contact$ is $\EFQSym$-decomposable, \ie $\tau([k]) = [k]$ for some~$k \in [n]$; \label{enum:perm}
\item the weak order minimal linear extension of~$\twist[S]\contact$ is $\EFQSym$-decomposable. \label{enum:minimalPerm}
\end{enumerate}
The $k$-twist~$\twist[S]$ is then called \defn{$\ETwist$-decomposable}. Otherwise, it is called \defn{$\ETwist$-indecomposable}, and we denote by~$\IndecomposableAcyclicTwists(n)$ the set of $\ETwist$-indecomposable acyclic $(k,n)$-twists.
\end{proposition}

\begin{proof}
The equivalence~\eqref{enum:decomposable} $\iff$ \eqref{enum:cut} is an immediate consequence of the description of the product~$\ETwist^{\twist} \product \ETwist^{\twist'} = \ETwist^{\underprod{\twist}{\twist'}}$ in Proposition~\ref{prop:productMultiplicativeBases}. The implication \eqref{enum:cut} $\implies$ \eqref{enum:perm} follows from the fact that for any cut~$\edgecut{X}{Y}$ of a directed acyclic graph~$G$, there exists a linear extension of~$G$ which starts with~$X$ and finishes with~$Y$. The implication \eqref{enum:perm} $\implies$ \eqref{enum:minimalPerm} follows from the fact that the $\EFQSym$-indecomposable permutations form an upper ideal of the weak order. Finally, if~$\tau$ is a decomposable linear extension of~$\twist[S]$, then the insertion algorithm on~$\tau$ first creates a twist labeled by~$[n] \ssm [k]$ and then inserts the pipes labeled by~$[k]$. Any arc between~$[k]$ and~$[n] \ssm [k]$ in~$\twist[S] = \surjectionPermBrick(\tau)$ will thus be directed from~$[k]$ to~$[n] \ssm [k]$.
\end{proof}

\begin{example}[Right-tilting $k$-twists]
Say that a $k$-twist is \defn{right-tilting} when it has no elbow in its first column. When~$k = 1$, the $\ETwist$-indecomposable $1$-twists are precisely the right-tilting $1$-twists. Therefore, the number of \mbox{$\ETwist$-indecomposable} $(1,n)$-twists is the Catalan number~$C_{n-1}$, and the $\ETwist$-indecomposable $(1,n)$-twists form a principal ideal of the increasing flip lattice. When~$k \ge 2$, right-tilting $k$-twists are $\ETwist$-indecomposable, but are not the only ones. The $\ETwist$-indecomposable $(k,n)$-twists form an upper ideal of the increasing flip lattice, but this ideal is not principal. \fref{fig:indecomposable} illustrates the $\ETwist$-indecomposable acyclic $(k,4)$-twists for~$k = 1,2$. 

\begin{figure}[t]
	\centerline{\includegraphics[width=\textwidth]{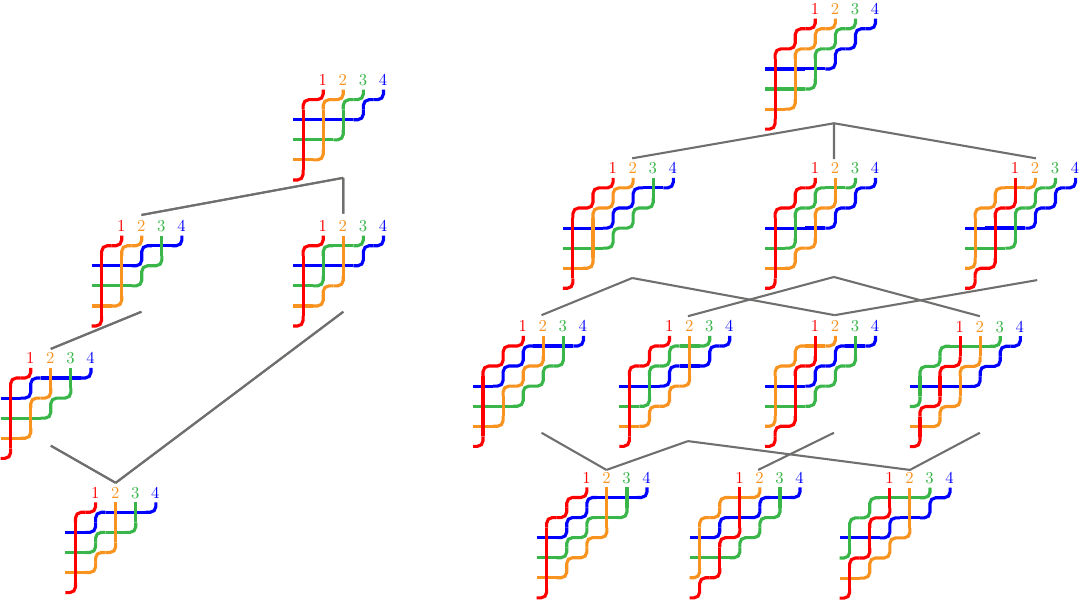}}
	\caption{The $\ETwist$-indecomposable acyclic $(k,4)$-twists for~$k = 1,2$.}
	\label{fig:indecomposable}
\end{figure}

\end{example}

\begin{proposition}
The $k$-twist algebra is freely generated by the elements~$\ETwist^{\twist}$ such that $\twist$ is \mbox{$\ETwist$-indecomposable}.
\end{proposition}

\begin{proof}
Let~$\twist$ be an acyclic $k$-twist and let~$\mu$ be the weak order minimal permutation such that~$\surjectionPermBrick(\mu) = \twist$. Decompose~$\mu = \underprod{\mu_1}{\underprod{\dots}{\mu_p}}$ into $\EFQSym$-indecomposable permutations~$\mu_1, \dots, \mu_p$. For~${i \in [p]}$, define~$\twist_i \eqdef \surjectionPermBrick(\mu_i)$. Since~$\mu_i$ avoids the patterns~$(k+2)1 \dash (\sigma_1+1) \dash \cdots \dash (\sigma_k+1)$ for all~$\sigma \in \fS_k$ (because~$\mu$ avoids these patterns), it is the weak order minimal permutation in the fiber of~$\twist_i$. Since~$\mu_i$ is $\EFQSym$-indecomposable, we get by Proposition~\ref{prop:indecomposableCharacterization}\,\eqref{enum:minimalPerm} that~$\twist_i$ is $\ETwist$-indecomposable. Using Lemma~\ref{lem:indecomposableFQSymTwist}, we thus obtained a decomposition~$\ETwist^{\twist} = \EFQSym^\mu = \EFQSym^{\mu_1} \cdots \EFQSym^{\mu_p} = \ETwist^{\twist_1} \cdots \ETwist^{\twist_p}$ of~$\twist$ into $\ETwist$-indecomposable $k$-twists~$\twist_1, \dots, \twist_p$.

Now, there is no relation between the elements~$\EFQSym^\tau$ of~$\FQSym$ corresponding to the $\EFQSym$-indecompo\-sable permutations. Hence, by Lemma~\ref{lem:indecomposableFQSymTwist} and Proposition~\ref{prop:productMultiplicativeBases}, there is no relation between the elements~$\ETwist^{\twist}$ of~$\Twist$ corresponding to the $\ETwist$-indecomposable $k$-twists.
\end{proof}

\begin{corollary}
\label{coro:numbersIndecomposableAcyclicTwists}
The generating functions of the numbers of $\ETwist$-indecomposable acyclic $(k,n)$-twists and of the numbers of all acyclic $(k,n)$-twists are related by
\[
\frac{1}{1 - \sum_{n \in \N} |\IndecomposableAcyclicTwists(n)| \, t^n} = \sum_{n \in \N} |\AcyclicTwists(n)| \, t^n.
\]
\end{corollary}

Table~\ref{table:numbersIndecomposableAcyclicTwists} gathers the numbers of $\ETwist$-indecomposable acyclic $(k,n)$-twists for~$k < n \le 10$. One recognizes on the diagonal the number of $\EFQSym$-indecomposable permutations of~$\fS_n$, see~\href{https://oeis.org/A003319}{\cite[A003319]{OEIS}}.

\begin{table}[h]
  \[
  	\begin{array}{l|rrrrrrrrrr}
	\raisebox{-.1cm}{$k$} \backslash \, \raisebox{.1cm}{$n$}
	  & 1 & 2 & 3 & 4 & 5 & 6 & 7 & 8 & 9 & 10 \\[.1cm]
	\hline
	0 & 1 & 1 & 1 &  1 &   1 &   1 &    1 &     1 &      1 &       1 \\
	1 & . & 1 & 2 &  5 & 14 &  42 &  132 &   429 &   1430 &    4862 \\
	2 & . & . & 3 & 11 & 47 & 219 & 1085 &  5619 &  30099 &  165555 \\
	3 & . & . & . & 13 & 65 & 365 & 2229 & 14465 &  98461 &  696337 \\
	4 & . & . & . &  . & 71 & 437 & 2967 & 21773 & 170047 & 1395733 \\
	5 & . & . & . &  . &  . & 461 & 3327 & 26213 & 222103 & 2000125 \\
	6 & . & . & . &  . &  . &   . & 3447 & 28373 & 253183 & 2419645 \\
	7 & . & . & . &  . &  . &   . &    . & 29093 & 268303 & 2668045 \\
	8 & . & . & . &  . &  . &   . &    . &     . & 273343 & 2789005 \\
	9 & . & . & . &  . &  . &   . &    . &     . &      . & 2829325
	\end{array}
  \]
  \caption{The number of $\ETwist$-indecomposable acyclic $(k,n)$-twists for~${k < n \le 10}$. Dots indicate that the value remains constant (equal to the number of $\ETwist$-indecomposable permutations) in the column.}
  \label{table:numbersIndecomposableAcyclicTwists}
  \vspace{-.4cm}
\end{table}

\enlargethispage{.3cm}
Again, the $\ETwist$-indecomposable $1$-twists are precisely the right-tilting $1$-twists, and are therefore counted by the Catalan number~$C_{n-1}$. Analogous results for~$k \ge 2$ remain to be found.

\begin{question}
Is there a simple characterization and a simple enumeration formula for the \mbox{$\ETwist$-indecomposable} acyclic $(k,n)$-twists?
\end{question}

\subsection{Integer point transform}
\label{subsec:integerPointTransform}

In this section, we observe that the product in the $k$-twist Hopf algebra~$\Twist$ can be interpreted in terms of the integer point transforms of the normal cones of the brick polytope~$\Brick$. To make this statement precise, we introduce some notations.

\begin{definition}
The \defn{integer point transform}~$\integerPointTransform_S$ of a subset~$S$ of~$\R^n$ is the multivariate generating function of the integer points inside~$S$:
\[
\integerPointTransform_S(t_1, \dots, t_n) = \sum_{(i_1, \dots, i_n) \in \Z^n \cap S} t_1^{i_1} \cdots t_n^{i_n}.
\]
\end{definition}

For a poset~$\less$, we denote by~$\integerPointTransform_\less$ the integer point transform of the cone
\[
\Cone\blackPolar(\less) \eqdef \set{\b{x} \in \R_+^n}{\begin{array}{c} x_i \le x_j \text{ for all } i \less j \text{ with } i < j \\  x_i < x_j \text{ for all } i \less j \text{ with } i > j \end{array}}.
\]
Note that this cone differs in two ways from the cone~$\Cone\polar(\less)$ defined in Section~\ref{subsec:normalFans}: first it leaves in~$\R_+^n$ and not in~$\HH$, second it excludes the facets of~$\Cone\polar(\less)$ corresponding to the decreasing relations of~$\less$ (\ie the relations~$i \less j$ with~$i > j$).

Following the notations of Section~\ref{subsec:normalFans}, we denote by~$\integerPointTransform_\tau$ the integer point transform of the chain~$\tau_1 \less \cdots \less \tau_n$ for a permutation~$\tau \in \fS_n$. The following statements are classical.

\begin{proposition}
\label{prop:integerPointTransform}
\begin{enumerate}[(i)]
\item For any permutation~$\tau \in \fS_n$, the integer point transform~$\integerPointTransform_\tau$ is given by
\[
\integerPointTransform_\tau(t_1, \dots, t_n) = \frac{\displaystyle \prod_{\substack{i \in [n-1] \\ \tau_i > \tau_{i+1}}} t_{\tau_i} \cdots t_{\tau_n}}{\displaystyle \prod\limits_{i \in [n]} \big( 1 - t_{\tau_i} \cdots t_{\tau_n} \big)}.
\]
\label{item:integerPointTransformPermutation}

\item The integer point transform of an arbitrary poset~$\less$ is given by
\[
\integerPointTransform_\less = \sum_{\tau \in \linearExtensions(\less)} \integerPointTransform_\tau,
\]
where the sum runs over the set~$\linearExtensions(\less)$ of linear extensions of~$\less$.
\label{item:integerPointTransformLinearExtensions}

\item The product of the integer point transforms~$\integerPointTransform_\tau$ and~$\integerPointTransform_{\tau'}$ of two permutations~${\tau \in \fS_n}$ and~${\tau' \in \fS_{n'}}$ is given by the shifted shuffle
\[
\integerPointTransform_\tau(t_1, \dots, t_n) \product \integerPointTransform_{\tau'}(t_{n+1}, \dots, t_{n+n'}) = \sum_{\sigma \in \tau \shiftedShuffle \tau'} \integerPointTransform_\sigma(t_1, \dots, t_{n+n'}).
\]
In other words, the linear map from~$\FQSym$ to the rational functions defined by~$\Psi : \F_\tau \mapsto \integerPointTransform_\tau$ is an algebra morphism.
\label{item:integerPointTransformProduct}
\end{enumerate}
\end{proposition}

\begin{proof}
For Point~\eqref{item:integerPointTransformPermutation}, we just observe that the cone~$\set{\b{x} \in \R_+^n}{x_{\tau_i} \le x_{\tau_{i+1}} \text{ for all } i \in [n-1]}$ is generated by the vectors~$\b{e}_{\tau_i} + \cdots + \b{e}_{\tau_n}$, for~$i \in [n]$, which form a (unimodular) basis of the lattice~$\Z^n$. A straightforward inductive argument shows that the integer point transform of the cone $\set{\b{x} \in \R_+^n}{x_{\tau_i} \le x_{\tau_{i+1}} \text{ for all } i \in [n-1]}$ is thus given by~$\prod_{i \in [n]} \big( 1- t_{\tau_i} \cdots t_{\tau_n} \big)^{-1}$. The numerator of~$\integerPointTransform_\tau$ is then given by the facets which are excluded from the cone~$\Cone\blackPolar(\tau)$.

Point~\eqref{item:integerPointTransformLinearExtensions} follows from the fact that the cone~$\Cone\blackPolar(\less)$ is partitioned by the cones~$\Cone\blackPolar(\tau)$ for the linear extensions~$\tau$ of~$\less$.

Finally, the product~$\integerPointTransform_\tau(t_1, \dots, t_n) \product \integerPointTransform_{\tau'}(t_{n+1}, \dots, t_{n+n'})$ is the integer point transform of the poset formed by the two disjoint chains~$\tau$ and~$\bar\tau'$, whose linear extensions are precisely the permutations which appear in the shifted shuffle of~$\tau$ and~$\tau'$. This shows Point~\eqref{item:integerPointTransformProduct}.
\end{proof}

For an acyclic $k$-twist~$\twist$, we denote by~$\integerPointTransform_{\twist}$ the integer point transform of the transitive closure of the contact graph~$\twist\contact$. It follows from Proposition~\ref{prop:integerPointTransform} that the product of the integer point transforms of two acyclic $k$-twists behaves as the product in the $k$-twist algebra~$\Twist$.

\begin{corollary}
For any two acyclic $k$-twists~$\twist, \twist'$, we have
\[
\integerPointTransform_{\twist}(t_1, \dots, t_n) \product \integerPointTransform_{\twist'}(t_{n+1}, \dots, t_{n+n'}) = \sum_{\underprod{\twist}{\twist'} \, \le \, \twist[S] \, \le \, \overprod{\twist}{\twist'}} \integerPointTransform_{\twist[S]}(t_1, \dots, t_{n+n'}).
\]
\end{corollary}

\begin{proof}
Hiding the variables~$(t_1, \dots, t_{n+n'})$ for concision, we have
\[
\integerPointTransform_{\twist} \product \integerPointTransform_{\twist'} = \Psi(\PTwist_{\twist}) \product \Psi(\PTwist_{\twist'}) = \Psi(\PTwist_{\twist} \cdot \PTwist_{\twist'}) = \Psi \bigg( \sum_{\twist[S]} \PTwist_{\twist[S]} \bigg) = \sum_{\twist[S]} \Psi(\PTwist_{\twist[S]}) = \sum_{\twist[S]} \integerPointTransform_{\twist[S]},
\]
where the sums run over the $k$-twists~$\twist[S]$ of the increasing flip lattice interval~$[\underprod{\twist}{\twist'}, \overprod{\twist}{\twist'}]$.
\end{proof}

\subsection{$k$-twistiform algebras}
\label{subsec:ktwistiform}

\enlargethispage{.3cm}
In this section, we extend the notion of dendriform algebras to $k$-twistiform algebras. Dendriform algebras were introduced by J.-L.~Loday in~\cite[Chap.~5]{Loday-dialgebras}. In a dendriform algebra, the product~$\product$ is decomposed into two partial products~$\op{l}$ and~$\op{r}$ satisfying:
\[
\sfx \; \op{l} \; \big(\sfy \product \sfz \big) = \big(\sfx \; \op{l} \; \sfy \big) \; \op{l} \; \sfz, \qquad\quad
\sfx \; \op{r} \; \big(\sfy \; \op{l} \; \sfz \big) = \big(\sfx \; \op{r} \; \sfy \big) \; \op{l} \; \sfz, \qquad\quad
\sfx \; \op{r} \; \big(\sfy \; \op{r} \; \sfz \big) = \big(\sfx \product \sfy \big) \; \op{r} \; \sfz.
\]
In our context, we will decompose the product of~$\FQSym$ (and of~$\Twist$) into~$2^k$ partial products satisfying~$3^k$ associativity relations (for the presentation, it is more convenient to use $3^k$ operations corresponding to all partial sums of the $2^k$ partial products). In this paper, we just give the definition and observe that the algebras~$\FQSym$ and~$\Twist$ are naturally endowed with a $k$-twistiform structure, as they motivated the definition. A detailed study of combinatorial and algebraic properties of $k$-twistiform algebras and operads is in progress in a joint work with F.~Hivert~\cite{HivertPilaud}.

We need to fix some natural notations on words. We denote by~$|W|$ the length of a word~$W$. For a subset~$P$ of positions in~$W$ and a subset~$L$ of letters of~$W$, we denote by~$W_P$ the subword of~$W$ consisting only of the letters at positions in~$P$ and by~$W^L$ the subword of~$W$ consisting only of the letters which belong to~$L$. 

\begin{definition}
\label{def:twistiform}
A \defn{$k$-twistiform algebra} is a vector space~$\algebra$ endowed with a collection~$\Operations \eqdef \{\op{l}, \op{m}, \op{r}\}^k$ of~$3^k$ bilinear operations which satisfy the following $k3^{k-1} + 3^k$ relations:
\begin{description}
\item[Split relations]
For any~$\operation[b],\operation[b'] \in \{\op{l}, \op{m}, \op{r}\}^*$ with~$|\operation[b]| + |\operation[b']| = k-1$, the operation~$\operation[b]\op{m}\operation[b'] \in \Operations$ decomposes into the operations~$\operation[b]\op{l}\operation[b] \in \Operations$ and~$\operation[b]\op{r}\operation[b] \in \Operations$:
\[
\sfx \; \operation[b]\op{m}\operation[b'] \,  \sfy \; = \; \sfx \; \operation[b]\op{l}\operation[b'] \, \sfy \; + \; \sfx \; \operation[b]\op{r}\operation[b'] \, \sfy \qquad\text{for all } \sfx, \sfy \in \algebra.
\]
\item[Associativity relations]
For any~$W \in \{x,y,z\}^k$, the operations~$\operation_W, \operation_W', \operation_W'', \operation_W''' \in \Operations$ defined by
\begin{gather*}
(\operation_W)_p \eqdef \begin{cases} \; \op{l} & \text{if } W_p = x \\ \; \op{r} & \text{if } W_p \in \{y,z\} \end{cases} \qquad
(\operation_W')_p \eqdef \begin{cases} \; \op{l} & \text{if } |W^{\{y,z\}}| \ge p \text{ and } (W^{\{y,z\}})_p = y \\ \; \op{r} & \text{if } |W^{\{y,z\}}| \ge p \text{ and } (W^{\{y,z\}})_p = z  \\ \; \op{m} & \text{otherwise} \end{cases} \\
(\operation_W'')_p \eqdef \begin{cases} \; \op{l} & \text{if } |W^{\{x,y\}}| \ge p \text{ and } (W^{\{x,y\}})_p = x \\ \; \op{r} & \text{if } |W^{\{x,y\}}| \ge p \text{ and } (W^{\{y,z\}})_p = y  \\ \; \op{m} & \text{otherwise} \end{cases} \qquad
(\operation_W''')_p \eqdef \begin{cases} \; \op{l} & \text{if } W_p \in \{x,y\} \\ \; \op{r} & \text{if } W_p = z \end{cases}
\end{gather*}
satisfy the associativity relation
\[
\sfx \; \operation_W \; \big(\sfy \; \operation_W' \; \sfz \big) = \big(\sfx \; \operation_W'' \; \sfy \big) \; \operation_W''' \; \sfz \qquad\text{for all } \sfx, \sfy, \sfz \in \algebra.
\]
\end{description}
\end{definition}

\begin{example}[$1$- and~$2$-twistiform algebras]
$1$-twistiform algebras are precisely dendriform algebras, \ie vector spaces endowed with three operations~$\op{l}, \, \op{m}, \, \op{r}$ which fulfill the~$4$ relations:
\[
\begin{array}{c@{\qquad}c@{\qquad}c}
& \sfx \; \op{m} \; \sfy \; = \; \sfx \; \op{l} \; \sfy \; + \; \sfx \; \op{r} \; \sfy, & \\[.1cm]
\sfx \; \op{l} \; \big(\sfy \; \op{m} \; \sfz \big) = \big(\sfx \; \op{l} \; \sfy \big) \; \op{l} \; \sfz, &
\sfx \; \op{r} \; \big(\sfy \; \op{l} \; \sfz \big) = \big(\sfx \; \op{r} \; \sfy \big) \; \op{l} \; \sfz, &
\sfx \; \op{r} \; \big(\sfy \; \op{r} \; \sfz \big) = \big(\sfx \; \op{m} \; \sfy \big) \; \op{r} \; \sfz. 
\end{array}
\]
$2$-twistiform algebras are vector spaces endowed with~$9$ operations $\op{l,l}$, $\op{l,m}$, $\op{l,r}$, $\op{m,l}$, $\op{m,m}$, $\op{m,r}$, $\op{r,l}$, $\op{r,m}$, $\op{r,r}$ which satisfy the following~$15$ relations: \\[.2cm]
\centerline{$
\begin{array}{@{}c@{\qquad}c@{\qquad}c@{}}
\sfx \; \op{m,l} \; \sfy \; = \; \sfx \; \op{l,l} \; \sfy \; + \; \sfx \; \op{r,l} \; \sfy, &
\sfx \; \op{m,m} \; \sfy \; = \; \sfx \; \op{l,m} \; \sfy \; + \; \sfx \; \op{r,m} \; \sfy, &
\sfx \; \op{m,r} \; \sfy \; = \; \sfx \; \op{l,r} \; \sfy \; + \; \sfx \; \op{r,r} \; \sfy, \\[.1cm]
\sfx \; \op{l,m} \; \sfy \; = \; \sfx \; \op{l,l} \; \sfy \; + \; \sfx \; \op{l,r} \; \sfy, &
\sfx \; \op{m,m} \; \sfy \; = \; \sfx \; \op{m,l} \; \sfy \; + \; \sfx \; \op{m,r} \; \sfy, &
\sfx \; \op{r,m} \; \sfy \; = \; \sfx \; \op{r,l} \; \sfy \; + \; \sfx \; \op{r,r} \; \sfy, \\[.1cm]
\sfx \; \op{l,l} \; \big(\sfy \; \op{m,m} \; \sfz \big) = \big(\sfx \; \op{l,l} \; \sfy \big) \; \op{l,l} \; \sfz, &
\sfx \; \op{l,r} \; \big(\sfy \; \op{l,m} \; \sfz \big) = \big(\sfx \; \op{l,r} \; \sfy \big) \; \op{l,l} \; \sfz, &
\sfx \; \op{l,r} \; \big(\sfy \; \op{r,m} \; \sfz \big) = \big(\sfx \; \op{l,m} \; \sfy \big) \; \op{l,r} \; \sfz, \\[.1cm]
\sfx \; \op{r,l} \; \big(\sfy \; \op{l,m} \; \sfz \big) = \big(\sfx \; \op{r,l} \; \sfy \big) \; \op{l,l} \; \sfz, &
\sfx \; \op{r,r} \; \big(\sfy \; \op{l,l} \; \sfz \big) = \big(\sfx \; \op{r,r} \; \sfy \big) \; \op{l,l} \; \sfz, &
\sfx \; \op{r,r} \; \big(\sfy \; \op{l,r} \; \sfz \big) = \big(\sfx \; \op{r,m} \; \sfy \big) \; \op{l,r} \; \sfz, \\[.1cm]
\sfx \; \op{r,l} \; \big(\sfy \; \op{r,m} \; \sfz \big) = \big(\sfx \; \op{l,m} \; \sfy \big) \; \op{r,l} \; \sfz, &
\sfx \; \op{r,r} \; \big(\sfy \; \op{r,l} \; \sfz \big) = \big(\sfx \; \op{r,m} \; \sfy \big) \; \op{r,l} \; \sfz, &
\sfx \; \op{r,r} \; \big(\sfy \; \op{r,r} \; \sfz \big) = \big(\sfx \; \op{m,m} \; \sfy \big) \; \op{r,r} \; \sfz.
\end{array}
$}

\vspace{.2cm}
\end{example}

\begin{remark}
Adding up all associativity relations, one obtains that
\[
\sfx \; \op{m}^k \; \big( \sfy \; \op{m}^k \; \sfz \big) = \big( \sfx \; \op{m}^k \; \sfy \big) \; \op{m}^k \; \sfz \qquad\text{for all } \sfx, \sfy, \sfz \in \algebra,
\]
so that the $k$-twistiform algebra~$(\algebra, \{\op{l}, \op{m}, \op{r}\}^k)$ defines in particular a structure of associative algebra~$(\algebra, \op{m}^k)$. Reciprocally, we say that an associative algebra~$(\algebra, \product)$ admits a \defn{$k$-twistiform structure} if it is possible to split the product~$\product$ into~$3^k$ operations~$\Operations \eqdef \{\op{l}, \op{m}, \op{r}\}^k$ defining a $k$-twistiform algebra on~$\algebra$.
\end{remark}

\enlargethispage{.1cm}
We now show that C.~Malvenuto and C.~Reutenauer's Hopf algebra on permutations~$\FQSym$ can be endowed with a structure of $k$-twistiform algebra. For an operation~$\operation \in \Operations$ and two words~$X \eqdef x\underline{X}$ and~$Y \eqdef y\underline{Y}$, we define
\[
X \; \operation \; Y =
\begin{cases}
X \shuffle Y & \text{if } \operation = \varnothing, \\
x (\underline{X} \; \underline{\operation} \; Y) & \text{if } \operation = \op{l}\underline{\operation}, \\
x (\underline{X} \; \underline{\operation} \; Y) \; \cup \; y (X \; \underline{\operation} \; \underline{Y}) & \text{if } \operation = \op{m}\underline{\operation}, \\
y (X \; \underline{\operation} \; \underline{Y}) & \text{if } \operation = \op{r}\underline{\operation}
\end{cases}
\]
with the initial conditions~$X \; \op{r}\underline{\operation} \; \varnothing = \varnothing \; \op{l}\underline{\operation} \; Y = 0$.

In other words, we consider the shuffle of~$X$ and~$Y$, except that the~$i$th letter of~$X \, \operation \, Y$ is forced to belong to~$X$ (resp.~to~$Y$) if the $i$th letter of~$\operation$ is~$\op{l}$ (resp.~is~$\op{r}$). For example, when~$k = 1$, the three operators are given by
\[
X \; \op{l} \; Y = x (\underline{X} \shuffle Y),
\qquad
X \; \op{m} \; Y = X \shuffle Y,
\qquad
X \; \op{r} \; Y = y (X \shuffle \underline{Y}),
\]
with the initial conditions~$X \; \op{r} \; \varnothing = \varnothing \; \op{l} \; Y = 0$.

Now for an operation~$\operation \in \Operations$ and two permutations~$\tau \in \fS_n$ and~$\tau' \in \fS_{n'}$, we define~${\tau \, \operation \, \tau' = \tau \, \operation \, \bar\tau'}$, where~$\bar\tau'$ is the permutation~$\tau'$ shifted by the length~$n$ of~$\tau$. Equivalently, $\tau \, \operation \, \tau'$ is the set of permutations~$\sigma \in \tau \shiftedShuffle \tau'$ such that for all~$i \in [k]$, we have~$\sigma_i \le n$ if~$\operation_i = \op{l}$ while~$\sigma_i > n$ if~$\operation_i = \op{r}$. Finally, we define the operations~$\Operations$ on the Hopf algebra~$\FQSym$ itself by
\[
\F_\tau \; \operation \; \F_{\tau'} = \sum_{\sigma \in \tau \, \operation \, \tau'} \F_\sigma.
\]

\begin{proposition}
The Hopf algebra~$\FQSym$, endowed with the operations~$\Operations$ described above, defines a $k$-twistiform algebra. The product of~$\FQSym$ is then given by~$\product = \op{m}^k$.
\end{proposition}

\begin{proof}
We have to show that the operations defined above on~$\FQSym$ indeed satisfy the $k3^{k-1}+3^k$ relations of Definition~\ref{def:twistiform}.
\begin{description}
\item[Split relations] Let~$\operation, \operation' \in \{\op{l}, \op{m}, \op{r}\}^*$ with~$|\operation| + |\operation'| = k-1$. It is immediate from the definitions that~$X \; \operation\op{m}\operation' \; Y = (X \; \operation\op{l}\operation' \; Y) \cup (X \; \operation\op{r}\operation' \; Y)$ for any two words~$X,Y$.

\smallskip
\item[Associativity relations] Let~$W \in \Operations$. It follows from the definition of the operations~$\operation_W$, $\operation_W'$, $\operation_W''$ and~$\operation_W'''$ that for any words~$X,Y,Z$
\[
X \; \operation_W \; (Y \; \operation_W' \; Z) = (X \; \operation_W'' \; Y) \; \operation_W''' \; Z
\]
is the set of all words in~$X \shuffle Y \shuffle Z$ whose $p$th letter is in the word~$X$ if~$W_p = x$, in the word~$Y$ if~$W_p = y$ and in the word~$Z$ if~$W_p = z$.
\end{description}
These equalities of sets then translate to the desired linear relations on the corresponding operations in~$\FQSym$.
\end{proof}

We say that the operations~$\Operations$ define the \defn{forward $k$-twistiform structure} on~$\FQSym$. There is also a \defn{backward $k$-twistiform structure} on~$\FQSym$ which considers the last~$k$ letters rather than the first~$k$ ones. Namely, for each operation~$\operation \in \Operations$, define  an operation~$\mirror{\operation}$ by~${V \; \mirror{\operation} \; W = \mirror{(\mirror{V} \; \operation \; \mirror{W})}}$ where~$\mirror{W} = w_n \cdots w_1$ denotes the mirror of a word~$W = w_1 \cdots w_n$. Clearly, the operations~$\mirror{\operation}$ for~$\operation \in \Operations$ still fulfill the relations of Definition~\ref{def:twistiform}. We have chosen to define the forward $k$-twistiform structure as it leads to a simpler presentation, but we need this backward $k$-twistiform structure in the next statement to be coherent with the insertion in $k$-twists (whose direction was chosen consistently with J.-L.~Loday and M.~Ronco's conventions).

\begin{proposition}
The subalgebra~$\Twist$ of~$\FQSym$ is stable by the operations~$\mirror{\operation}$ for~$\operation \in \Operations$ and therefore inherits a $k$-twistiform structure.
\end{proposition}

\begin{proof}
Let~$\twist$ be an acyclic $k$-twist. We claim that the last~$k$ entries are the same in all linear extensions of~$\twist\contact$. Indeed, pick a linear extension~$\tau$ of~$\twist\contact$ and let~$\twist' \eqdef \varnothing \insertion{} \tau_n \insertion{} \cdots \insertion{} \tau_{n-k+1}$ denote the $k$-twist obtain after the insertion of the last~$k$ values of~$\tau$. All pipes of~$\twist'$ are then comparable by Lemma~\ref{lem:canopy} and will be comparable to all other pipes in~$\twist = \twist' \insertion{} \tau_{n-k} \insertion{} \cdots \insertion{} \tau_1$. Thus the last $k$ values of~$\tau$ form a chain at the end of the contact graph~$\twist\contact$.

It then follows that the operation~$\mirror{\operation}$ stabilizes~$\Twist$ for any~$\operation \in \Operations$. Indeed, for any two acyclic twists~$\twist \in \AcyclicTwists(n)$ and~$\twist' \in \AcyclicTwists(n')$, we have
\[
\PTwist_{\twist} \; \mirror{\operation} \; \PTwist_{\twist'} = \sum_{\twist[S]} \PTwist_{\twist[S]}
\]
where the sum runs over all acyclic twists~$\twist[S] \in \AcyclicTwists(n+n')$ such that~$\underprod{\twist}{\twist'} \le \twist[S] \le \overprod{\twist}{\twist'}$ and~$\sigma_{n+n'+1-i} \le n$ if~$\operation_i = \op{l}$ and~$\sigma_{n+n'+1-i} > n$ if~$\operation_i = \op{r}$ for any linear extension~$\sigma$ of~$\twist[S]\contact$.
\end{proof}

\begin{remark}
One can also define similarly $k$-cotwistiform coalgebras, and such a structure on both~$\FQSym$ and~$\Twist$. Details will be given in~\cite{HivertPilaud}.
\end{remark}

\part{Three extensions}
\label{part:extensions}

In the second part of this paper, we present three independent extensions of the combinatorial, geometric, and algebraic constructions of Part~\ref{part:acyclicTwists}:
\begin{description}
\item[Cambrianization] Following the same direction as~\cite{Reading-CambrianLattices, HohlwegLange, PilaudSantos-brickPolytope, ChatelPilaud}, we show that our constructions can be parametrized by a sequence of signs (\ie a type~$A$ Coxeter element). We obtain generalizations of the Cambrian lattices, mention their connections to certain well-chosen brick polytopes, and construct a Cambrian twist Hopf algebra with similar ideas as in~\cite[Part~1]{ChatelPilaud}. \\[-.3cm]

\item[Tuplization] Motivated by the Hopf algebras on diagonal rectangulations described by S.~Law and N.~Reading~\cite{LawReading} and on twin binary trees described by S.~Giraudo~\cite{Giraudo}, G.~Chatel and V.~Pilaud \cite[Part~2]{ChatelPilaud} defined Cambrian tuple algebras. We briefly extend this construction to Cambrian~twist tuples. \\[-.3cm]

\item[Schr\"oderization] In~\cite{Chapoton}, F.~Chapoton described Hopf algebras on all faces of the permutahedra, associahedra and cubes. This was extended to the Cambrian algebras in~\cite[Part~3]{ChatelPilaud}. We show that these algebras and the corresponding combinatorics can be translated as well to all faces of the brick polytopes and zonotopes.
\end{description}
These extensions can moreover be combined: we could present in a single structure all these generalizations. For the sake of clarity, we have preferred to present these extensions separately, and we leave it to the reader to combine them.

\section{Cambrianization}
\label{sec:Cambrianization}

Our first extension concerns Cambrian twists, which are natural generalizations of twists parametrized by the choice of a signature~$\signature \in \pm^n$. When~$k = 1$, Cambrian $1$-twists correspond to Cambrian trees~\cite{ChatelPilaud}, whose 
\begin{itemize}
\item combinatorial properties are encoded in N.~Reading's Cambrian lattices~\cite{Reading-CambrianLattices}, 
\item geometric properties are encoded in C.~Hohlweg and C.~Lange's realizations of the associahedron~\cite{HohlwegLange}, and
\item algebraic properties are encoded in G.~Chatel and V.~Pilaud's Cambrian algebras~\cite{ChatelPilaud}.
\end{itemize}
We show how these three structures extend to arbitrary~$k$.

\subsection{Combinatorics of Cambrian twists}
\label{subsec:combinatoricsCambrian}

Given a signature~${\signature \eqdef \signature_1 \cdots \signature_n \in \pm^n}$, we denote by~$|\signature|_-$ and~$|\signature|_+$ the number of $-$ and $+$ in~$\signature$, respectively. For any integer~$k \in \N$ and signature~${\signature \in \pm^n}$, we define a \defn{shape}~$\shape$ formed by four monotone lattices paths:
\begin{enumerate}[(i)]
\item \defn{enter path}: from~$(|\signature|_+, 0)$ to~$(0, |\signature|_-)$ with $p$th step north if~$\signature_p = -$ and west if~$\signature_p = +$,
\item \defn{exit path}: from~$(|\signature|_+ + k, n + k)$ to~$(n + k, |\signature|_- + k)$ with $p$th step east if~$\signature_p = -$~and~south~if~${\signature_p = +}$,
\item \defn{accordion paths}:
\begin{tabular}[t]{@{}l}
the path~$(NE)^{|\signature|_+ + k}$ from~$(0, |\signature|_-)$ to~$(|\signature|_+ + k, n + k)$ and \\
the path~$(EN)^{|\signature|_- + k}$ from~$(|\signature|_+, 0)$ to~$(n + k, |\signature|_- + k)$.
\end{tabular}
\end{enumerate}
We call \defn{$p$th entry} (resp.~\defn{$p$th exit}) of~$\shape$ the $p$th step of its enter path (resp.~exit path).
\fref{fig:shapes} illustrates these families of shapes for different integers~$k$ and signatures~$\signature$.

\begin{figure}[h]
	\centerline{\includegraphics[scale=1.3]{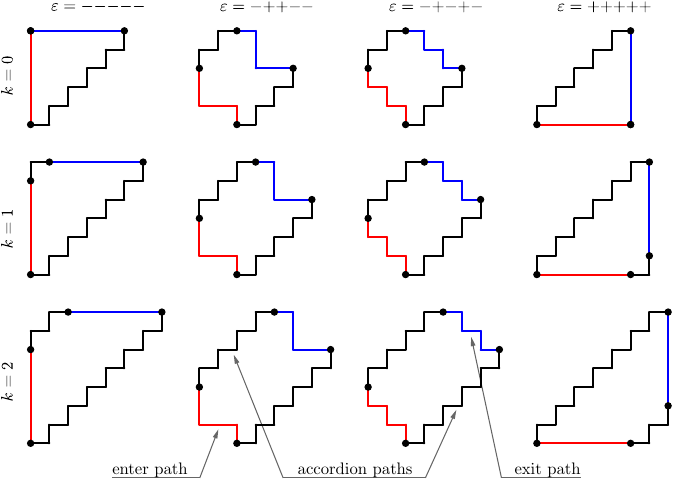}}
	\caption{The shapes~$\shape$ for~$k \in \{0,1,2\}$ and different signatures~$\signature$. The enter path is in red, the exit path is in blue, and the accordion paths are in black.}
	\label{fig:shapes}
\end{figure}

\begin{definition}
A \defn{Cambrian $(k, \signature)$-twist} is a filling of the shape~$\shape$ with crosses~\cross{} and elbows~\elbow{} such that
\begin{itemize}
\item no two pipes cross twice (the pipe dream is reduced),
\item the pipe which enters at the $p$th entry exits at the $p$th exit, and is called the $p$th pipe.
\end{itemize}
\end{definition}

\begin{example}[Constant signatures]
The Cambrian $(k,-^n)$-twists are precisely the $(k,n)$-twists. Similarly, the Cambrian $(k,+^n)$-twists are obtained from the $(k,n)$-twists by reflection with respect to the diagonal~$x=y$.
\end{example}

\enlargethispage{-.2cm}
Contact graphs of Cambrian twists are defined precisely as for classical twists in Definition~\ref{def:contactGraph} of Section~\ref{subsec:twists}.

\begin{definition}
The \defn{contact graph} of a Cambrian $(k,\signature)$-twist~$\twist$ is the multigraph~$\twist\contact$ with vertex set~$[n]$ and with an arc from the \SE-pipe to the \WN-pipe of each elbow in~$\twist$. A Cambrian twist~$\twist$ is \defn{acyclic} if its contact graph~$\twist\contact$ is (no oriented cycle). We denote by~$\AcyclicTwists(\signature)$ the set of acyclic Cambrian~$(k,\signature)$-twists.
\end{definition}

\fref{fig:twistsCambrian} illustrates examples of Cambrian $(k,\signature)$-twists and their contact graphs for~$k = 0, 1, 2, 3$ and the signature~$\signature = {-}{+}{+}{-}{-}$. The first two are acyclic, the last two are not.

\begin{figure}[h]
	\centerline{\includegraphics[scale=1.4]{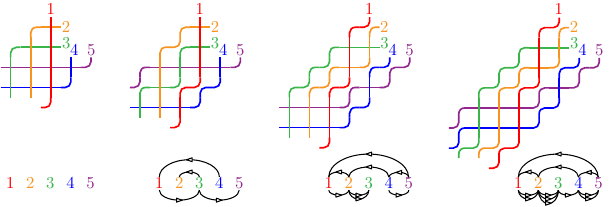}}
	\caption{Cambrian $(k,\signature)$-twists and their contact graphs for~$k = 0, 1, 2, 3$ and~$\signature = {-}{+}{+}{-}{-}$.}
	\label{fig:twistsCambrian}
\end{figure}

\begin{remark}[Subword complexes]
Following Remark~\ref{rem:subwordComplexes}, a Cambrian $(k,\signature)$-twist could be equivalently defined as a reduced expression of the longest permutation~$w_\circ \eqdef [n,n-1, \dots, 2,1]$ in the word $\Q_\signature \eqdef \sqc_\signature^k w_\circ(\sqc_\signature)$, where~$\sqc_\signature$ is a reduced expression for the type~$A$ Coxeter element given by the signature~$\signature$ (meaning that~$s_i$ appears before~$s_{i+1}$ in~$\sqc_\signature$ if~$\signature_i = -$, and after~$s_{i+1}$ if~$\signature_i = +$).
\end{remark}

\begin{example}[Cambrian $1$-twists, triangulations and Cambrian trees]
\label{exm:1twistsTriangulationsCambrian}
Following Example~\ref{exm:1twistsTriangulations}, Cambrian $(1,\signature)$-twists are in bijective correspondence with triangulations of a convex $(n+2)$-gon. This bijection, illustrated in \fref{fig:1twistsTriangulationsCambrian}, is similar to that described in Example~\ref{exm:1twistsTriangulations}, except that the polygon is replaced by the $(n+2)$-gon $\polygon$ whose vertices are labeled by~$\{0, \dots, n+1\}$ from left to right and where vertex~$i$ is located above the diagonal~$[0,n+1]$ is~$\signature_i = +$ and below it if~$\signature_i = -$. See~\cite{PilaudPocchiola, PilaudSantos-brickPolytope}. The contact graph of a Cambrian $(1,\signature)$-twist~$\twist$ coincides with the dual tree of~$\twist\duality$ where each node is labeled by the middle vertex of the corresponding triangle, and each arc is oriented from the triangle below to the triangle above the corresponding diagonal. These trees were called ``spines'' in~\cite{LangePilaud}, ``mixed cobinary trees'' in~\cite{IgusaOstroff}, and ``Cambrian trees''~in~\cite{ChatelPilaud}.

\begin{figure}[h]
	\centerline{\includegraphics[scale=1.4]{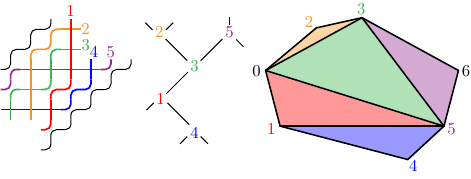}}
	\caption{The bijection between Cambrian $(1,\signature)$-twists~$\twist$ (left) and triangulations~$\twist\duality$ of~$\polygon$ (right) sends the pipes of~$\twist$ to the triangles of~$\twist\duality$, the elbows of~$\twist$ to the diagonals of~$\twist\duality$, the contact graph of~$\twist$ to the dual Cambrian tree of~$\twist\duality$ (middle), and the elbow flips in~$\twist$ to the diagonal flips~in~$\twist\duality$.}
	\label{fig:1twistsTriangulationsCambrian}
	\vspace*{-.4cm}
\end{figure}
\end{example}

\begin{remark}[Cambrian $k$-twists and $k$-triangulations]
In fact, one can simultaneously as well extend the results of Section~\ref{subsec:ktriangulations}. Namely, set~$\bar\signature = (-)^{k-1}\signature (-)^{k-1}$. Choosing~$\shape$ instead of the triangular shape and~$\polygon[\bar\signature]$ instead of the standard convex $(n+2)$-gon, the map described in Theorem~\ref{theo:ktwistsktriangulations} defines a bijective correspondence between the Cambrian $(k,\signature)$-twists and the \mbox{$k$-triangulations} of~$\polygon[\bar\signature]$. Details can be found in~\cite{PilaudPocchiola, PilaudSantos-brickPolytope}.
\end{remark}

\begin{remark}[Numerology]
\begin{table}[b]
  \centerline{$
  	\begin{array}{l|rrrrrrrrrr}
	\raisebox{-.1cm}{$k$} \backslash \, \raisebox{.1cm}{$n$}
	  & 1 & 2 & 3 & 4 & 5 & 6 & 7 & 8 & 9 & 10 \\[.1cm]
	\hline
	0 & 1 & 1 & 1 &  1 &   1 &   1 &    1 &     1 &      1 &       1 \\
	1 & . & 2 & 5 & 14 &  42 & 132 &  429 &  1430 &   4862 &   16796 \\
	2 & . & . & 6 & 24 & 114 & 608 & 3532 & 21950 & 143776 &  982324 \\
	3 & . & . & . &  . & 120 & 720 & 4920 & 37104 & 303072 & 2643156 \\
	4 & . & . & . &  . &   . &   . & 5040 & 40320 & 357840 & 3453120 \\
	5 & . & . & . &  . &   . &   . &    . &     . & 362880 & 3628800 \\
	\end{array}
  $}
  \caption{The number~$|\AcyclicTwists((+-)^{n/2})|$ of acyclic Cambrian $(k,(+-)^{n/2})$-twists for~$2k \le n \le 10$. Dots indicate that the value remains constant (equal to~$n!$) in the column.}
  \label{table:numbersAcyclicTwistsCambrian}
  \vspace{-.4cm}
\end{table}
According to Theorem~\ref{theo:numberktriangulations} in Section~\ref{subsec:numerology}, the previous remark ensures that the number of Cambrian $(k,\signature)$-twists is the Hankel determinant~$\det(C_{n+2k-i-j})_{i,j \in [k]}$, no matter the signature~$\signature \in \pm^n$. In contrast, the number of acyclic Cambrian $(k,\signature)$-twists depends on the signature~$\signature$. For example~$|\AcyclicTwists[2]({-}{-}{-}{-})| = 22 \ne 24 = |\AcyclicTwists[2]({+}{-}{+}{+})|$. Table~\ref{table:numbersAcyclicTwistsCambrian} gathers the numbers of acyclic Cambrian $(k,(+-)^{n/2})$-twists for~$2k \le n \le 10$. Comparing this table to Table~\ref{table:numbersAcyclicTwists} raises the question to determine which are the numbers of Cambrian $(k,\signature)$-twists for all possible signatures~$\signature \in \pm^n$. We conjecture in particular that for any signature~$\signature \in \pm^n$, the number of acyclic Cambrian $(k,\signature)$-twists is always bounded between the number of acyclic $(k,n)$-twists and the number of acyclic Cambrian $(k,(+-)^{n/2})$-twists.
\end{remark}

As in Section~\ref{subsec:elementaryPropertiesPipes}, we need to establish some elementary properties of pipes in Cambrian twists. The properties of Lemma~\ref{lem:pipeProperties} still hold for Cambrian twists, except that a negative pipe has $k$ \SE-elbows and $k+1$ \WN-elbows, while a positive pipe has $k+1$ \SE-elbows and $k$ \WN-elbows. In contrast, Lemma~\ref{lem:comparable} does not hold for Cambrian twists: for example, in the Cambrian $1$-twist of \fref{fig:twistsCambrian}, the $2$nd and $5$th pipe are not comparable even if the $5$th pipe has an elbow south-east of an elbow of the $2$nd pipe. We need to refine the statement as follows. We call \defn{zone} of a pipe~$\pipe$ the boxes located in between the entry and exit points of~$\pipe$. In other words, it is the zone of the possible elbows of~$\pipe$. For example, \fref{fig:zones} illustrates the zone of the third pipe in the shapes~$\shape[2]$ of \fref{fig:shapes}.

\begin{figure}[t]
	\centerline{\includegraphics[scale=1.4]{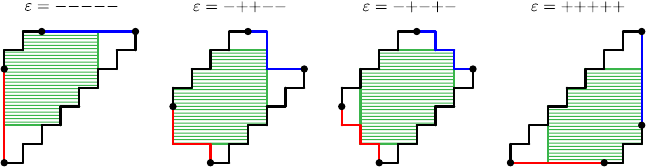}}
	\caption{The zone of the third pipe the shapes~$\shape[2]$ of \fref{fig:shapes}.}
	\label{fig:zones}
\end{figure}

\begin{lemma}
\label{lem:comparableCambrian}
Consider two pipes~$\pipe, \pipe'$ in a Cambrian $k$-twist~$\twist$. If there is an elbow of~$\pipe$ below~$\pipe'$ (resp.~above~$\pipe'$) and in the zone of~$\pipe'$, then there is a path from~$\pipe$ to~$\pipe'$ (resp.~from~$\pipe'$ to~$\pipe$) in the contact graph~$\twist\contact$ of~$\twist$.
\end{lemma}

\begin{proof}
The proof is essentially the same as that of Lemma~\ref{lem:comparable}. Up to adding one more edge in the beginning of the path, we can assume that $\pipe$ has a \WN-elbow~$\elb$ below and in the zone of~$\pipe'$. We claim that there exists a sequence~$\elb = \elb_0, \elb_1, \dots, \elb_\ell$ of elbows, all located below and in the zone of~$\pipe'$, such that the elbows~$\elb_i$ and~$\elb_{i+1}$ are connected by a pipe~$\pipe_{i+1}$, and where~$\elb_\ell$ is an elbow of~$\pipe'$. Suppose that the elbows~$\elb = \elb_0, \elb_1, \dots, \elb_i$ are constructed and consider the \WN-pipe~$\pipe_{i+1}$ at~$\elb_i$. Either the previous or the next elbow along~$\pipe_{i+1}$ is still below and in the zone of~$\pipe'$. Otherwise, the pipe~$\pipe_{i+1}$ would cross the pipe~$\pipe'$ twice, since~$\elb_i$ is in the zone of~$\pipe'$. Choose~$\elb_{i+1}$ accordingly (pick arbitrarily one if the two options are possible). The process ends on an elbow~$\elb_\ell$ of~$\pipe'$ since the distance to~$\pipe'$ decreases at each step. Finally, the sequence of pipes~$\pipe = \pipe_1, \dots, \pipe_\ell = \pipe'$ gives a path from~$\pipe$ to~$\pipe'$ in the contact graph~$\twist\contact$. The proof is similar if~$\pipe$ has an elbow above and in the zone of~$\pipe'$.
\end{proof}

Similar to Section~\ref{subsec:pipeInsertionDeletion}, there is a natural insertion (resp.~deletion) operation which transforms a Cambrian $(k,\signature)$-twist~$\twist$ into a Cambrian $(k,\signature')$-twist~$\twist'$ where~$\signature$ and~$\signature'$ differ by the insertion (resp.~deletion) of a single sign. The insertion is defined as in Definition~\ref{def:insertion}, except that it now depends on whether the inserted pipe is negative or positive. Namely, the inserted pipe becomes a source in~$\twist'{}\contact$ when the new sign is negative, but a sink in~$\twist'{}\contact$ when the new sign is positive. These insertions are illustrated in \fref{fig:pipeInsertionCambrian}. Conversely, the deletion is defined as in Definition~\ref{def:deletion} except that a negative (resp.~positive) pipe of a $(k,\signature)$-twist~$\twist$ can be deleted only when it is a source (resp.~sink) of the contact graph~$\twist\contact$.

We will use these insertion and deletion operations to simplify some proofs in this section. However, we cannot use repeated insertions to define the Cambrian $k$-twist correspondence below. Indeed, we would need an insertion operation where the inserted pipe is always a source in the new twist, and such an operation is not easy to describe when the inserted pipe is positive. We will thus prefer to avoid repeated single insertions and to insert pipes directly in the final shape.

\begin{figure}[t]
	\centerline{\includegraphics[scale=1.4]{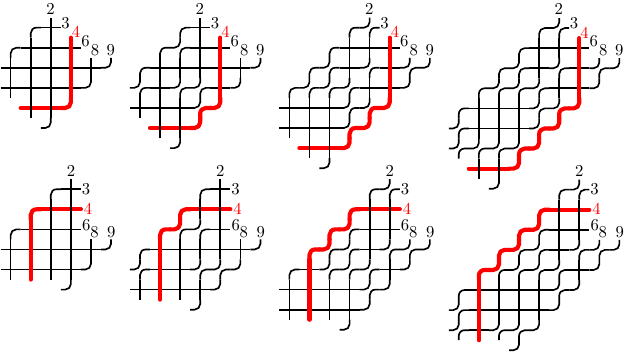}}
	\caption{Inserting~$4$ as a negative (top) or positive (bottom) pipe in the $(k,{-}{+}{+}{-}{-})$-twists of \fref{fig:twistsCambrian}. The inserted pipe is in bold red.}
	\label{fig:pipeInsertionCambrian}
\end{figure}

\bigskip
Similar to Section~\ref{subsec:twistCorrespondence}, we now define the \defn{Cambrian $k$-twist correspondence}. Following the presentation in~\cite{ChatelPilaud}, we work with signed permutations. A \defn{signed permutation} is a permutation~$\tau \in \fS_n$ endowed with a signature. It is convenient to think about these signs to be associated either to the positions or to the values of~$\tau$. We call \defn{p-signature} (resp.~\defn{v-signature}) of~$\tau$ the sequence~$\psignature(\tau)$ (resp.~$\vsignature(\tau)$) of signs of~$\tau$ ordered by positions (resp.~values). In concrete examples, we underline negative positions/values while we overline positive positions/values. For example, we write~$\up{3}\down{154}\up{2}$ for the permutation~${\tau = 31542}$ endowed with the signatures~${\psignature = {+}{-}{-}{-}{+}}$ and~${\vsignature = {-}{+}{+}{-}{-}}$. For a signature~$\signature \in \pm^n$, we denote by~$\fS_\signature$ (resp.~by~$\fS^\signature$) the set of signed permutations~$\tau$ with p-signature~$\psignature(\tau) = \signature$ (resp.~with v-signature~$\vsignature(\tau) = \signature$).

We now define a surjection from the signed permutations of~$\fS^\signature$ to the Cambrian $(k,\signature)$-twists. Consider a permutation~$\tau \in \fS^\signature$. Starting from the empty shape~$\shape$, we construct a Cambrian $(k,\signature)$-twist~$\surjectionPermBrick(\tau)$ by inserting the pipes~$\tau_n, \dots, \tau_1$ one by one such that each new pipe is as northwest as possible in the space left by the pipes already inserted. This procedure is illustrated in \fref{fig:insertionAlgorithmCambrian} for the signed permutation~$\up{3}\down{154}\up{2}$ and different values of~$k$. Observe that~$\signature$ does not appear in the notation~$\surjectionPermBrick(\tau)$ as this information is already carried by the signed permutation~${\tau \in \fS^\signature}$.

\hvFloat[floatPos=p, capWidth=h, capPos=r, capAngle=90, objectAngle=90, capVPos=c, objectPos=c]{figure}
{\includegraphics[scale=1.08]{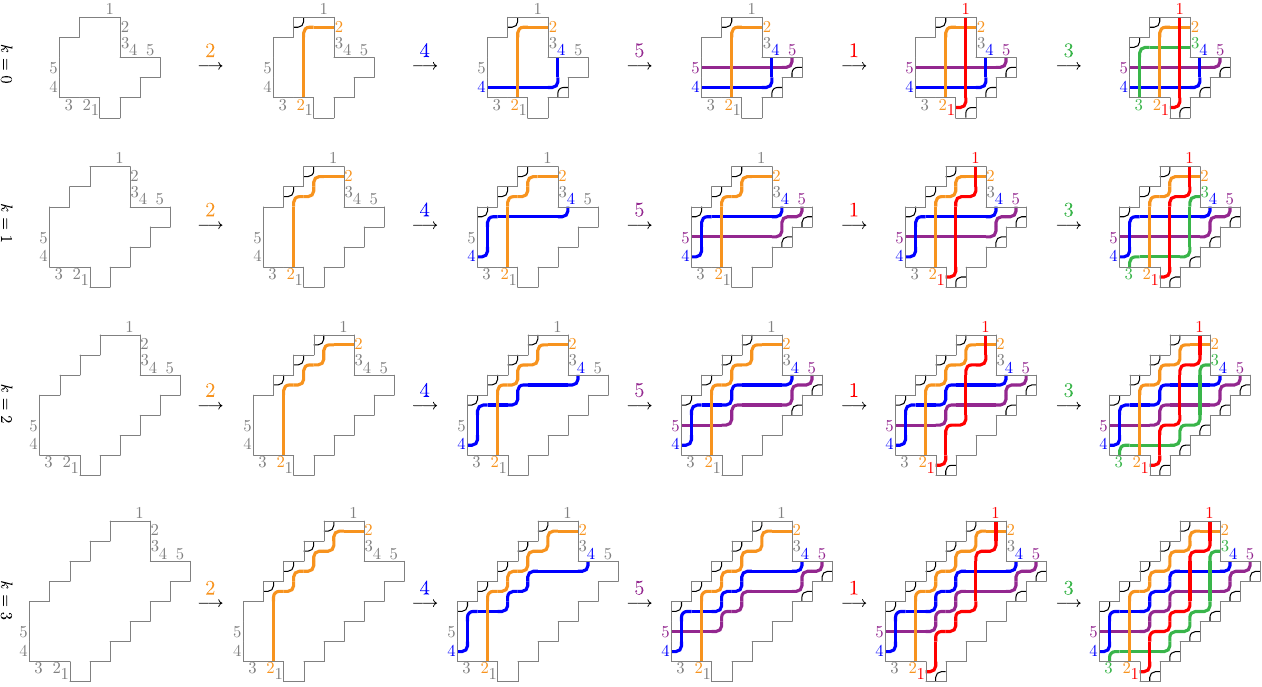}}
{Insertion of the signed permutation~$\up{3}\down{154}\up{2}$ in Cambrian $k$-twists for~$k = 0,1,2,3$.}
{fig:insertionAlgorithmCambrian}

\begin{proposition}
\label{prop:ktwistInsertionCambrian}
For any Cambrian $(k,\signature)$-twist~$\twist$, the permutations~$\tau \in \fS^\signature$ such that~${\surjectionPermBrick(\tau) = \twist}$ are precisely the linear extensions of the contact graph of~$\twist$. In particular, $\surjectionPermBrick$ is a surjection from the permutations of~$\fS^\signature$ to the acyclic Cambrian $(k,\signature)$-twists.
\end{proposition}

\begin{proof}
During the insertion algorithm, the last inserted pipe is always a source of the contact graph of the pipes already inserted. Otherwise, there would be an elbow~$\elb$ where the last inserted pipe~$\pipe$ is above a pipe~$\pipe[q]$ already inserted. Flipping this elbow~$\elb$ would transform these two pipes~$\pipe$ and~$\pipe[q]$ into two new pipes~$\pipe'$ and~$\pipe[q]'$ compatible with all already inserted pipes, and the pipe~$\pipe[q]'$ would be northwest of~$\pipe[q]$, contradicting the definition of~$\pipe[q]$. Now, since each insertion creates a source of the contact graph, it is immediate that any signed permutation~$\tau$ is a linear extension of the contact graph of~$\surjectionPermBrick(\tau)$.

Conversely, consider a Cambrian $(k, \signature)$-twist~$\twist$ and a linear extension~$\tau$ of~$\twist\contact$. We show by induction that the pipe inserted at step~$n+1-i$ in~$\surjectionPermBrick(\tau)$ coincide with the $\tau_i$th pipe of~$\twist$. Indeed, since~$\tau$ is a linear extension of~$\twist\contact$, no \SE-elbow of the $\tau_i$th pipe of~$\twist$ can touch a \WN-elbow of the $\tau_j$th pipe of~$\twist$ for~$j < i$. In other words, the $\tau_i$th pipe of~$\twist$ is the most possible northwest pipe in the space left by the pipes~$\tau_n, \dots, \tau_{i+1}$ of~$\twist$. The latter coincide by induction with the first~$n-i$ inserted pipes in~$\surjectionPermBrick(\tau)$, so that the $\tau_i$th pipe of~$\twist$ indeed coincides with the pipe inserted at step~$n+1-i$ in~$\surjectionPermBrick(\tau)$. We conclude that~$\surjectionPermBrick(\tau) = \twist$.
\end{proof}

\begin{example}[Cambrian $1$-twist correspondence and Cambrian correspondence]
When~${k = 1}$, the map~$\surjectionPermBrick[1]$ can also be described directly on Cambrian trees as in~\cite{ChatelPilaud}. This description is even more visual since the nodes of the Cambrian tree~$\surjectionPermBrick[1](\tau)$ form the permutation table of~$\tau$.
\end{example}

Our next statement extends Proposition~\ref{prop:intervals} to Cambrian $k$-twists, using again the characterization of weak order intervals of Definition~\ref{def:WOIP} and Proposition~\ref{prop:WOIP} given in Section~\ref{subsec:twistClasses}.

\begin{proposition}
\label{prop:intervalsCambrian}
The fiber of any Cambrian $k$-twist under~$\surjectionPermBrick$ is an interval of the weak order.
\end{proposition}

\begin{proof}
The proof is essentially the same as that of Proposition~\ref{prop:intervals}. We consider a Cambrian $(k,\signature)$-twist~$\twist$ and check that it fulfills the WOIP criterion of Definition~\ref{def:WOIP}. Assume that~$a < b < c$ are such that~$a \contactLess{\twist} c$. Since~$a \contactLess{\twist} c$, there is a path~$\pi$ from the $a$th entry to the $c$th exit of~$\shape$ which travels along pipes and elbows of~$\twist$. Since~$a < b < c$, the path~$\pi$ starts (weakly) south and ends (weakly) west of the zone of the $b$th pipe. Therefore, it must have an elbow~$\elb$ in the zone of the $b$th pipe, which ensures by Lemma~\ref{lem:comparableCambrian} that either~$a \contactLess{T} b$ (if~$\elb$ is above the $b$th pipe), or~$b \contactLess{T} c$ (if~$\elb$ is below the~$b$th~pipe). The proof is similar if~$a \contactMore{T} c$.
\end{proof}

Our next step is to extend the $k$-twist congruence of Section~\ref{subsec:twistCongruence}. In other words, we characterize the fibers of the insertion map~$\surjectionPermBrick$ by rewriting rules.

\begin{definition}
\label{def:ktwistCongruenceCambrian}
The \defn{Cambrian $(k, \signature)$-twist congruence} is the equivalence relation~$\equiv^k_\signature$ on~$\fS^\signature$ defined as the transitive closure of the rewriting rule~$U ac V \equiv^k_\signature U ca V$ if there exist~$a < a' < c' < c$ such~that
\begin{align*}
				& |\set{u \in U}{a' \le u \le c', \signature_u = +}| - |\set{v \in V}{a' \le v \le c', \signature_v = -}| \ge k, \\
\text{or}\quad 	& |\set{v \in V}{a' \le v \le c', \signature_v = -}| - |\set{u \in U}{a' \le u \le c', \signature_u = +}| \ge k,
\end{align*}
where~$a, c$ are elements of~$[n]$ while~$U, V$ are (possibly empty) words on~$[n]$.
\end{definition}

This condition can as well be reformulated in terms of congruence witnesses. Namely, we have $UacV \equiv^k_\signature UcaV$ if their exist~$p \ge 0$ and~$k+p$ witnesses~$a < b_1 < \dots < b_{k+p} < c$ such that
\begin{itemize}
\item either~$\signature_{b_i} = +$, $b_i \in U$ and~$|\set{v \in V}{b_1 \le v \le b_{k+p}, \signature_v = -}| = p$,
\item or \quad\; $\signature_{b_i} = -$, $b_i \in V$ and~$|\set{u \in U}{b_1 \le u \le b_{k+p}, \signature_u = +}| = p$.
\end{itemize}
With this formulation, the main differences between the $k$-twist congruence~$\equiv^k$ of Definition~\ref{def:ktwistCongruence} and the Cambrian $(k,\signature)$-twist congruence~$\equiv^k_\signature$ of Definition~\ref{def:ktwistCongruenceCambrian} are that:
\begin{itemize}
\item the signs of the congruence witnesses matters: \eg $\down{1}\down{3}\down{2} \equiv^1 \down{3}\down{1}\down{2}$ while $\down{1}\down{3}\up{2} \not\equiv^1 \down{3}\down{1}\up{2}$,
\item we have positive witnesses in~$U$ or negative witnesses in~$V$: \eg $\down{1}\down{3}\down{2} \equiv^1 \down{3}\down{1}\down{2}$ and $\up{2}\down{1}\down{3} \equiv^1 \up{2}\down{3}\down{1}$,
\item the remoteness of the witnesses matters: \eg $\up{2}\down{1}\down{5}\down{3}\down{4} \equiv^2 \up{2}\down{5}\down{1}\down{3}\down{4}$ while $\up{3}\down{1}\down{5}\down{2}\down{4} \not\equiv^2 \up{3}\down{5}\down{1}\down{2}\down{4}$.
\item the number of witnesses is not always~$k$ as it depends on their remoteness: \eg we really need~$\down{2}$, $\down{3}$, $\down{5}$ and~$\down{6}$ to witness the congruence~$\up{4}\down{1}\down{7}\down{2}\down{3}\down{5}\down{6} \equiv^3 \up{4}\down{7}\down{1}\down{2}\down{3}\down{5}\down{6}$.
\end{itemize}

\begin{example}[Cambrian $(1,\signature)$-twist congruence and Cambrian congruence]
When~$k = 1$ and~$\signature \in \pm^n$, the Cambrian $(1,\signature)$-twist congruence is the Cambrian congruence defined by N.~Reading~\cite{Reading-CambrianLattices} as the transitive closure of~$UacV\down{b}W \equiv_\signature UcaV\down{b}W$ and~$U\up{b}VacW \equiv_\signature U\up{b}VcaW$.
\end{example}

\begin{proposition}
\label{prop:ktwistCongruenceCambrian}
For any~$\tau, \tau' \in \fS^\signature$, we have~$\tau \equiv^k_\signature \tau' \iff \surjectionPermBrick(\tau) = \surjectionPermBrick(\tau')$. In other words, the fibers of the Cambrian $(k,\signature)$-twists by~$\surjectionPermBrick$ are precisely the Cambrian $(k, \signature)$-twist congruence classes.
\end{proposition}

\begin{proof}
By Proposition~\ref{prop:ktwistInsertionCambrian}, each fiber of~$\surjectionPermBrick$ gathers the linear extensions of a $(k,\signature)$-twist, so that it is connected by simple transpositions. Therefore, we just need to show the proposition for any two signed permutations~$\tau = UacV$ and~$\tau' = UcaV$ of~$\fS_\signature$ which differ by a simple transposition. By definition of~$\surjectionPermBrick$, we have~$\surjectionPermBrick(\tau) = \surjectionPermBrick(\tau')$ if and only if $a$th and~$c$th pipes have no common elbow. Observe that exchanging two consecutive letters in~$U$ or in~$V$ just flips the corresponding pipes in~$\surjectionPermBrick(\tau)$, and thus does not perturb the relative positions of the $a$th and~$c$th pipes in~$\surjectionPermBrick(\tau)$. By such exchanges, we can therefore assume that all negative letters of~$U$ appear at the beginning of~$U$ while all positive letters of~$V$ appear at the end of~$V$. Using the deletion operation discussed earlier (see \fref{fig:pipeInsertionCambrian}), we thus reduce the study to the situation where~$U$ has only positive letters while $V$ has only negative letters. With some additional exchanges, we can moreover assume that~$U = u_1 \cdots u_p$ with~$u_1 < \dots < u_p$ while~$V = v_1 \cdots v_q$ with~$v_1 < \dots < v_q$. We now show the equivalence in this restricted situation.

Suppose that there exist~$i < j$ such that the letters~$v_i, \dots, v_j$ in~$V$ are $\equiv^k_\signature$-congruence witnesses for the equivalence~$\tau \equiv^k_\signature \tau'$. In our specific situation, this means that $1+2(j-i) - (v_j-v_i) \ge k$ since $|\set{u \in U}{v_i \le u \le v_j, \signature_u = +}| = (v_j-v_i)-(j-i)$ and~$|\set{v \in V}{v_i \le v \le v_j, \signature_v = -}| = j-i+1$. Let~$\lambda$ denote the column of the $a$th exit and~$\mu$ denote the row of the $c$th entry of~$\shape$. A simple induction on~$\ell \in [i,j]$ shows that the $v_\ell$th pipe of~$\surjectionPermBrick(\tau)$ has $1+2(\ell-i)-(v_\ell-v_i)$ \SE-elbows to the right of column~$\lambda$. It follows that the $v_j$th pipe of~$\surjectionPermBrick(\tau)$ has $1+2(j-i) - (v_j-v_i) \ge k$ \SE-elbows to the right of column~$\lambda$, and thus it has only one \SE-elbow~$\elb$ to the left of column~$\lambda$. Now, if the $a$th and $c$th pipes of~$\surjectionPermBrick(\tau)$ shared an elbow~$\elb'$, it would be to the left of column~$\lambda$, above row~$\mu$, and below the $v_j$th pipe (since the $a$th and $c$th pipes are inserted after the $v_j$th pipe). The elbow~$\elb'$ would then lie in the rectangle defined by the two steps of the $v_j$th pipe incident to~$\elb$, a contradiction. Therefore, the $a$th and $c$th pipes share no elbows and we have~$\surjectionPermBrick(\tau) = \surjectionPermBrick(\tau')$. The proof is similar if there exist~$i < j$ such that the letters~$u_i, \dots, u_j$ in~$U$ are $\equiv^k_\signature$-congruence witnesses. For the converse, one easily checks that if there are neither witnesses in~$U$ nor witnesses in~$V$, then the $a$th and $c$th pipes of~$\surjectionPermBrick(\tau)$ share an elbow, and thus~$\surjectionPermBrick(\tau) \ne \surjectionPermBrick(\tau')$.
\end{proof}

\begin{remark}[Numerology again]
Similar to Remark~\ref{rem:numerologyAgain}, we observe that
\[
|\AcyclicTwists(\signature)| = n! \quad\text{for any } \signature \in \pm^n \text{ with } n \le k+1, and
\]
\[
|\AcyclicTwists((+-)^k{+})| = |\AcyclicTwists[2k-1](-^{2k+1})| = (2k+1)! - (2k-1)!
\]
since the only non-trivial Cambrian $(k,(+-)^k{+})$-twist congruence classes are the pairs of permutations~$\{1(2k+1)(\sigma_1+1)\cdots(\sigma_{2k-1}+1), (2k+1)1(\sigma_1+1)\cdots(\sigma_{2k-1}+1)\}$ for~$\sigma \in \fS_{2k-1}$. The other numbers in Table~\ref{table:numbersAcyclicTwistsCambrian} remain mysterious for us.
\end{remark}

\enlargethispage{.4cm}
Using definitions and techniques already presented in Section~\ref{subsec:latticeCongruences}, we show that the Cambrian $k$-twist congruences are lattice congruences, generalizing the result of N.~Reading~\cite{Reading-CambrianLattices}.

\begin{theorem}
\label{theo:ktwistLatticeCongruenceCambrian}
For any signature~$\signature \in \pm^n$ and any~$k \in \N$, the Cambrian $(k,\signature)$-twist congruence~$\equiv^k$ is a lattice congruence of the weak order.
\end{theorem}

\begin{proof}
We already observed in Proposition~\ref{prop:intervalsCambrian} that the Cambrian $(k,\signature)$-twist congruence classes are intervals of the weak order. The proof that the up and down projection maps are order preserving is similar to the proof Theorem~\ref{theo:ktwistLatticeCongruence}. The only delicate observation was already made in the proof of Proposition~\ref{prop:ktwistCongruenceCambrian}: namely, $UacV \equiv^k_\signature UcaV \iff U'acV' \equiv^k_\signature U'caV'$ for any permutations~$U'$ of~$U$ and~$V'$ of~$V$.
\end{proof}

As in Section~\ref{subsec:flips}, define the \defn{increasing flip order} to be the transitive closure of the increasing flip graph on acyclic Cambrian $(k,\signature)$-twists. See \fref{fig:IncreasingFlipLatticesCambrian} for an illustration when~$k = 2$ and~$n = 4$.

\hvFloat[floatPos=p, capWidth=h, capPos=r, capAngle=90, objectAngle=90, capVPos=c, objectPos=c]{figure}
{\includegraphics[scale=.82]{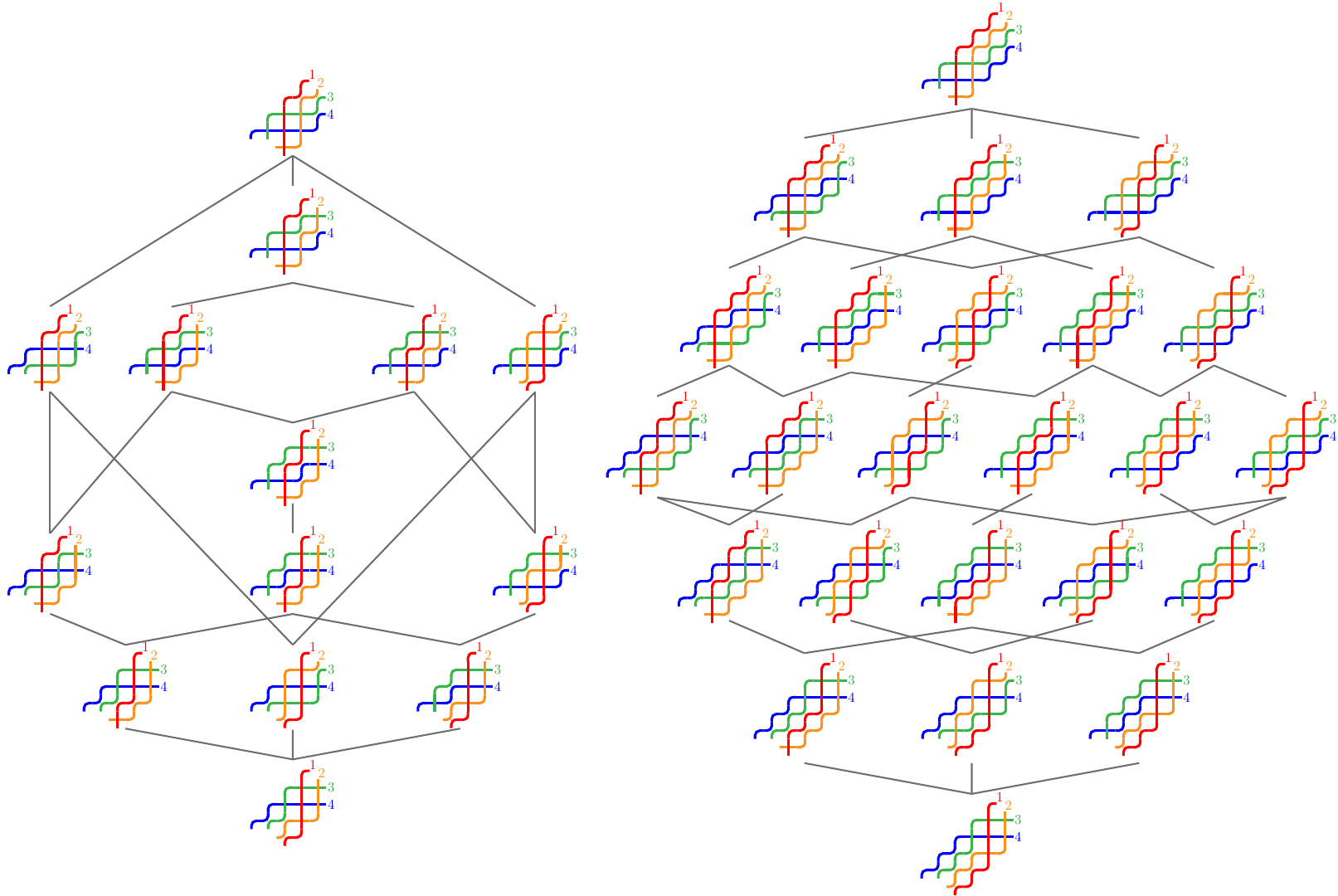}}
{The increasing flip lattice on Cambrian $(k,{+}{-}{+}{+})$-twists for~$k = 1$ (left) and~$k = 2$ (right).}
{fig:IncreasingFlipLatticesCambrian}

\begin{proposition}
\label{prop:latticeQuotientCambrian}
The increasing flip order on acyclic Cambrian $(k,\signature)$-twists is isomorphic to:
\begin{itemize}
\item the quotient lattice of the weak order by the Cambrian $(k,\signature)$-twist congruence~$\equiv^k_\signature$,
\item the subposet of the weak order induced by the permutations avoiding the increasing rewriting rules (maximums of Cambrian $(k,\signature)$-twist congruence classes),
\item the subposet of the weak order induced by the permutations avoiding the decreasing rewriting rules (minimums of Cambrian $(k,\signature)$-twist congruence classes).
\end{itemize}
\end{proposition}

\begin{proof}
Same as the proof of Proposition~\ref{prop:latticeQuotient}.
\end{proof}

\begin{example}[Cambrian lattices]
When~$k = 1$, the increasing flip lattice is the $\signature$-Cambrian lattice of N.~Reading~\cite{Reading-CambrianLattices}.
\end{example}

As for classical twists, the $i$th and $j$th pipes of any Cambrian $(k,\signature)$-twist~$\twist$ are comparable in~$\contactLess{\twist}$ as soon as~$|i-j| \le k$. We can therefore define as in Section~\ref{subsec:canopy} the \defn{canopy scheme} of~$\twist$ as the orientation~${\surjectionBrickZono(\twist) \in \AcyclicOrientations(n)}$ with an edge~$i \to j$ for all~$i,j \in [n]$ such that~$|i-j| \le k$ and~$i \contactLess{\twist} j$. We still call \defn{$k$-canopy} the map~$\surjectionBrickZono: \AcyclicTwists(\signature) \to \AcyclicOrientations(n)$. Define also the \defn{$k$-recoil scheme} $\surjectionPermZono(\tau)$ of a signed permutation~$\tau$ as the $k$-recoil scheme of~$\tau$ when forgetting its signs.

\begin{proposition}
\label{prop:latticeHomomorphismsCambrian}
The maps~$\surjectionPermBrick$, $\surjectionBrickZono$, and~$\surjectionPermZono$ define the following commutative diagram of lattice homomorphisms:
\[
\begin{tikzpicture}
  \matrix (m) [matrix of math nodes,row sep=1.2em,column sep=5em,minimum width=2em]
  {
     \fS^\signature  	&								& \AcyclicOrientations(n)	\\
						& \AcyclicTwists(\signature) 	&							\\
  };
  \path[->>]
    (m-1-1) edge node [above] {$\surjectionPermZono$} (m-1-3)
                 edge node [below] {$\surjectionPermBrick\quad$} (m-2-2.west)
    (m-2-2.east) edge node [below] {$\surjectionBrickZono$} (m-1-3);
\end{tikzpicture}
\]
\end{proposition}

\begin{proof}
Same as the proof of Proposition~\ref{prop:latticeHomomorphisms}.
\end{proof}

\subsection{Geometry of acyclic Cambrian twists}
\label{subsec:geometryCambrian}

As in Section~\ref{sec:geometry}, the geometric properties of acyclic Cambrian twists are driven by the normal fan of the brick polytope, defined as follows.

\begin{definition}
The \defn{brick vector} of a Cambrian $(k,\signature)$-twist~$\twist$ is the vector~$\b{x}(\twist) \in \R^n$ whose $i$th coordinate is the number of bricks of~$\shape$ below the $i$th pipe of~$\twist$, minus a constant~$\alpha_{k, \signature, i}$ (independent of~$\twist$) chosen so that the origin is in the middle of the brick vectors of the minimal and maximal Cambrian $(k,\signature)$-twists. The \defn{brick polytope}~$\Brick[k][\signature]$ is the convex hull of the brick vectors of all Cambrian $(k,\signature)$-twists.
\end{definition}

The brick polytope~$\Brick[k][\signature]$ is illustrated in \fref{fig:PermBrickZonoCambrian} for~$\signature = {+}{-}{+}{+}$ and~$k = 1,2,3$.

\begin{remark}[Inclusions]
To be consistent with Section~\ref{sec:geometry}, we have translated the brick polytope of~\cite{PilaudSantos-brickPolytope} by the vector~$\sum_{i \in [n]} \alpha_{k, \signature, i} \, \b{e}_i$ so that the permutahedron~$\Perm$ is contained inside the brick polytope~$\Brick[k][\signature]$. Nevertheless, the zonotope~$\Zono$ does not necessarily contain anymore the brick polytope~$\Brick[k][\signature]$ when~$k \ge 2$, as is illustrated in the last two rows of \fref{fig:PermBrickZonoCambrian}.
\end{remark}

Applying the results of~\cite{PilaudSantos-brickPolytope}, the properties presented in Section~\ref{subsec:normalFans} directly translate in the Cambrian world:
\begin{enumerate}[(i)]
\item The collection of cones~$\bigset{\Cone\polar(\twist)}{\twist \in \AcyclicTwists(\signature)}$, together with all their faces, form the normal fan of the brick polytope~$\Brick[k][\signature]$.
\item The surjections~$\surjectionPermBrick : \fS^\signature \to \AcyclicTwists(\signature)$ and~$\surjectionBrickZono : \AcyclicTwists(\signature) \to \AcyclicOrientations(n)$ are characterized by polar cone inclusions: $\Cone\polar(\tau) \subseteq \Cone\polar \big( \surjectionPermBrick(\tau) \big)$ and~$\Cone\polar(\twist) \subseteq \Cone\polar \big( \surjectionBrickZono(\twist) \big)$.
\item The $1$-skeleton of the brick polytope~$\Brick[k][\signature]$, oriented in the direction~$\sum_{i \in [n]} (n+1-2i) \, \b{e}_i$ is the Hasse diagram of the increasing flip lattice on acyclic Cambrian $(k,\signature)$-twists.
\end{enumerate}

\begin{figure}[p]
	\centerline{\includegraphics[scale=.68]{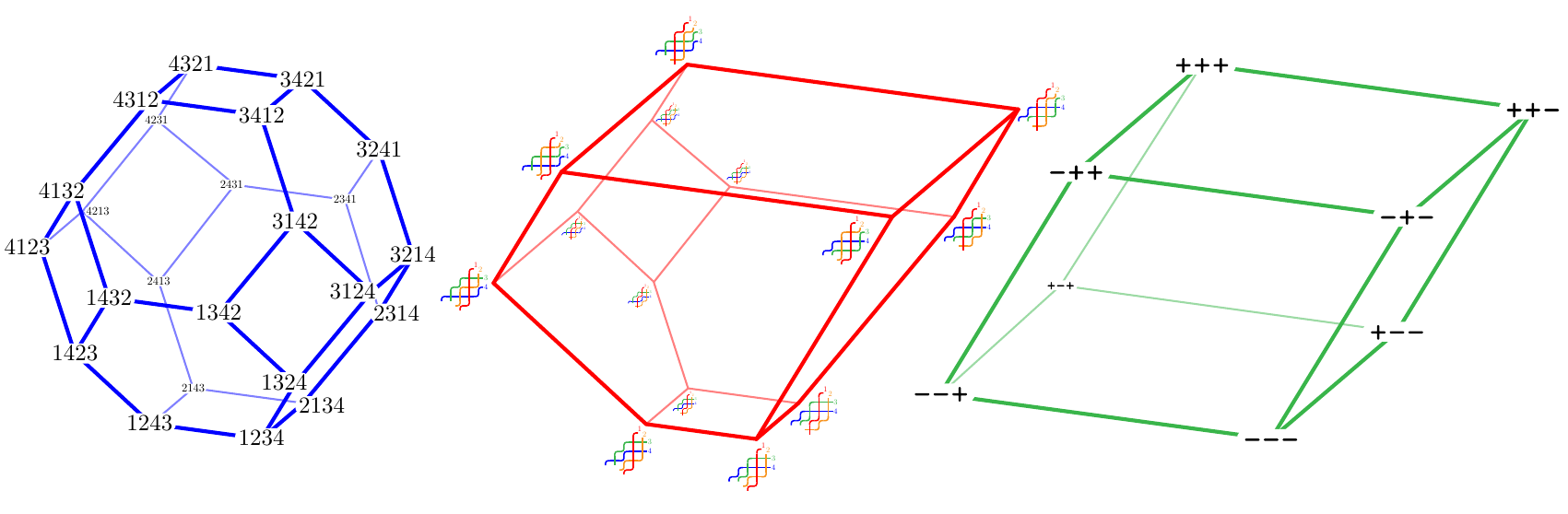}} \medskip
	\centerline{\includegraphics[scale=.68]{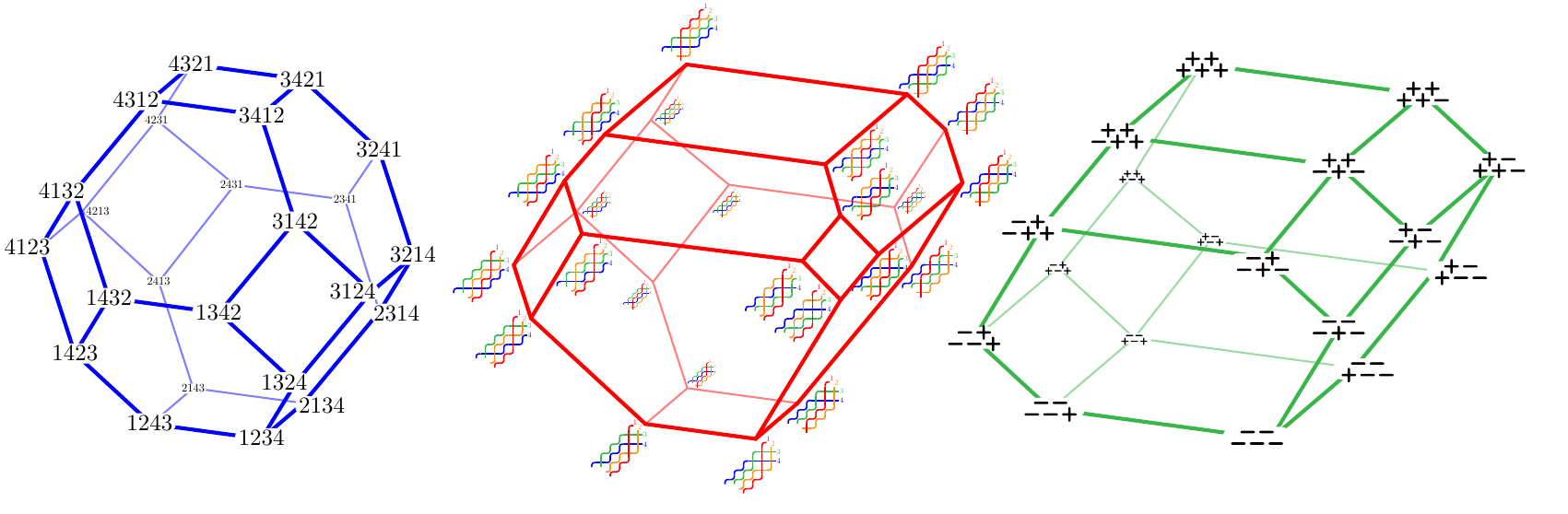}} \medskip
	\centerline{\includegraphics[scale=.68]{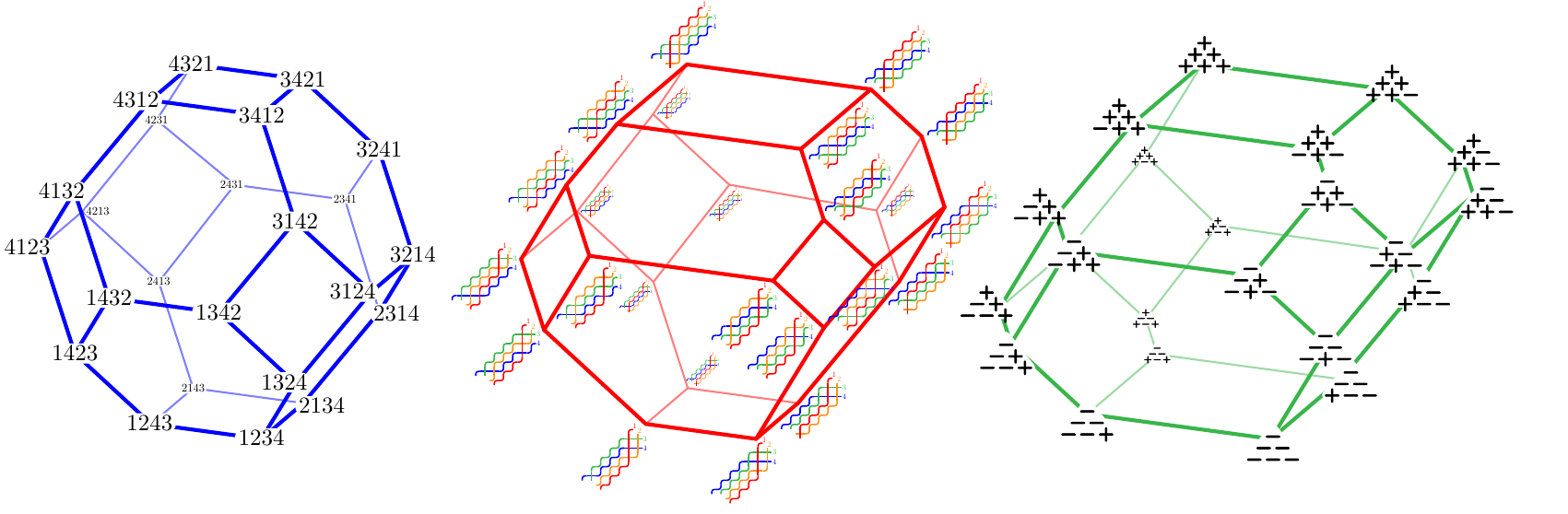}}
	\caption{The permutahedron~$\Perm[k][4]$ (left), the brick polytope $\Brick[k][{+}{-}{+}{+}]$ (middle) and the zonotope~$\Zono[k][4]$ (right) for $k = 1$~(top), $k = 2$~(middle) and~$k = 3$ (bottom). For readibility, we represent orientations of~$\Gkn$ by pyramids of signs.}
	\label{fig:PermBrickZonoCambrian}
\end{figure}

\subsection{Algebra of acyclic Cambrian twists}
\label{subsec:algebraCambrian}

To extend the twist algebra to the Cambrian setting, we follow the same path as~\cite{ChatelPilaud}. We first generalize the shifted shuffle and the convolution product to signed permutations. The \defn{signed shifted shuffle product}~$\tau \shiftedShuffle \tau'$ is defined as the shifted product of the permutations where signs travel with their values, while the \defn{signed convolution product}~$\tau \convolution \tau'$ is defined as the convolution product of the permutations where signs stay at their positions. For example,
\begin{align*}
\upr{1}\downr{2} \shiftedShuffle \downb{23}\upb{1} & = \{ \upr{1}\downr{2}\downb{45}\upb{3}, \upr{1}\downb{4}\downr{2}\downb{5}\upb{3}, \upr{1}\downb{45}\downr{2}\upb{3}, \upr{1}\downb{45}\upb{3}\downr{2}, \downb{4}\upr{1}\downr{2}\downb{5}\upb{3}, \downb{4}\upr{1}\downb{5}\downr{2}\upb{3}, \downb{4}\upr{1}\downb{5}\upb{3}\downr{2}, \downb{45}\upr{1}\downr{2}\upb{3}, \downb{45}\upr{1}\upb{3}\downr{2}, \downb{45}\upb{3}\upr{1}\downr{2} \}, \\
\upr{1}\downr{2} \convolution \downb{23}\upb{1} & = \{ \upr{1}\downr{2}\downb{45}\upb{3}, \upr{1}\downr{3}\downb{45}\upb{2}, \upr{1}\downr{4}\downb{35}\upb{2}, \upr{1}\downr{5}\downb{34}\upb{2}, \upr{2}\downr{3}\downb{45}\upb{1}, \upr{2}\downr{4}\downb{35}\upb{1}, \upr{2}\downr{5}\downb{34}\upb{1}, \upr{3}\downr{4}\downb{25}\upb{1}, \upr{3}\downr{5}\downb{24}\upb{1}, \upr{4}\downr{5}\downb{23}\upb{1} \}.
\end{align*}
Note that the shifted shuffle is compatible with signed values, while the convolution is compatible with signed positions: $\fS^\signature \shiftedShuffle \fS^{\signature'} = \fS^{\signature\signature'}$ and~$\fS_\signature \convolution \fS_{\signature'} = \fS_{\signature\signature'}$. In any case, both~$\shiftedShuffle$ and~$\convolution$ are compatible with the distribution of positive and negative signs: $|\tau \shiftedShuffle \tau'|_+ = |\tau|_+ + |\tau'|_+ = |\tau \convolution \tau'|_+$ and~$|\tau \shiftedShuffle \tau'|_- = |\tau|_- + |\tau'|_- = |\tau \convolution \tau'|_-$.

We denote by
\[
\fS_\pm \eqdef \bigsqcup_{\substack{n \in \N \\ \signature \in \pm^n}} \fS_\signature = \bigsqcup_{\substack{n \in \N \\ \signature \in \pm^n}} \fS^\signature
\]
the set of all signed permutations of arbitrary size and arbitrary signature. We consider the Hopf algebra~$\FQSym_\pm$ with basis~$(\F_\tau)_{\tau \in \fS_\pm}$ and whose product and coproduct are defined by
\[
\F_\tau \product \F_{\tau'} = \sum_{\sigma \in \tau \shiftedShuffle \tau'} \F_\sigma
\qquad\text{and}\qquad
\coproduct \F_\sigma = \sum_{\sigma \in \tau \convolution \tau'} \F_\tau \otimes \F_{\tau'}.
\]
This Hopf algebra is bigraded by the number of positive and the number of negative signs of the signed permutations.

We now consider the vector subspace~$\Twist_\pm$ of~$\FQSym_\pm$ generated by the elements
\[
\PTwist_{\twist} \eqdef \sum_{\substack{\tau \in \fS_\pm \\ \surjectionPermBrick(\tau) = \twist}} \F_\tau = \sum_{\tau \in \linearExtensions(\twist\contact)} \F_\tau,
\]
for all acyclic Cambrian $(k,\signature)$-twists~$\twist$ for all signatures~$\signature \in \pm^*$. For example, for the signature~${\signature = {-}{+}{+}{-}{-}}$ and the Cambrian $(k,\signature)$-twists of \fref{fig:insertionAlgorithmCambrian}, we have
\[
\PTwist_{\hspace{-.3cm}\includegraphics[scale=.5]{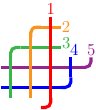}} = \sum_{\tau \in \fS^\signature} \F_\tau
\qquad
\PTwist\raisebox{.05cm}{$_{\hspace{-.4cm}\includegraphics[scale=.5]{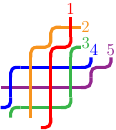}}$} = \!\!\!\! \begin{array}[t]{c} \phantom{+} \; \F_{\up{3}\down{1}\up{2}\down{54}} + \F_{\up{3}\down{15}\up{2}\down{4}} + \F_{\up{3}\down{51}\up{2}\down{4}} \\ + \; \F_{\down{5}\up{3}\down{1}\up{2}\down{4}} + \F_{\up{3}\down{154}\up{2}} + \F_{\up{3}\down{514}\up{2}} \\ + \; \F_{\down{5}\up{3}\down{14}\up{2}} + \F_{\up{3}\down{541}\up{2}} + \F_{\down{5}\up{3}\down{41}\up{2}} \end{array}
\qquad
\PTwist\raisebox{.1cm}{$_{\hspace{-.5cm}\includegraphics[scale=.5]{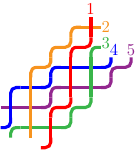}}$} = \F_{\up{3}\down{154}\up{2}}
\qquad
\PTwist\raisebox{.15cm}{$_{\hspace{-.6cm}\includegraphics[scale=.5]{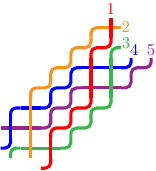}}$} = \F_{\up{3}\down{154}\up{2}}.
\]

\begin{theorem}
\label{theo:twistSubalgebraCambrian}
$\Twist_\pm$ is a Hopf subalgebra of~$\FQSym_\pm$.
\end{theorem}

\begin{proof}
As in the proof of Theorem~\ref{theo:twistSubalgebra}, we check that the subspace~$\Twist_\pm$ is stable by the product and coproduct of~$\FQSym_\pm$. This follows from the fact that the condition of Definition~\ref{def:ktwistCongruenceCambrian} for two permutations~$\tau = UacV$ and~$\tau' = UcaV$ to be in the same Cambrian $(k,\signature)$-twist congruence class only depends on the positions of the entries of~$\tau$ and~$\tau'$ between~$a$ and~$c$. Details are left to the reader.
\end{proof}

\begin{example}[Cambrian algebra]
The bijection given in Example~\ref{exm:1twistsTriangulationsCambrian} (see also \fref{fig:1twistsTriangulationsCambrian}) defines an isomorphism from the Cambrian $1$-twist algebra~$\Twist[1]_\pm$ to G.~Chatel and V.~Pilaud's Hopf algebra~$\Camb$ on Cambrian trees~\cite{ChatelPilaud}.
\end{example}

As in Section~\ref{subsec:subalgebra}, it is interesting to describe combinatorially the product and coproduct in~$\Twist_\pm$ directly in terms of Cambrian $k$-twists:
\begin{itemize}
\item For~$\twist \in \AcyclicTwists(\signature)$ and~$\twist' \in \AcyclicTwists(\signature')$, the product~$\PTwist_{\twist} \product \PTwist_{\twist'}$ in~$\Twist_\pm$ is given by
\[
\PTwist_{\twist} \product \PTwist_{\twist'} = \sum_{\twist[S]} \PTwist_{\twist[S]},
\]
where the sum runs over the interval of the increasing flip lattice between the Cambrian $(k, \signature\signature')$-twists~$\underprod{\twist}{\twist'}$ and~$\overprod{\twist}{\twist'}$, defined as in the unsigned case. \\[-.3cm]
\item For~$\twist[S] \in \AcyclicTwists(\signature)$, the coproduct~$\coproduct\PTwist_{\twist[S]}$ is given by
\[
\coproduct\PTwist_{\twist[S]} = \sum_\gamma \bigg( \sum_\tau \PTwist_{\surjectionPermBrick(\tau)} \bigg) \otimes \bigg( \sum_{\tau'} \PTwist_{\surjectionPermBrick(\tau')} \bigg),
\]
where~$\gamma$ runs over all cuts of~$\twist[S]$, and $\tau$ (resp.~$\tau'$) runs over a set of representatives of the Cambrian $k$-twist congruence classes of the linear extensions of~$B\contact(\twist[S],\gamma)$ (resp.~$A\contact(\twist[S],\gamma)$). Again, this description of the coproduct is not completely satisfactory as it essentially computes the product back in~$\FQSym_\pm$.
\end{itemize}

The reader is also invited to work out the combinatorial descriptions of the product and coproduct in the dual Hopf algebra~$\Twist[k*]_\pm$, which are similar to that of Section~\ref{subsec:quotientAlgebra}.

\section{Tuplization}
\label{sec:tuplization}

In this section, we extend our results to tuples of Cambrian $k$-twists. The ideas developed in~\cite[Section~2.3]{ChatelPilaud} apply verbatim to the context of $k$-twists. The motivation comes from S.~Law and N.~Reading's work on diagonal quadrangulations~\cite{LawReading} and S.~Giraudo's work on twin binary trees~\cite{Giraudo}.

\subsection{Cambrian $k$-twist $\ell$-tuples}
\label{subsec:CambrianTwistTuples}

For a $\ell$-tuple~$\tuple$ and~$i \in [\ell]$, we denote by~$\tuple_{[i]}$ the $i$th element of~$\tuple$ (this notation avoids confusion when we will use tuples of permutations: $\tau_{[i]}$ is a permutation while $\tau_i$ is an entry in a permutation). We focus on the following objects.

\begin{definition}
For a signature $\ell$-tuple~$\signatures \eqdef [\signatures_{[1]}, \dots, \signatures_{[\ell]}] \in (\pm^n)^\ell$, a \defn{Cambrian $(k, \signatures)$-twist tuple} is a \mbox{$\ell$-tuple} $\tuple \eqdef [\tuple_{[1]}, \dots, \tuple_{[\ell]}]$ where~$\tuple_{[i]}$ is a Cambrian $(k,\signatures_{[i]})$-twist and such that the union of the contact graphs~$\tuple_{[1]}\contact, \dots, \tuple_{[\ell]}\contact$ is acyclic. Let~$\TwistTuples(\signatures)$ denote the set of Cambrian \mbox{$(k,\signatures)$-twist} tuples. A \defn{Cambrian $k$-twist $\ell$-tuple} is a Cambrian $(k, \signatures)$-twist tuple for some signature $\ell$-tuple~$\signatures \in (\pm^n)^\ell$.
\end{definition}

To define an analogue of the Cambrian $k$-twist correspondence, we need permutations recording $\ell$ signatures. Call \defn{$\ell$-signed permutation} a permutation where each position/value receives an~$\ell$-tuple of signs.
For example, $\uptilde{\upw{3}} \uptilde{\downw{1}} \downtilde{\downw{5}} \uptilde{\downw{4}} \downtilde{\upw{2}}$ is a $2$-signed permutation whose signatures are marked with~$\up{\phantom{a}}/\down{\phantom{a}}$ and~$\uptilde{\phantom{a}}/\downtilde{\phantom{a}}$ respectively.
For a $\ell$-signed permutation~$\tau$ and~$i \in [\ell]$, we denote by~$\tau_{[i]}$ the signed permutation where we only keep the~$i$th signature.
For example
\[
\uptilde{\upw{3}} \uptilde{\downw{1}} \downtilde{\downw{5}} \uptilde{\downw{4}} \downtilde{\upw{2}}_{[1]} = \up{3}\down{154}\up{2}
\qquad\text{and}\qquad
\uptilde{\upw{3}} \uptilde{\downw{1}} \downtilde{\downw{5}} \uptilde{\downw{4}} \downtilde{\upw{2}}_{[2]} = \simtilde{3} \simtilde{1} \downtilde{5} \simtilde{4} \downtilde{2}.
\]
We denote by~$\fS_{\pm^\ell}$ the set of all $\ell$-signed permutations and by~$\fS_\signatures$ (resp.~$\fS^\signatures$) the set of $\ell$-signed permutations with p-signatures (resp.~v-signatures)~$\signatures$. Applying~$\ell$ Cambrian $k$-twist correspondences in parallel yields a map form $\ell$-signed permutations to Cambrian $\ell$-tuples.

\begin{proposition}
The map~$\surjectionPermBrick_\ell$ defined by~$\surjectionPermBrick_\ell(\tau) = \big[ \surjectionPermBrick(\tau_{[1]}), \dots, \surjectionPermBrick(\tau_{[\ell]}) \big]$ defines a surjection from the $\ell$-signed permutations to the Cambrian $k$-twist $\ell$-tuples.
\end{proposition}

\begin{proof}
Consider a Cambrian $k$-twist $\ell$-tuple~$\tuple$. For a $\ell$-signed permutation~$\tau$, we have~$\surjectionPermBrick_\ell(\tau) = \tuple$ if and only if~$\surjectionPermBrick(\tau_{[i]}) = \tuple_{[i]}$ for all~$i \in [\ell]$, or equivalently by Proposition~\ref{prop:ktwistInsertionCambrian}, if and only if~$\tau$ is a linear extension of each Cambrian $k$-twist~$\tuple_{[i]}$. We therefore obtain that the fiber of~$\tuple$ under~$\surjectionPermBrick_\ell$ is precisely the set of linear extensions of the union of the contact graphs~$\tuple_{[1]}\contact, \dots, \tuple_{[\ell]}\contact$.
\end{proof}

We now consider the congruence on~$\fS_{\pm^\ell}$ recording which $\ell$-signed permutations have the same image by~$\surjectionPermBrick_\ell$.

\begin{definition}
\label{def:ktwistTupleCongruence}
For a signature $\ell$-tuple~$\signatures \in (\pm^n)^\ell$, the \defn{Cambrian $(k,\signatures)$-twist tuple congruence} on~$\fS^\signatures$ is the intersection~$\equiv^k_\signatures \eqdef \bigcap_{i \in [\ell]} \equiv^k_{\signatures_{[i]}}$ of all Cambrian $(k,\signatures_{[i]})$-twist congruences. We denote by~$\equiv^k_\ell$ the congruence relation on~$\fS_{\pm^\ell}$ obtained as the union of the Cambrian $(k,\signatures)$-twist congruences~$\equiv^k_\signatures$ for all signature $\ell$-tuples~$\signatures$.
\end{definition}

\begin{proposition}
For any~$\tau, \tau' \in \fS_{\pm^\ell}$, we have~$\tau \equiv^k_\ell \tau' \iff \surjectionPermBrick_\ell(\tau) = \surjectionPermBrick_\ell(\tau')$. In other words, the fibers of~$\surjectionPermBrick_\ell$ are precisely the congruence classes of~$\equiv^k_\ell$.
\end{proposition}

\begin{proof}
We have
\[
\begin{array}[b]{c@{\;\iff\;}l@{\qquad}r}
\tau \equiv^k_\ell \tau'
& \tau_{[i]} \equiv^k_{\signatures_{[i]}} \tau'_{[i]} \text{ for all } i \in [\ell] & \text{(by definition of~$\equiv^k_\ell$),} \\
& \surjectionPermBrick(\tau_{[i]}) = \surjectionPermBrick(\tau'_{[i]}) \text{ for all } i \in [\ell] & \text{(by Proposition~\ref{prop:ktwistCongruenceCambrian}),}\\
& \surjectionPermBrick_\ell(\tau) = \surjectionPermBrick_\ell(\tau') & \text{(by definition of~$\surjectionPermBrick_\ell$).}
\end{array}\qedhere
\]
\end{proof}

\begin{proposition}
For any fixed signature $\ell$-tuple~$\signatures \in (\pm^n)^\ell$, the Cambrian $(k,\signatures)$-twist tuple congruence~$\equiv^k_\signatures$ is a lattice congruence of the weak order on~$\fS_n$.
\end{proposition}

\begin{proof}
An intersection of lattice congruences is a lattice congruence, see~\cite{Reading-LatticeCongruences} for details.
\end{proof}

We now define flips in Cambrian $k$-twist $\ell$-tuples. It is essentially an elbow flip performed in all elements of the $\ell$-tuple in which it is possible.

\begin{definition}
Let~$\tuple$ be a Cambrian $k$-twist $\ell$-tuple and~$p \to q$ be an arc in the union of the contact graphs~$\tuple_{[1]}\contact, \dots, \tuple_{[\ell]}\contact$. We say that this arc is \defn{flippable} if either~$p \to q$ is an arc or $p$~and~$q$ are incomparable in each contact graph~$\tuple_{[i]}\contact$. The \defn{flip} of~$p \to q$ then transforms~$\tuple$ to the $\ell$-tuple~$\tuple' \eqdef [\tuple'_{[1]}, \dots, \tuple'_{[\ell]}]$ where~$\tuple'_{[i]}$ is obtained from~$\tuple_{[i]}$ by flipping the elbow between pipes~$p$ and~$q$ farthest from their crossing if~$p \to q$ is an arc in~$\tuple_{[i]}\contact$, while~$\tuple'_{[i]} = \tuple_{[i]}$ otherwise. We say that the flip of~$p \to q$ is \defn{increasing} if~$p < q$.
\end{definition}

\begin{proposition}
The increasing flip order on Cambrian $(k,\signatures)$-twist tuples is the quotient lattice of the weak order by the Cambrian $(k,\signatures)$-twist congruence~$\equiv^k_\signatures$.
\end{proposition}

\begin{proof}
Consider two distinct Cambrian $(k,\signatures)$-twist tuples~$\tuple,\tuple'$ and their Cambrian $(k,\signatures)$-twist tuple congruence classes~$C,C'$. If there exists adjacent representatives of~$C$ and~$C'$, then the Cambrian $(k,\signatures_{[i]})$-twists~$\tuple_{[i]}$ and~$\tuple'_{[i]}$ are either equal or differ by an elbow flip for each~$i \in [\ell]$ by Proposition~\ref{prop:latticeQuotientCambrian}, and thus the two distinct Cambrian $(k,\signatures)$-twist tuples~$\tuple,\tuple'$ differ by a flip. Conversely, if $\tuple,\tuple'$ differ by the flip of an arc~$p \to q$, then there exists a linear extension~$\tau$ of the union of the contact graphs~$\tuple_{[1]}\contact, \dots, \tuple_{[\ell]}\contact$ where $p$ and~$q$ are consecutive. Since~$\tuple'_{[i]}$ is obtained by flipping the elbow between~$p$ and~$q$ in~$\tuple_{[i]}$ if possible and~$\tuple'_{[i]} = \tuple_{[i]}$ otherwise, the permutation~$\tau'$ obtained by switching~$p$ and~$q$ in~$\tau$ is a linear extension of~${\tuple'_{[1]}}\!\!\!\contact, \dots, {\tuple'_{[\ell]}}\!\!\!\contact$. We thus obtained two adjacent permutations~$\tau \in C$~and~$\tau \in C'$.
\end{proof}

In a Cambrian $k$-twist $\ell$-tuple~$\tuple$, all Cambrian $k$-twists~$\tuple_{[i]}$ have the same $k$-canopy~$\surjectionBrickZono(\tuple_{[i]})$. Indeed all these Cambrian $k$-twists have a common linear extension and the $k$-canopy of a Cambrian $k$-twist is given by the relative order of the last $k$ values of its linear extensions. We can thus define the \defn{$k$-canopy} of the Cambrian $k$-twist $\ell$-tuple~$\tuple$ as~${\surjectionBrickZono_\ell(\tuple) = \surjectionBrickZono(\tuple_{[1]}) = \dots = \surjectionBrickZono(\tuple_{[\ell]})}$. Define also the \defn{$k$-recoil scheme} $\surjectionPermZono(\tau)$ of a $\ell$-signed permutation~$\tau$ as the $k$-recoil scheme of~$\tau$ when forgetting its signs. The next statement is an immediate consequence of these definitions.

\begin{proposition}
The maps~$\surjectionPermBrick_\ell$, $\surjectionBrickZono_\ell$, and~$\surjectionPermZono$ define a commutative diagram of lattice homomorphisms:
\[
\begin{tikzpicture}
  \matrix (m) [matrix of math nodes,row sep=1.2em,column sep=5em,minimum width=2em]
  {
     \fS^\signatures  	&								& \AcyclicOrientations(n)	\\
						& \TwistTuples(\signatures) 	&							\\
  };
  \path[->>]
    (m-1-1) edge node [above] {$\surjectionPermZono$} (m-1-3)
                 edge node [below] {$\surjectionPermBrick_\ell\quad$} (m-2-2.west)
    (m-2-2.east) edge node [below] {$\surjectionBrickZono_\ell$} (m-1-3);
\end{tikzpicture}
\]
\end{proposition}

\subsection{Geometric realization}
\label{subsec:geometricRealizationTuples}

Following~\cite{LawReading}, we observe in this section that the increasing flip order on Cambrian $k$-twist $\ell$-tuples can be realized geometrically using Minkowski sums of brick polytopes. For a Cambrian $(k,\signatures)$-twist tuple~$\tuple$, we denote by~$\Cone(\tuple)$ and~$\Cone\polar(\tuple)$ the incidence and braid cones of the transitive closure of the union of the contact graphs~$\tuple_{[1]}\contact, \dots, \tuple_{[\ell]}\contact$.

\begin{proposition}
The Minkowski sum~$\Brick[k][\signatures]$ of the brick polytopes~$\Brick[k][\signatures_{[i]}]$ provides a geometric realization of the increasing flip lattice Cambrian $(k,\signatures)$-twist tuples:
\begin{enumerate}[(i)]
\item The collection of cones~$\bigset{\Cone\polar(\tuple)}{\tuple \in \TwistTuples(\signatures)}$, together with all their faces, form the normal fan of~$\Brick[k][\signatures]$.
\label{item:normalFanMinkowskiSum}
\item The surjections~$\surjectionPermBrick_\ell : \fS^\signatures \to \TwistTuples(\signatures)$ and~$\surjectionBrickZono : \TwistTuples(\signatures) \to \AcyclicOrientations(n)$ are characterized by polar cone inclusions: $\Cone\polar(\tau) \subseteq \Cone\polar \big( \surjectionPermBrick_\ell(\tau) \big)$ and~$\Cone\polar(\twist) \subseteq \Cone\polar \big( \surjectionBrickZono_\ell(\twist) \big)$.
\label{item:surjection}
\item The $1$-skeleton of~$\Brick[k][\signatures]$, oriented in the direction~$\sum_{i \in [n]} (n+1-2i) \, \b{e}_i$ is the Hasse diagram of the increasing flip lattice on Cambrian $(k,\signatures)$-twist tuples.
\label{item:skeleton}
\end{enumerate}
\end{proposition}

\begin{proof}
The normal fan of a Minkowski sum is the common refinement of the normal fans of its summands. For~$\twist_1 \in \AcyclicTwists(\signatures_{[1]}), \dots, \twist_\ell \in \AcyclicTwists(\signatures_{[\ell]})$, the normal cones~$\Cone\polar(\twist_1), \dots, \Cone\polar(\twist_\ell)$ intersect precisely when the contact graphs~$\twist_1\contact, \dots, \twist_\ell\contact$ have a common linear extension, and their intersection is the braid cone of the transitive closure of the union of the contact graphs~$\twist_1\contact, \dots, \twist_\ell\contact$. Point~\eqref{item:normalFanMinkowskiSum} follows. Point~\eqref{item:surjection} follows from the fact that a poset~$\less$ is an extension of a poset~$\less'$ if and only if~$\Cone\polar(\less) \subseteq \Cone\polar(\less')$. Finally, Point~\eqref{item:skeleton} is a consequence of Point~\eqref{item:normalFanMinkowskiSum} and the fact that increasing flips are oriented in the direction~$\sum_{i \in [n]} (n+1-2i) \, \b{e}_i$.
\end{proof}

For example, the polytope~$\Brick[1][\signatures]$ for~$\signatures = [{-}{-}{-}{-}{-}, {-}{+}{+}{-}{-}]$ is represented in \fref{fig:BrickTuple}.

\begin{figure}[h]
	\centerline{\includegraphics[scale=.92]{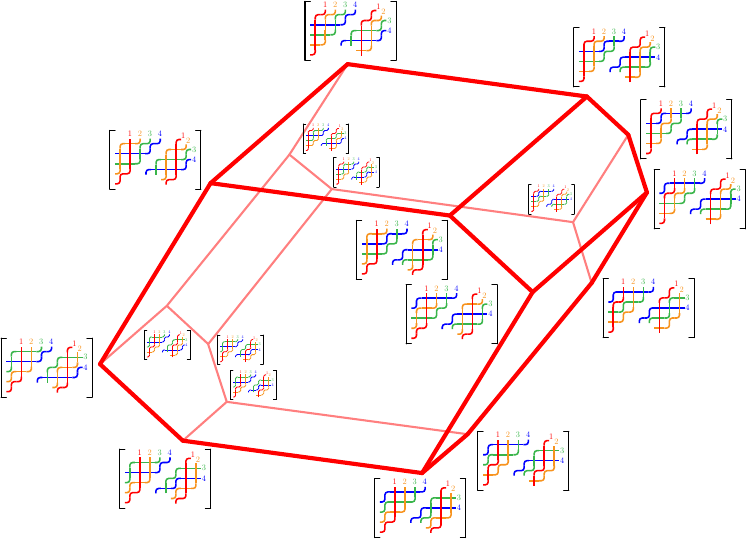}}
	\caption{The polytope~$\Brick[1][\signatures]$ for the signatures~${\signatures = [{-}{-}{-}{-}, {+}{-}{+}{+}]}$. It is the Minkowski sum of the associahedra of Figures~\ref{fig:PermBrickZono} and~\ref{fig:PermBrickZonoCambrian}\,(top center).}
	\label{fig:BrickTuple}
\end{figure}

\subsection{Algebra of acyclic Cambrian twist tuples}
\label{subsec:algebraAcyclicCambrianTwistTuples}

The \defn{shifted shuffle product}~$\tau \shiftedShuffle \tau'$ (resp.~\defn{convolution product}~$\tau \convolution \tau'$) of two $\ell$-signed permutations~$\tau, \tau'$ is still defined as the shifted product (resp.~convolution product) where signs travel with their values (resp.~stay at their positions). We denote by~$\FQSym_{\pm^\ell}$ the Hopf algebra with basis~$(F_\tau)_{\tau \in \fS_{\pm^\ell}}$ indexed by $\ell$-signed permutations and whose product and coproduct are defined by
\[
\F_\tau \product \F_{\tau'} = \sum_{\sigma \in \tau \shiftedShuffle \tau'} \F_\sigma
\qquad\text{and}\qquad
\coproduct \F_\sigma = \sum_{\sigma \in \tau \convolution \tau'} \F_\tau \otimes \F_{\tau'}.
\]
We then consider the vector subspace~$\Twist_\ell$ of~$\FQSym_{\pm^\ell}$ generated by the elements
\[
\PTwist_{\tuple} \eqdef \sum_{\substack{\tau \in \fS_{\pm^\ell} \\ \surjectionPermBrick_\ell(\tau) = \tuple}} \F_\tau = \sum_{\tau \in \linearExtensions\big(\bigcup\limits_{k \in [\ell]} \tuple_{[k]}\big)} \F_\tau,
\]
for all Cambrian $k$-twists $\ell$-tuples~$\tuple$.
\renewcommand{\uptilde}[1]{\accentset{\vspace{-.05cm}\scalebox{.55}{$\sim$}}{#1}}%
\renewcommand{\downtilde}[1]{\underaccent{\,\scalebox{.55}{$\sim$}}{#1}}%
For example,
\[
\PTwist_{\includegraphics[scale=1]{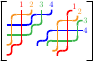}} = \F_{\uptilde{\down{1}} \uptilde{\down{3}} \uptilde{\down{4}} \downtilde{\down{2}}} + \F_{\uptilde{\down{3}} \uptilde{\down{1}} \uptilde{\down{4}} \downtilde{\down{2}}} + \F_{\uptilde{\down{3}} \uptilde{\down{4}} \uptilde{\down{1}} \downtilde{\down{2}}},
\]
where the signatures are marked with~$\up{\phantom{a}}/\down{\phantom{a}}$ and~$\uptilde{\phantom{a}}/\downtilde{\phantom{a}}$ respectively. The next statement is proved as usual by showing that the product and coproduct of~$\FQSym_{\pm^\ell}$ stabilize~$\Twist_\ell$.

\begin{theorem}
\label{theo:twistTupleSubalgebra}
$\Twist_\ell$ is a Hopf subalgebra of~$\FQSym_{\pm^\ell}$.
\end{theorem}

The product and coproduct of~$\PTwist$-basis elements of~$\Twist_\ell$ can again be described directly in terms of combinatorial operations on Cambrian $k$-twist $\ell$-tuples. Details are left to the reader. See \eg \cite[Section~2.3]{ChatelPilaud} for the $k = 1$ case.

\subsection{Twin twists}
\label{subsec:twinTwists}

When $k = 1$, pairs of twin Cambrian trees are particularly relevant Cambrian tuples as they have additional relevant combinatorial properties, see~\cite[2.1]{ChatelPilaud} and the references therein. They extend to $k$-twists as follows.

\begin{definition}
A \defn{pair of twin Cambrian $k$-twists} is a pair~$[\twist_\circ, \twist_\bullet]$ where~$\twist_\circ \in \AcyclicTwists(\signature)$ while ${\twist_\bullet \in \AcyclicTwists(-\signature)}$ for some signature~$\signature$ (where $-\signature$ is obtained by reversing all signs of~$\signature$).
\end{definition}

In particular, we call them pairs of twin $k$-twists when~$\signature = -^n$ and pairs of twin alternating $k$-twists when~${\signature = (-+)^{n/2}}$. When $k = 1$, these special cases have interesting combinatorial properties:
\begin{enumerate}[(i)]
\item pairs of twin $1$-twists are counted by Baxter numbers (see \href{https://oeis.org/A001181}{\cite[A001181]{OEIS}}), and are therefore in bijection with many other Baxter families~\cite{ChungGrahamHoggattKleiman, DulucqGuibert1, DulucqGuibert2, YaoChenChengGraham, FelsnerFusyNoyOrden, BonichonBousquetMelouFusy, LawReading, Giraudo},
\item pairs of twin alternating $1$-twists are counted by central binomial coefficients~\cite[Prop.~57]{ChatelPilaud}.
\end{enumerate}
For completeness, we have gathered in Tables~\ref{table:numbersBaxterAcyclicTwists} and~\ref{table:numbersAlternatingBaxterAcyclicTwists} the numbers of twin $(k,n)$-twists which motivates the following questions.

\begin{question}
Are there nice formulas for the number of pairs of twin $k$-twists? What about the number of pairs of twin alternating $k$-twists? Are there other combinatorial families in bijection with these pairs of twin $k$-twists?
\end{question}

\begin{table}[h]
  \centerline{$
  	\begin{array}{l|rrrrrrrrrr}
	\raisebox{-.1cm}{$k$} \backslash \, \raisebox{.1cm}{$n$}
	  & 1 & 2 & 3 & 4 & 5 & 6 & 7 & 8 & 9 & 10 \\[.1cm]
	\hline
	0 & 1 & 1 & 1 &  1 &   1 &   1 &    1 &     1 &      1 &       1 \\
	1 & . & 2 & 6 & 22 &  92 & 422 & 2074 & 10754 &  58202 &  326240 \\
	2 & . & . & . & 24 & 120 & 696 & 4512 & 31936 & 242728 & 1956072 \\
	3 & . & . & . &  . &   . & 720 & 5040 & 39600 & 341280 & 3175632 \\
	4 & . & . & . &  . &   . &   . &    . & 40320 & 362880 & 3588480 \\
	5 & . & . & . &  . &   . &   . &    . &     . &      . & 3628800
	\end{array}
  $}
  \caption{The number of pairs of twin $(k,n)$-twists for~$2k \le n \le 10$. Dots indicate that the value remains constant (equal to~$n!$) in the column.}
  \label{table:numbersBaxterAcyclicTwists}
  \vspace{-.4cm}
\end{table}

\begin{table}[h]
  \centerline{$
  	\begin{array}{l|rrrrrrrrrr}
	\raisebox{-.1cm}{$k$} \backslash \, \raisebox{.1cm}{$n$}
	  & 1 & 2 & 3 & 4 & 5 & 6 & 7 & 8 & 9 & 10 \\[.1cm]
	\hline
	0 & 1 & 1 & 1 &  1 &   1 &   1 &    1 &     1 &      1 &       1 \\
	1 & . & 2 & 6 & 20 &  70 & 252 &  924 &  3432 &  12870 &   48620 \\
	2 & . & . & . & 24 & 120 & 672 & 4188 & 27884 & 197904 & 1462384 \\
	3 & . & . & . &  . &   . & 720 & 5040 & 38880 & 328560 & 2963376 \\
	4 & . & . & . &  . &   . &   . &    . & 40320 & 362880 & 3548160 \\
	5 & . & . & . &  . &   . &   . &    . &     . &      . & 3628800
	\end{array}
  $}
  \caption{The number of pairs of twin alternating $(k,n)$-twists for~$2k \le n \le 10$. Dots indicate that the value remains constant (equal to~$n!$) in the column.}
  \label{table:numbersAlternatingBaxterAcyclicTwists}
  \vspace{-.4cm}
\end{table}

\begin{remark}[Twin $k$-twists versus $k$-twists with opposite canopy]
When~$k = 1$, two $1$-twists ${\twist_\circ \in \AcyclicTwists[1](-^n)}$ and~$\twist_\bullet \in \AcyclicTwists[1](+^n)$ are twin if and only if they have the same canopy. Up to reflecting~$\twist_\bullet$ with respect to the diagonal~$x=y$, we could equivalently consider two $(1,n)$-twists with opposite canopy, which would be closer to S.~Giraudo's presentation~\cite{Giraudo}. However, when~$k > 1$ there are Cambrian $k$-twists~$\twist_\circ \in \AcyclicTwists(-^n)$ and~$\twist_\bullet \in \AcyclicTwists(+^n)$ with the same canopy such that the union of their contact graphs~$\twist_\circ\contact \cup \twist_\bullet\contact$ is cyclic. For example, $\surjectionPermBrick[2](\down{2413})$ and~$\surjectionPermBrick[2](\up{4213})$ have the same canopy but their contact graphs has no common linear extension.
\end{remark}

\section{Schr\"oderization}
\label{sec:Schroderization}

Our last extension deals with all faces of the permutahedra, brick polytopes, and zonotopes. It follows F.~Chapoton's ideas~\cite{Chapoton} with some reformulations already used in~\cite[Part~3]{ChatelPilaud}.

\subsection{Ordered partitions, hypertwists, and partial orientations}
\label{subsec:faces}

We first recall combinatorial descriptions of the faces of~$\Perm$, $\Brick$ and~$\Zono$ and describe lattice structures on these families extending the weak order on permutations on~$\fS_n$, the increasing flip lattice on acyclic $(k,n)$-twists, and the increasing flip lattice on acyclic orientations of~$\Gkn$.

\para{Ordered partitions}
Remember first that the faces of the permutahedron~$\Perm$ correspond to the \defn{ordered partitions} of~$[n]$. Namely, the face of~$\Perm$ corresponding to the ordered partition~$\lambda$ of~$[n]$ is the convex hull of the vertices~$\b{x}(\tau)$ of~$\Perm$ given by the permutations~$\tau$ refining~$\lambda$. We denote by~$\fP_n$ the set of ordered partitions of~$[n]$.

The weak order on permutations was extended to all ordered partitions by D.~Krob, M.~Latapy, J.-C.~Novelli, H.~D.~Phan and~S.~Schwer in~\cite{KrobLatapyNovelliPhanSchwer}. See also~\cite{PalaciosRonco, DermenjianHohlwegPilaud} for an extension to all finite Coxeter systems.

\begin{definition}
The \defn{coinversion map}~${\coinv(\lambda) : \binom{[n]}{2} \to \{-1,0,1\}}$ of an ordered partition~$\lambda \in \fP_n$ is the map defined for~$i < j$ by
\[
\coinv(\lambda)(i,j) = 
\begin{cases}
-1 & \text{if } \lambda^{-1}(i) < \lambda^{-1}(j), \\
\phantom{-}0 & \text{if } \lambda^{-1}(i) = \lambda^{-1}(j), \\
\phantom{-}1 & \text{if } \lambda^{-1}(i) > \lambda^{-1}(j).
\end{cases}
\]
The \defn{weak order}~$\le$ on~$\fP_n$ is defined by~$\lambda \le \lambda'$ if~$\coinv(\lambda)(i,j) \le \coinv(\lambda')(i,j)$ for all~$i < j$.
\end{definition}

The weak order on~$\fP_3$ is represented in \fref{fig:SchroderLattices}\,(left). Note that the restriction of the weak order on~$\fP_n$ to~$\fS_n$ is the classical weak order on permutations, which is a lattice. This property was extended to the weak order on~$\fP_n$ in~\cite{KrobLatapyNovelliPhanSchwer}.

\begin{proposition}[\cite{KrobLatapyNovelliPhanSchwer}]
The weak order~$\le$ on the set of ordered partitions~$\fP_n$  is a lattice.
\end{proposition}

For~$X, Y \subset \N$, we define the notation
\[
X \ll Y \; \iff \; \max(X) < \min(Y) \; \iff \; x < y \text{ for all~$x \in X$ and~$y \in Y$.}
\]

\begin{proposition}[\cite{KrobLatapyNovelliPhanSchwer}]
The cover relations of the weak order~$<$ on~$\fP_n$ are given by
\begin{gather*}
\lambda_1 \sep \cdots \sep \lambda_i \sep \lambda_{i+1} \sep \cdots \sep \lambda_k \;\; < \;\; \lambda_1 \sep \cdots \sep \lambda_i\lambda_{i+1} \sep \cdots \sep \lambda_k \qquad\text{if } \lambda_i \ll \lambda_{i+1}, \\
\lambda_1 \sep \cdots \sep \lambda_i\lambda_{i+1} \sep \cdots \sep \lambda_k \;\; < \;\; \lambda_1 \sep \cdots \sep \lambda_i \sep \lambda_{i+1} \sep \cdots \sep \lambda_k \qquad\text{if } \lambda_{i+1} \ll \lambda_i.
\end{gather*}
\end{proposition}

\para{Hypertwists}
The main characters of this section are the following objects.

\begin{definition}
A \defn{hyperpipe} is the union of pipes of a twist, where we have additionally changed their common elbows into crossings. A \defn{$(k,n)$-hypertwist} is a collection of hyperpipes obtained by merging some subsets of pipes of a $(k,n)$-twist~$\twist$ which are connected in the contact graph~$\twist\contact$. The contact graph of a $(k,n)$-hyperpipe~$\twist[H]$ is the directed multigraph~$\twist[H]\contact$ whose nodes are labeled by the partition of~$[n]$ induced by the hyperpipes of~$\twist[H]$ and with an arc from the \SE-hyperpipe to the \WN-hyperpipe of each elbow of~$\twist[H]$. An hypertwist is \defn{acyclic} if its contact graph is. We denote by~$\AcyclicHyperTwists(n)$ the set of acyclic $(k,n)$-hypertwists.
\end{definition}

Examples of acyclic $(k,n)$-hypertwists appear in Figures~\ref{fig:1twistsDissections}, \ref{fig:SchroderLattices} and~\ref{fig:insertionSchroder}. Note that we could have equivalently defined the contact graph of the hypertwist~$\twist[H]$ as the graph obtained by contracting in~$\twist\contact$ the subgraphs induced by the merged subsets of pipes. Observe also that each hypertwist~$\twist[H]$ can be obtained by merging pipes in various twists. In fact, the face of the brick polytope~$\Brick$ corresponding to~$\twist[H]$ is the convex hull of the brick vectors~$\b{x}(\twist)$ of the $k$-twists~$\twist$ refining~$\twist[H]$.

\begin{example}[$1$-hypertwists, dissections and Schr\"oder trees]
\label{exm:1hypertwistsDissections}
The map defined in Example~\ref{exm:1twistsTriangulations} --- which sends an elbow in row~$i$ and column~$j$ of the triangular shape to the diagonal~$[i,j]$ of a convex~$(n+2)$-gon --- defines a bijective correspondence between the $(1,n)$-hypertwists and the dissections of a convex $(n+2)$-gon (\ie the crossing-free sets of diagonals). The contact graph of a $1$-hypertwist~$\twist[H]$ is sent to the dual Schr\"oder tree of~$\twist[H]\duality$, as illustrated in \fref{fig:1twistsDissections}. As a corollary, the number of $(1,n)$-hypertwists with $e$ elbows is the Schr\"oder number
\[
\frac{1}{e+1}\binom{n+2+e}{e+1}\binom{n-1}{e+1},
\]
see~\href{https://oeis.org/A033282}{\cite[A033282]{OEIS}}.
\begin{figure}[t]
	\centerline{\includegraphics[scale=1.4]{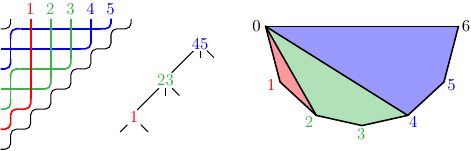}}
	\caption{The bijection between $(1,n)$-hypertwists~$\twist[H]$ (left) and dissections~$\twist[H]\duality$ of the $(n+2)$-gon (right) sends the hyperpipes of~$\twist[H]$ to the polygonal cells of~$\twist[H]\duality$, the elbows of~$\twist[H]$ to the diagonals of~$\twist[H]\duality$, and the contact graph of~$\twist[H]$ to the dual Schr\"oder tree of~$\twist[H]\duality$ (middle).}
	\label{fig:1twistsDissections}
\end{figure}
\end{example}

\begin{remark}[$k$-hypertwists and $(k+1)$-crossing free sets of diagonals]
Similarly to Example~\ref{exm:1hypertwistsDissections}, the map defined in Theorem~\ref{theo:ktwistsktriangulations} --- which sends an elbow in row~$i$ and column~$j$ of the ${(n+2k) \times (n+2k)}$ triangular shape to the diagonal~$[i,j]$ of a convex $(n+2k)$-gon --- defines a bijective correspondence between the $(k,n)$-hypertwists and the subsets of $k$-relevant diagonals of the $(n+2k)$-gon with no $k+1$ pairwise crossing diagonals.
\end{remark}

\begin{definition}
Consider a $k$-hypertwist~$\twist[H]$ and an arc~$u \to v$ of its contact graph~$\twist[H]\contact$ (thus, $u$ and~$v$ are subsets of~$[n]$). We denote by~$\twist[H]/_{u \to v}$ the $k$-hypertwist obtained by merging the hyperpipes~$u$ and~$v$. We define the \defn{Schr\"oder poset}~$\le$ on acyclic $(k,n)$-hypertwist as the transitive closure of the relations~$\twist[H] < \twist[H]/_{u \to v}$ (resp.~$\twist[H]/_{u \to v} < \twist[H]$) for any acyclic $(k,n)$-hypertwist and any arc~$u \to v \in \twist[H]\contact$ such that~$u \ll v$ (resp.~$u \gg v$).
\end{definition}

The Schr\"oder poset on $(1,3)$-hypertwists is represented in \fref{fig:SchroderLattices}. Note that the $k$-hyper\-twists~$\twist[H]$ and~$\twist[H]/_{u \to v}$ are not comparable when~$u \not\ll v$ and~$u \not\gg v$. Observe also that the restriction of the Schr\"oder poset to the acyclic $(k,n)$-twists is the increasing flip lattice on acyclic $(k,n)$-twists: indeed, $\twist < \twist'$ is a cover relation in the increasing flip lattice when~$\twist'$ is obtained by flipping an elbow~$i \to j$ in~$\twist$, \ie when~$\twist < \twist /_{i \to j} = \twist' /_{j \to i} < \twist'$ are cover relations in the Schr\"oder poset.

\begin{figure}[t]
	\centerline{\includegraphics[scale=.9]{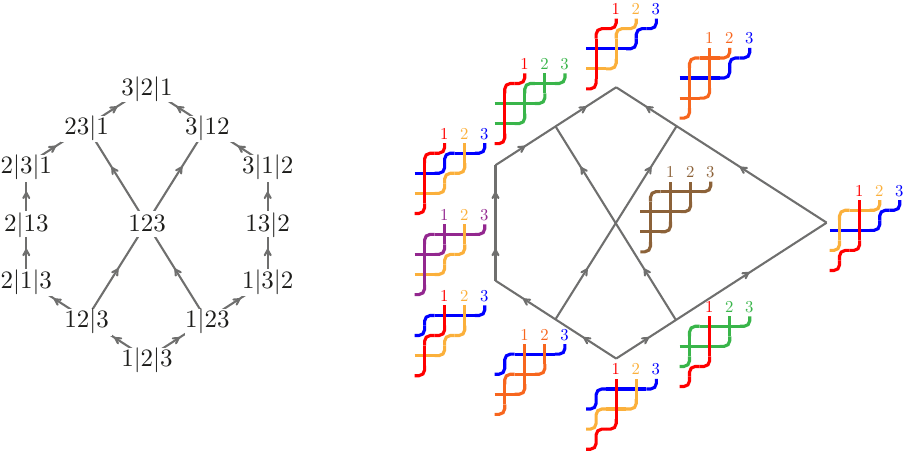}}
	\caption{The weak order on ordered partitions of~$[3]$ (left) and the Schr\"oder lattice on $(1,3)$-hypertwists (right).}
	\label{fig:SchroderLattices}
\end{figure}

\medskip
We now define the insertion map~$\surjectionPermBrick : \fP_n \to \AcyclicHyperTwists(n)$ which sends ordered partitions of~$[n]$ to acyclic $(k,n)$-hypertwists. 
Given an ordered partition~$\lambda \eqdef \lambda_1 \sep \cdots \sep \lambda_p \in \fP_n$, we construct a $(k,n)$-hypertwist~$\surjectionPermBrick(\lambda)$ obtained from the $(k,0)$-hypertwist by inserting successively~$\lambda_p, \dots, \lambda_1$. When we insert a new part~$\lambda_i$, we first insert its elements in an arbitrary order and then merge the subsets of pipes of~$\lambda_i$ which induce connected subgraphs of the resulting contact graph. Note that it does not always merge all pipes of~$\lambda_i$ together ($\lambda_i$ can have more than one connected component after the insertion step). In particular, the number of hyperpipes of~$\surjectionPermBrick(\lambda)$ is at least the number of parts of~$\lambda$, but could be strictly larger. This insertion algorithm is illustrated in \fref{fig:insertionSchroder}.

\begin{figure}[h]
	\centerline{\includegraphics[scale=1.08]{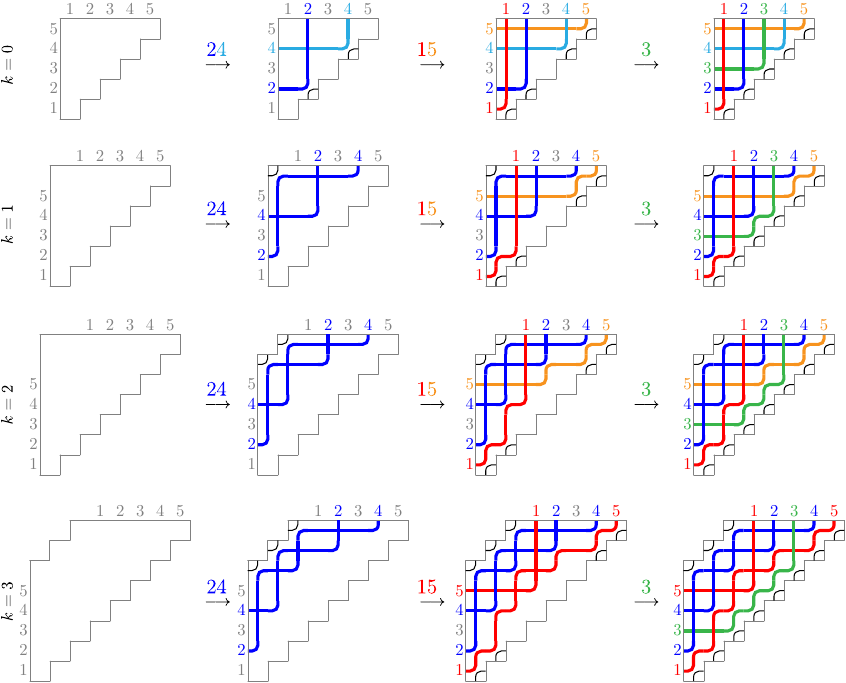}}
	\caption{Insertion of the ordered partition~$3 \sep 15 \sep 24$ in a $k$-hypertwist for ${k = 0, 1, 2, 3}$.}
	\label{fig:insertionSchroder}
\end{figure}

\begin{proposition}
The map~$\surjectionPermBrick$ is a surjection from the ordered partitions of~$[n]$ to the acyclic $(k,n)$-hypertwists.
\end{proposition}

\begin{proof}
Consider an ordered partition~$\lambda$ and a linear extension~$\tau$ of~$\lambda$. By definition of the map~$\surjectionPermBrick$ on ordered partitions, $\surjectionPermBrick(\lambda)$ is obtained from~$\surjectionPermBrick(\tau)$ by merging subsets of pipes which are connected in~$\surjectionPermBrick(\tau)\contact$. Since~$\surjectionPermBrick(\tau)$ is acyclic by Proposition~\ref{prop:ktwistInsertion}, it implies that~$\surjectionPermBrick(\lambda)$ is also acyclic. Conversely, since the new inserted nodes are sources of the contact graph at each step, any acyclic $(k,n)$-hypertwist~$\twist[H]$ is the image of any linear extension of its contact graph~$\twist[H]\contact$.
\end{proof}

In fact, the fiber of a $(k,n)$-hypertwist~$\twist[H]$ by~$\surjectionPermBrick$ is the set of all ordered partitions obtained from a linear extension of the contact graph~$\twist[H]\contact$ by merging parts which label incomparable vertices of~$\twist[H]\contact$. We now extend the $k$-twist congruence on all ordered partitions.

\begin{definition}
The \defn{$k$-hypertwist congruence} is the equivalence relation~$\equiv^k$ on~$\fP_n$ defined as the transitive closure of the rewriting rules
\[
U \sep \b{a} \sep \b{c} \sep V \equiv^k U \sep \b{ac} \sep V \equiv^k U \sep \b{c} \sep \b{a} \sep V
\]
where~$\b{a}, \b{c}$ are parts while~$U,V$ are sequences of parts of~$[n]$, and there exist~$b_1, \dots, b_k$ in~$\bigcup V$ such that~$\max(\b{a}) < b_i < \min(\b{c})$ for all~$i \in [\ell]$.
\end{definition}

\begin{proposition}
\label{prop:khypertwistCongruence}
For any~$\lambda, \lambda' \in \fP_n$, we have~$\lambda \equiv^k \lambda' \iff \surjectionPermBrick(\lambda) = \surjectionPermBrick(\lambda')$. In other words, the fibers of~$\surjectionPermBrick$ are precisely the congruence classes of~$\equiv^k$.
\end{proposition}

\begin{proof}
It is sufficient to prove that~$\lambda \equiv^k \lambda' \iff \surjectionPermBrick(\lambda) = \surjectionPermBrick(\lambda')$ for any two ordered partitions~$\lambda = U \sep \b{a} \sep \b{c} \sep V$ and~$\lambda' = U \sep \b{ac} \sep V$ which differ by merging two consecutive parts~$\b{a} \ll \b{c}$ (the other case is similar). As in the proof of Proposition~\ref{prop:ktwistCongruence}, the hyperpipes labeled by~$\b{a}$ and~$\b{c}$ in the insertion will share an elbow if and only if there is no~$b_1, \dots, b_k$ in~$\bigcup V$ such that~$\max(\b{a}) < b_i < \min(\b{c})$ for all~$i \in [\ell]$. The result follows immediately.
\end{proof}

\begin{theorem}
The $k$-hypertwist congruence~$\equiv^k$ is a lattice congruence of the weak order on~$\fP_n$.
\end{theorem}

\begin{proof}
We only sketch the proof as it is similar to that of Theorem~\ref{theo:ktwistLatticeCongruence} and of~\cite[Prop.~105]{ChatelPilaud}. There are two steps:
\begin{itemize}
\item The congruence classes of~$\equiv^k$ are intervals of the weak order on~$\fP_n$. To see it, consider a $k$-hypertwist~$\twist[H]$. Since incomparable nodes in~$\twist[H]\contact$ have comparable label sets by Proposition~\ref{prop:khypertwistCongruence}, there is a minimal~$\mu$ and a maximal~$\omega$ linear extension of~$\twist[H]\contact$. The fiber of~$\twist[H]$ under~$\surjectionPermBrick$ is then the weak order interval~$[\mu,\omega]$ of~$\fP_n$.
\item The up and down projection maps of the congruence classes are order preserving. The proof is very similar to that of Theorem~\ref{theo:ktwistLatticeCongruence}, except that there are height instead of four cases. These cases are described in details in the proof of~\cite[Lemma~107]{ChatelPilaud} for the case~$k = 1$ in the Cambrian setting. It then suffices to consider only the signature~$-^n$ and to replace the Schr\"oder Cambrian witness~$b$ by $k$-twist congruence witnesses~$b_1, \dots, b_k$. \qedhere
\end{itemize}
\end{proof}

\begin{proposition}
The map~$\surjectionPermBrick : \fP_n \to \AcyclicHyperTwists(n)$ defines an isomorphism from the lattice quotient of the weak order on~$\fP_n$ by the $k$-hypertwist congruence~$\equiv^k$ to the Schr\"oder poset on acyclic $(k,n)$-hypertwists. In particular, the Schr\"oder poset defines in fact a lattice.
\end{proposition}

\begin{proof}
Consider a cover relation~$\lambda < \lambda'$ in the weak order on~$\fP_n$. Assume that~$\lambda'$ is obtained by merging two parts~$\lambda_i \ll \lambda_{i+1}$ of~$\lambda$ (the other cover relation is treated similarly). Let~$u$ and~$v$ be the nodes of~$\surjectionPermBrick(\lambda)$ containing~$\max(\lambda_i)$ and~$\min(\lambda_{i+1})$ respectively. If~$u$ and~$v$ are incomparable in~$\surjectionPermBrick(\lambda)\contact$, then~$\surjectionPermBrick(\lambda) = \surjectionPermBrick(\lambda')$. Otherwise, there is an arc from~$u$ to~$v$ in~$\surjectionPermBrick(\lambda)\contact$, and $\surjectionPermBrick(\lambda)' = \surjectionPermBrick(\lambda)/_{u \to v}$.
\end{proof}

\enlargethispage{.2cm}
\para{Acyclic partial orientations}
Finally, we recall that the faces of a graphical zonotope~$\Zono[][\graphG]$ correspond to \defn{acyclic partial orientations} of~$\graphG$, \ie to partial orientations of~$G$ where
\begin{itemize}
\item the unoriented edges form connected induced subgraphs of~$\graphG$,
\item the oriented graph obtained by contracting the unoriented edges is acyclic.
\end{itemize}
The face of~$\Zono[][\graphG]$ corresponding to a partial orientation~$\orientation$ is the convex hull of the vertices~$\b{x}(\orientation)$ of~$\Zono[][\graphG]$ given by the acyclic orientations of~$\graphG$ refining~$\orientation$. We denote by~$\AcyclicPartialOrientations(n)$ the set of all acyclic partial orientations of~$\Gkn$. We can also think of a partial orientation of~$\Gkn$ as a map~$\orientation : \set{(i,j) \in [n]^2}{i < j \le i+k} \to \{-1, 0, 1\}$ where~$\orientation(i,j) = -1$ if the edge~$\{i,j\}$ is oriented from~$j$ to~$i$, $\orientation(i,j) = 0$ if the edge remains unoriented, and~$\orientation(i,j) = 1$ if the edge is oriented from~$i$ to~$j$. We then define a lattice structure on~$\AcyclicPartialOrientations(n)$ by~$\orientation \le \orientation'$ if $\orientation(i,j) \le \orientation'(i,j)$ for all~$i < j \le i+k$.

We define the \defn{$k$-recoils scheme} of an ordered partition~$\lambda \in \fP_n$ as the acyclic partial orientation~$\surjectionPermZono(\lambda)$ defined by the restriction of the coinversions of~$\lambda$ to~$\set{(i,j) \in [n]^2}{i < j \le i+k}$, and the \defn{$k$-recoils scheme} of a $(k,n)$-hypertwist~$\twist[H]$ as the partial orientation~$\surjectionBrickZono(\twist[H])$ with an edge~$i \to j$ if there is an oriented path in the contact graph~$\twist[H]\contact$ from the node containing~$i$ to the node containing~$j$. The next statement is an immediate consequence of these definitions.

\begin{proposition}
The maps~$\surjectionPermBrick$, $\surjectionBrickZono$, and~$\surjectionPermZono$ define a commutative diagram of lattice homomorphisms:
\[
\begin{tikzpicture}
  \matrix (m) [matrix of math nodes,row sep=1.2em,column sep=5em,minimum width=2em]
  {
     \fP_n  	&									& \AcyclicPartialOrientations(n)	\\
				& \AcyclicHyperTwists(\signature) 	&									\\
  };
  \path[->>]
    (m-1-1) edge node [above] {$\surjectionPermZono$} (m-1-3)
                 edge node [below] {$\surjectionPermBrick\quad$} (m-2-2.west)
    (m-2-2.east) edge node [below] {$\surjectionBrickZono$} (m-1-3);
\end{tikzpicture}
\]
\end{proposition}

\subsection{Algebra of acyclic hypertwists}
\label{subsec:algebraFaces}

We now define Hopf algebras whose bases are indexed by ordered partitions, acyclic $k$-hypertwists, and acyclic partial orientations of~$\Gkn$. To begin with, we reformulate as in~\cite{ChatelPilaud} the Hopf algebra of F.~Chapoton~\cite{Chapoton} indexed by the faces of the permutahedra in terms of ordered partitions instead of surjections.

We first define two restrictions on ordered partitions. Consider an ordered partition~$\mu$ of~$[n]$ into~$p$ parts. For~${I \subseteq [p]}$, we let~$n_I \eqdef |\set{j \in [n]}{\exists \, i \in I, \; j \in \mu_i}|$ and we denote by~$\mu_{|I}$ the ordered partition of~$[n_I]$ into $|I|$ parts obtained from~$\mu$ by deletion of the parts in~$[p] \ssm I$ and standardization. Similarly, for~$J \subseteq [n]$, we let~$p_J \eqdef |\set{i \in [k]}{\exists \, j \in J, \; j \in \mu_i}|$ and we denote by~$\mu^{|J}$ the ordered partition of~$[|J|]$ into~$p_J$ parts obtained from~$\mu$ by deletion of the entries in~$[n] \ssm J$ and standardization. For example, for the ordered partition~$\mu = 16 \sep 27 \sep 4 \sep 35$ we have~$\mu_{|\{2,3\}} = 13 \sep 2$ and~$\mu^{|\{1,3,5\}} = 1 \sep 23$.

The \defn{shifted shuffle product}~$\lambda \shiftedShuffle \lambda'$ and the \defn{convolution product}~$\lambda \convolution \lambda'$ of two ordered partitions~${\lambda \in \fP_n}$ and~$\lambda' \in \fP_{n'}$ are defined by:
\begin{align*}
\lambda \shiftedShuffle \lambda' & \eqdef \bigset{\mu \in \fP_{n+n'}}{\mu^{|\{1, \dots, n\}} = \lambda \text{ and } \mu^{|\{n+1, \dots, n+n'\}} = \lambda'}, \\
\text{and}\qquad \lambda \convolution \lambda' & \eqdef \bigset{\mu \in \fP_{n+n'}}{\mu_{|\{1, \dots, k\}} = \lambda \text{ and } \mu_{|\{k+1, \dots, k+k'\}} = \lambda'}.
\end{align*}
Note that if~$\lambda$ has $p$ parts and~$\lambda'$ has~$p'$ parts, then the partitions of~$\lambda \shiftedShuffle \lambda'$ have at most~$p+p'$ parts, while the partitions of~$\lambda \convolution \lambda'$ have precisely~$p+p'$ parts.

For example,
\begin{align*}
{\red 1} \sep {\red 2} \shiftedShuffle {\blue 2} \sep {\blue 13} = \{ &
	{\red 1} \sep {\red 2} \sep {\blue 4} \sep {\blue 35}, \;
	{\red 1} \sep {\red 2}{\blue 4} \sep {\blue 35}, \;
	{\red 1} \sep {\blue 4} \sep {\red 2} \sep {\blue 35}, \;
	{\red 1} \sep {\blue 4} \sep {\red 2}{\blue 35}, \;
	{\red 1} \sep {\blue 4} \sep {\blue 35} \sep {\red 2}, \;
	{\red 1}{\blue 4} \sep {\red 2} \sep {\blue 35}, \;
	{\red 1}{\blue 4} \sep {\red 2}{\blue 35}, \\[-.1cm]
&	{\red 1}{\blue 4} \sep {\blue 35} \sep {\red 2}, \;
	{\blue 4} \sep {\red 1} \sep {\red 2} \sep {\blue 35}, \;
	{\blue 4} \sep {\red 1} \sep {\red 2}{\blue 35}, \;
	{\blue 4} \sep {\red 1} \sep {\blue 35} \sep {\red 2}, \;
	{\blue 4} \sep {\red 1}{\blue 35} \sep {\red 2}, \;
	{\blue 4} \sep {\blue 35} \sep {\red 1} \sep {\red 2} \}, \\[.2cm]
{\red 1} \sep {\red 2} \convolution {\blue 2} \sep {\blue 13} = \{ &
	{\red 1} \sep {\red 2} \sep {\blue 4} \sep {\blue 35}, \;
	{\red 1} \sep {\red 3} \sep {\blue 4} \sep {\blue 25}, \;
	{\red 1} \sep {\red 4} \sep {\blue 3} \sep {\blue 25}, \;
	{\red 1} \sep {\red 5} \sep {\blue 3} \sep {\blue 24}, \;
	{\red 2} \sep {\red 3} \sep {\blue 4} \sep {\blue 15}, \\[-.1cm]
&	{\red 2} \sep {\red 4} \sep {\blue 3} \sep {\blue 15}, \;
	{\red 2} \sep {\red 5} \sep {\blue 3} \sep {\blue 14}, \;
	{\red 3} \sep {\red 4} \sep {\blue 2} \sep {\blue 15}, \;
	{\red 3} \sep {\red 5} \sep {\blue 2} \sep {\blue 14}, \;
	{\red 4} \sep {\red 5} \sep {\blue 2} \sep {\blue 13} \}.
\end{align*}

We denote by~$\OrdPart$ the Hopf algebra with basis~$(\F_\lambda)_{\lambda \in \fP}$ and whose product and coproduct are defined by
\[
\F_\lambda \product \F_{\lambda'} = \sum_{\mu \in \lambda \shiftedShuffle \lambda'} \F_\mu
\qquad\text{and}\qquad
\coproduct \F_\mu = \sum_{\mu \in \lambda \convolution \lambda'} \F_\lambda \otimes \F_{\lambda'}.
\]
This indeed defines a Hopf algebra according to the work of F.~Chapoton~\cite{Chapoton}.

We then construct two Hopf subalgebras of~$\OrdPart$ indexed by acyclic $(k,n)$-hypertwists and acyclic partial orientations of~$\Gkn$ for~$n \in \N$. We denote by~$\HyperTwist$ the vector subspace of~$\OrdPart$ generated by the elements
\[
\PTwist_{\twist[H]} \eqdef \sum_{\substack{\lambda \in \fP \\ \surjectionPermBrick(\lambda) = \twist[H]}} \F_\lambda,
\]
for all acyclic~$k$-hypertwists~$\twist[H]$.
For example, for the $(k,5)$-hypertwists of \fref{fig:insertionSchroder}, we have
\[
\PTwist_{\includegraphics[scale=.5]{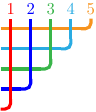}} = \sum_{\lambda \in \fP_5} \F_\lambda
\qquad
\PTwist\raisebox{.05cm}{$_{\hspace{-.1cm}\includegraphics[scale=.5]{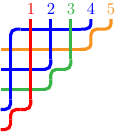}}$} = \!\!\!\! \begin{array}[t]{l} \phantom{+} \; \F_{1 \sep 3 \sep 5 \sep 24} + \F_{1 \sep 5 \sep 3 \sep 24} + \F_{3 \sep 1 \sep 5 \sep 24} + \F_{5 \sep 1 \sep 3 \sep 24} + \F_{3 \sep 5 \sep 1 \sep 24} + \F_{5 \sep 3\sep 1 \sep 24} \\ + \; \F_{1 \sep 35 \sep 24 \phantom{\sep}} + \F_{3 \sep 15 \sep 24 \phantom{\sep}} + \F_{5 \sep 13 \sep 24 \phantom{\sep}} + \F_{13 \sep 5 \sep 24 \phantom{\sep}} + \F_{15 \sep 3 \sep 24 \phantom{\sep}} + \F_{35 \sep 1 \sep 24 \phantom{\sep}} \\ + \; \F_{135 \sep 24 \phantom{\sep\sep}} \end{array}
\]
\[
\PTwist\raisebox{.1cm}{$_{\hspace{-.2cm}\includegraphics[scale=.5]{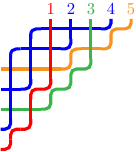}}$} = \F_{3 \sep 1 \sep 5 \sep 24} + \F_{3 \sep 15 \sep 24} + \F_{3 \sep 5 \sep 1 \sep 24}
\qquad\qquad\text{and}\qquad\qquad
\PTwist\raisebox{.15cm}{$_{\hspace{-.3cm}\includegraphics[scale=.5]{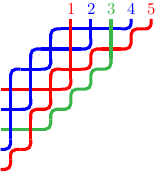}}$} = \F_{3 \sep 15 \sep 24}.
\hspace{2.5cm}
\]
Moreover, we denote by $\HyperRec$ the vector subspace of~$\OrdPart$ generated by the elements
\[
\XRec_{\orientation} \eqdef \sum_{\substack{\lambda \in \fP \\ \surjectionPermZono(\lambda) = \orientation}} \F_\lambda = \sum_{\substack{\twist[H] \in \AcyclicHyperTwists \\ \surjectionBrickZono(\twist[H]) = \orientation}} \PTwist_{\twist[H]},
\]
for all acyclic partial orientations~$\orientation$ of the graph~$\Gkn$ (for all~$n \in \N$). The proof of the following statement is again a straightforward verification.

\begin{theorem}
$\HyperTwist$ and~$\HyperRec$ are both Hopf subalgebras of~$\OrdPart$, and we have the inclusions~$\HyperRec \subset \HyperTwist \subset \OrdPart$.
\end{theorem}

As in Section~\ref{subsec:subalgebra}, one can describe combinatorially the product and coproduct in~$\HyperTwist_\pm$ directly in terms of $k$-hypertwists. In particular, the product~$\PTwist_{\twist[H]} \product \PTwist_{\twist[H]'}$ in~$\HyperTwist$ of two $k$-hypertwists~$\twist[H] \in \AcyclicHyperTwists(n)$ and~$\twist[H]' \in \AcyclicHyperTwists(n')$ is given by
\[
\PTwist_{\twist[H]} \product \PTwist_{\twist[H]'} = \sum_{\twist[K]} \PTwist_{\twist[K]},
\]
where the sum runs over the interval of the Schr\"oder lattice between the $(k, n+n')$-hypertwists~$\underprod{\twist[H]}{\twist[H]'}$ and~$\overprod{\twist[H]}{\twist[H]'}$, respectively obtained by inserting~$\twist[H]$ in the first rows and columns of~$\twist[H]'$, and by inserting~$\twist[H]'$ in the last rows and columns of~$\twist[H]$. We invite the reader to work out the combinatorial descriptions of the coproduct in~$\HyperTwist$ and of the product and coproduct in the dual Hopf algebra~$\HyperTwist[k*]$, which are more subtle than the product (when performing cuts, one should be careful which hyperpipes can be merge).

\section*{Acknoledgements}

I am grateful to N.~Bergeron and C.~Ceballos for sharing their preliminary results on Hopf algebras on pipe dreams, to N.~Reading for pointing out relevant references and sharing his expertise on lattice congruences, and to F.~Hivert for very interesting discussions on monoids, Hopf algebras, and dendriform structures. I also thank three anonymous referees for relevant suggestions on this paper.

\bibliographystyle{alpha}
\bibliography{brickHopfAlgebra}
\label{sec:biblio}

\end{document}